\documentclass[11pt]{amsart}
\usepackage[english]{babel}
\usepackage{amssymb}
\usepackage{mathrsfs}
\usepackage{enumerate}
\usepackage{amsrefs}
\usepackage[all]{xy}
\SelectTips{cm}{}
\usepackage{pdfsync} 
\usepackage{color}
\usepackage{mathptmx}
\usepackage{eucal}
\usepackage{hyperref} 
\index{}
\newcommand{\bbN}{{\mathbf N}}
\newcommand{\bbQ}{{\mathbf Q}}
\newcommand{\bbR}{{\mathbf R}}
\newcommand{\bbZ}{{\mathbf Z}}
\newcommand{\bbC}{{\mathbf C}}
\newcommand{\bbT}{{\mathbf T}}
\newcommand{\bbF}{{\mathbf F}}
\newcommand{\HH}{{\mathbf H}}

\newcommand{\scrC}{\mathscr{C}}
\newcommand{\scrF}{\mathscr{F}}
\newcommand{\scrR}{\mathscr{R}}
\newcommand{\scrH}{\mathscr{H}}
\newcommand{\scrS}{\mathscr{S}}

\newcommand{\SL}{\operatorname{SL}}

\newcommand{\acts}{\curvearrowright}

\newcommand{\Rel}{\scrR}
\newcommand{\Srel}{\scrS}

\newcommand{\Isom}{\operatorname{Isom}}
\newcommand{\hypsp}{{\mathbb H}}
\newcommand{\Ad}{\operatorname{Ad}}

\newcommand{\Aut}{\operatorname{Aut}}
\newcommand{\Inn}{\operatorname{Inn}}
\newcommand{\Out}{\operatorname{Out}}
\newcommand{\II}{\operatorname{II}}

\newcommand{\cost}{\operatorname{cost}}
\newcommand{\gcost}{\operatorname{\scrC}}
\newcommand{\gcostinf}{\operatorname{\scrC}_{*}}
\newcommand{\gcostsup}{\operatorname{\scrC}^{*}}
\newcommand{\Grps}{\operatorname{\mathscr{C}}}
\newcommand{\Gdsc}{\operatorname{\mathscr{G}_{\rm dsc}}}
\newcommand{\Gcpt}{\operatorname{\mathscr{G}_{\rm cpt}}}
\newcommand{\Creg}{\operatorname{\mathcal{C}_{\rm reg}}}
\newcommand{\Cmix}{\operatorname{\mathcal{C}}}
\newcommand{\Dyn}{\operatorname{\mathcal{D}}}
\newcommand{\Dynam}{\operatorname{\mathcal{D}_{\rm ea}}}
\newcommand{\Ufin}{\operatorname{\mathscr{U}_{\rm fin}}}
\newcommand{\Ginv}{\operatorname{\mathscr{G}_{\rm binv}}}
\newcommand{\Galg}{\operatorname{\mathscr{G}_{\rm alg}}}
\newcommand{\rank}{\operatorname{rank}}
\newcommand{\srank}{\operatorname{RG}}

\newcommand{\me}{\stackrel{\text{\rm\tiny ME}}{\sim}}
\newcommand{\ME}[1]{{\rm ME}({#1})}
\newcommand{\qi}{\stackrel{\text{\rm\tiny QI}}{\sim}}
\newcommand{\QI}[1]{{\rm QI}{#1}}
\newcommand{\ore}{\stackrel{\text{\rm\tiny OE}}{\sim}}

\newcommand{\soe}{\stackrel{\text{\rm\tiny SOE}}{\sim}}

\newcommand{\overto}[1]{{\buildrel{#1}\over\longrightarrow}}
\newcommand{\overfrom}[1]{{\buildrel{#1}\over\longleftarrow}}
\newcommand{\setdef}[2]{ \left\{ {#1}\ :\ {#2} \right\} }

\newcommand{\embf}[1]{{\bf {#1}}}

\newtheorem{theorem}{Theorem}[section]
\newtheorem{thm}[theorem]{Theorem}
\newtheorem{lemma}[theorem]{Lemma}

\newtheorem{corollary}[theorem]{Corollary}
\newtheorem{cor}[theorem]{Corollary}

\newtheorem{prop}[theorem]{Proposition}
\newtheorem{question}[theorem]{Question}
\newtheorem*{ques*}{Question}

\theoremstyle{definition}
\newtheorem{definition}[theorem]{Definition}
\newtheorem{defn}[theorem]{Definition}

\newtheorem{example}[theorem]{Example}
\newtheorem{remark}[theorem]{Remark}

\numberwithin{equation}{section}

\begin{document}

\title{A survey of Measured Group Theory}

\author{Alex Furman}\thanks{Supported by NSF Grant DMS 0604611, and BSF Grant 2004345}
\address{University of Illinois at Chicago}

\begin{center}
	To Robert Zimmer on the occasion of his 60th birthday. 
\end{center}

\begin{abstract}
Measured Group Theory is an area of research that studies infinite groups 
using measure-theoretic tools, and studies the restrictions that group structure 
imposes on ergodic-theoretic properties of their actions.
The paper is a survey of recent developments focused on the notion of Measure Equivalence
between groups, and Orbit Equivalence between group actions.
\end{abstract}

\date{\today}.

\maketitle
\tableofcontents


\section{Introduction} 
\label{sec:introduction}

This survey concerns an area of mathematics which studies infinite countable groups 
using measure-theoretic tools, and studies Ergodic Theory of group actions, 
emphasizing the impact of group structure on the actions.
\emph{Measured Group Theory} is a particularly fitting title as it suggests
an analogy with \emph{Geometric Group Theory}.  
The origins of Measured Group Theory go back to the seminal paper of Robert Zimmer \cite{Zimmer:cocyclesuper:80},
which established a deep connection between questions on Orbit Equivalence in Ergodic Theory to
Margulis' celebrated superrigidity theorem for lattices in semi-simple groups.  
The notion of amenable actions, introduced by Zimmer in an earlier work \cite{Zimmer:Amenable:78},
became an indispensable tool in the field.
Zimmer continued to study orbit structures of actions of large groups
in \cite{Zimmer:OE:1981, Zimmer:1981:cohom, MR672181, Zimmer:1982:Bull, Zimmer:1983:products, 
Zimmer:book:84, Feldman+Sutherland+Zimmer:1989, Cowling+Zimmer:89sp, Zimmer:1991:trans} and \cite{Stuck+Zimmer:1994}. 
The monograph \cite{Zimmer:book:84} had a particularly big impact on both ergodic theorists
and people studying big groups, as well as researchers in other areas, such as Operator Algebras
and Descriptive Set Theory\footnote{
 Zimmer's cocycle superrigidity proved in \cite{Zimmer:cocyclesuper:80} plays a central role in another
area of research, vigorously pursued by Zimmer and other, concerning actions of
large groups on manifolds. 
David Fisher surveys this direction in \cite{Fisher:survey} in this volume.}.

In the recent years several new layers of results have been added to what
we called Measured Group Theory, and this paper aims to give an overview of the current state of the subject.
Such a goal is unattainable -- any survey is doomed to be partial, biased, and outdated before it 
appears.
Nevertheless, we shall try our best, hoping to encourage further interest in this topic.
The reader is also referred to Gaboriau's paper \cite{Gaboriau:2005exmps}, which contains
a very nice overview of some of the results discussed here,
and to Shalom's survey \cite{Shalom:2005ECM} which is even closer to the present paper (hence the similarity of the titles).
The monographs by Kechris and Miller 
\cite{Kechris+Miller:2004book} and the forthcoming one \cite{Kechris:new} by Kechris
include topics in Descriptive Set Theory related to Measured Group Theory.
Readers interested in connections to von Neumann algebras are referred to Vaes' \cite{Vaes:2007sb}, 
Popa's \cite{Popa:2007ICM}, and references therein.

The scope of this paper is restricted to interaction of infinite Groups with Ergodic theory, 
leaving out the connections to the theory of von Neumann algebras
and Descriptive Set Theory. 
When possible, we try to indicate proofs or ideas of proofs for the stated results. 
In particular, we chose to include a proof of one
cocycle superrigidity theorem \ref{T:GammaonK},
which enables a self-contained presentation of a number
of important results: a very rigid equivalence relation (Theorem~\ref{T:OE-rigidity-combined})
with trivial fundamental group and outer automorphism group (Theorem \ref{T:Out1}),
an equivalence relation which cannot be generated by an essentially free action of any group (\S \ref{ssub:FM_question}).

\medskip

\subsection*{Disclaimer}
As usual, the quoted results are often presented not in the full possible generality,
so the reader should consult the original papers for full details.
The responsibility for inaccuracies, misquotes and other flaws lies solely with the author
of these notes.

\subsection*{Acknowledgements}
I would like to express my deep appreciation to Bob Zimmer for his singular contribution
to the subject. 
I would also like to thank Miklos Abert, Aur\'elien Alvarez, Uri Bader, Damien Gaboriau, 
Alexander Kechris, Sorin Popa, Yehuda Shalom, and the referee for the corrections and 
comments on the earlier drafts of the paper.

\subsection*{Organization of the paper}
The paper is organized as follows: the next section is devoted 
to a general introduction that emphasizes the relations between Measure Equivalence, 
Quasi-Isometry and Orbit Equivalence in Ergodic Theory.
One may choose to skip most of this, but read Definition~\ref{D:ME} and the following remarks.
Section~\ref{sec:measure_equivalence} concerns groups considered up to Measure Equivalence.
Section~\ref{sec:equivalence_relations} focuses on the notion of equivalence relations with orbit relations
as a prime (but not only) example.
In both of these sections we consider separately the \emph{invariants} of the studied objects
(groups and relations) and \emph{rigidity} results, which pertain to questions of classification.
Section~\ref{sec:techniques} describes the main techniques used in these theories (mostly for rigidity):
a discussion of superrigidity phenomena
and some of the ad hoc tools used in the subject; generalities on cocycles appear 
in the appendix \ref{sec:cocycles}.

\newpage
\section{Preliminary discussion and remarks} 
\label{sec:general_discussion}

This section contains an introduction to Measure Equivalence and related topics
and contains a discussion of this framework.
Readers familiar with the subject (especially definition~\ref{D:ME} and the following remarks) 
may skip to the next section in the first reading.

\bigskip

There are two natural entry points to Measured Group Theory, corresponding
to the ergodic-theoretic and group-theoretic perspectives. 
Let us start from the latter.

\bigskip

\subsection{Lattices, and other countable groups} 
\label{sub:groups}\hfill\\

When should two infinite discrete groups be viewed as closely related?
Isomorphism of abstract groups is an obvious, maybe trivial, answer.
The next degree of closeness would be \embf{commensurability}: two groups 
are commensurable if they contain isomorphic subgroups of finite index. 
This relation might be relaxed a bit further, by allowing to pass to 
a quotient modulo finite normal subgroups. 
The algebraic notion of being commensurable, modulo finite kernels, 
can be vastly generalized in two directions: Measure Equivalence (Measured Group Theory)
and Quasi-Isometry (Geometric Group Theory). 

\medskip

The key notion discussed in this paper is that of \embf{Measure Equivalence} of groups, 
introduced by Gromov in \cite[0.5.E]{Gromov:1993asinv}.

\begin{defn}\label{D:ME}
Two infinite discrete countable groups $\Gamma$, $\Lambda$ are \embf{Measure Equivalent} 
(abbreviated as ME, and denoted $\Gamma\me \Lambda$) if there exists an infinite measure space $(\Omega,m)$ 
with a measurable, measure preserving 
action of $\Gamma\times \Lambda$, so that both actions $\Gamma\acts (\Omega,m)$ and 
$\Lambda\acts (\Omega,m)$ admit finite measure fundamental domains $Y,X\subset\Omega$: 
\[
	\Omega=\bigsqcup_{\gamma\in\Gamma} \gamma Y=\bigsqcup_{\lambda\in\Lambda} \lambda X.
\]
The space $(\Omega,m)$ is called a $(\Gamma,\Lambda)$-\embf{coupling} or ME-coupling.
The \embf{index} of $\Gamma$ to $\Lambda$ in $\Omega$ is the ratio of the measures 
of the fundamental domains 
\[
	[\Gamma:\Lambda]_\Omega=\frac{m(X)}{m(Y)}\quad 
		\left(=\frac{{\rm meas}(\Omega/\Lambda)}{{\rm meas}(\Omega/\Gamma)}\right).
\]
\end{defn}

\medskip

We shall motivate this definition after making a few immediate comments. 

\medskip

\embf{(a)} The index $[\Gamma:\Lambda]_\Omega$ is well defined -- it does not depend on the choice of the 
fundamental domains $X$, $Y$ for $\Omega/\Lambda$, $\Omega/\Gamma$ respectively,
because their measures are determined by the group actions on $(\Omega,m)$.
However, a given pair $(\Gamma,\Lambda)$ might have ME-couplings with different indices
(the set $\{[\Gamma:\Lambda]_\Omega\}$ is a coset of a subgroup of $\bbR^*_+$ corresponding
to possible indices $[\Gamma:\Gamma]_\Omega$ of self $\Gamma$-couplings. 
Here it makes sense to focus on \emph{ergodic} couplings only). 

\medskip

\embf{(b)} Any ME-coupling can be decomposed into an integral over a probability space
of \emph{ergodic} ME-couplings, i.e., ones for which the $\Gamma\times\Lambda$-action
is ergodic.

\medskip

\embf{(c)} Measure Equivalence is indeed an \emph{equivalence relation} between groups:
for any countable $\Gamma$ the action of $\Gamma\times\Gamma$ on $\Omega=\Gamma$
with the counting measure $m_\Gamma$ by $(\gamma_1,\gamma_2):\gamma\mapsto \gamma_1\gamma\gamma_2^{-1}$
provides the \embf{trivial} self ME-coupling, giving reflexivity; symmetry is obvious from the 
definition\footnote{ One should formally distinguish between $(\Omega,m)$ as a $(\Gamma,\Lambda)$
coupling, and the same space with the same actions as a $(\Lambda,\Gamma)$
coupling; hereafter we shall denote the latter by $(\check{\Omega},\check{m})$.
Example~\ref{E:ME-lattices} illustrates the need to do so.};    
while transitivity follows from the following construction of \embf{composition},
or \embf{fusion}, of ME-couplings. 
If $(\Omega,m)$ is a $(\Gamma_1,\Gamma_2)$ coupling and $(\Omega',m')$ 
is a $(\Gamma_2,\Gamma_3)$ coupling, then the quotient $\Omega''=\Omega\times_{\Gamma_2}\Omega'$
of $\Omega\times\Omega'$ under the diagonal action 
$\gamma_2:(\omega,\omega')\mapsto (\gamma_2\omega,\gamma_2^{-1}\omega')$
inherits a measure $m''=m\times_{\Gamma_2} m'$ so that $(\Omega'',m'')$ becomes
a $(\Gamma_1,\Gamma_3)$ coupling structure. The indices satisfy:
\[
	[\Gamma_1:\Gamma_3]_{\Omega''}=[\Gamma_1:\Gamma_2]_{\Omega}\cdot[\Gamma_2:\Gamma_3]_{\Omega'}.
\]

\medskip

\embf{(d)} The notion of ME can be extended to the broader class of all \emph{unimodular locally compact
second countable} groups: a ME-coupling of $G$ and $H$ is a measure space
$(\Omega,m)$ with measure space isomorphisms 
\[
	i:(G,m_G)\times (Y,\nu)\cong(\Omega,m),\qquad j:(H,m_H)\times (X,\mu)\cong (\Omega,m)
\]
with $(X,\mu)$, $(Y,\nu)$ being \emph{finite} measure spaces, so that 
the actions $G\acts (\Omega,m)$, $H\acts (\Omega,m)$
given by $g:i(g',y)\mapsto i(gg',y)$, $h:j(h',x)\mapsto j(hh',x)$, commute. 
The index is defined by $[G:H]_\Omega=\mu(X)/\nu(Y)$.

\medskip

\embf{(e)} Measure Equivalence between countable groups can be viewed as
a \emph{category}, whose \emph{objects} are countable groups
and \emph{morphisms} between, say $\Gamma$ and $\Lambda$, 
are possible $(\Gamma,\Lambda)$ couplings.
Composition of morphisms is the operation of composition of ME-couplings
as in (c).
The trivial ME-coupling $(\Gamma,m_\Gamma)$ is nothing but the \emph{identity} of the object $\Gamma$.
It is also useful to consider \embf{quotient maps} $\Phi:(\Omega_1,m_1)\to(\Omega_2,m_2)$
between $(\Gamma,\Lambda)$-couplings (these are 2-morphisms in the category),
which are assumed to be $\Gamma\times\Lambda$ non-singular maps, i.e., $\Phi_*[m_1]\sim m_2$.
Since preimage of a fundamental domain is a fundamental domain, it follows 
(under ergodicity assumption) that $m_1(\Phi^{-1}(E))=c\cdot m_2(E)$, $E\subset \Omega_2$, where $0<c<\infty$.
ME self couplings of $\Gamma$ which have the trivial
$\Gamma$-coupling are especially useful, their cocycles are conjugate to isomorphisms. 
Similarly, $(\Gamma,\Lambda)$-couplings which have 
a discrete coupling as a quotient, correspond to virtual isomorphisms (see Lemma~\ref{L:v-iso}).

\medskip

\embf{(f)} Finally, one might relax the definition of quotients by considering equivariant
maps $\Phi:\Omega_1\to\Omega_2$ between $(\Gamma_i,\Lambda_i)$-couplings $(\Omega_i,m_i)$
with respect to homomorphisms $\Gamma_1\to\Gamma_2$, $\Lambda_1\to\Lambda_2$
with finite kernels and co-kernels.

\medskip

Gromov's motivation for ME comes from the theory of lattices.
Recall that a subgroup $\Gamma$ of a locally compact second countable (\embf{lcsc} for short) group $G$
is a \embf{lattice} if $\Gamma$ is discrete in $G$ and the quotient space $G/\Gamma$
carries a finite $G$-invariant Borel regular measure (necessarily unique up to normalization);
equivalently, if the $\Gamma$-action on $G$ by left (equivalently, right) translations
admits a Borel fundamental domain of finite positive Haar measure.
A discrete subgroup $\Gamma<G$ with $G/\Gamma$ being compact is automatically a lattice. 
Such lattices are called \embf{uniform} or \embf{cocompact}; others are \embf{non-uniform}. 
The standard example of a non-uniform lattice is $\Gamma=\SL_n(\bbZ)$ in $G=\SL_n(\bbR)$.
Recall that a lcsc group which admits a lattice is necessarily unimodular.

A common theme in the study of lattices (say in Lie, or algebraic groups over local fields) 
is that certain properties of the ambient group are inherited by its lattices.
From this perspective it is desirable to have a general framework in which 
lattices in the same group are considered equivalent.
Measure Equivalence provides such a framework.

\begin{example}\label{E:ME-lattices}
If $\Gamma$ and $\Lambda$ are lattices in the same lcsc group $G$, then $\Gamma\me \Lambda$;
the group $G$ with the Haar measure $m_G$ is a $(\Gamma,\Lambda)$ coupling
where 
\[
	(\gamma,\lambda):g\mapsto \gamma g\lambda^{-1}.
\]
(In fact, $\Gamma\me G\me \Lambda$ if ME is considered in the broader 
context of  unimodular lcsc groups: $G\times\{pt\}\cong \Gamma\times G/\Gamma$).
This example also illustrates the fact that the \emph{dual} $(\Lambda,\Gamma)$-coupling $\check G$ 
is better related to the original $(\Gamma,\Lambda)$-coupling $G$ via $g\mapsto g^{-1}$ rather
than the identity map.
\end{example}	 

\medskip

In \embf{Geometric Group Theory} the basic notion of equivalence is \embf{quasi-isometry} (QI).
Two metric spaces $(X_i,d_i)$, $i=1,2$ are quasi-isometric (notation: $X_1\qi X_2$) 
if there exist maps $f:X_1\to X_2$, $g:X_2\to X_1$, and constants $M,A$ so that
\begin{align*}
	&d_2(f(x),f(x'))<M\cdot d_1(x,x')+A\qquad &(x,x'\in X_1)\\
	&d_2(g(y),g(y'))<M\cdot d_2(y,y')+A\qquad &(y,y'\in X_2)\\
	&d_1(g\circ f(x),x)<A\qquad &(x\in X_1) \\
	&d_2(f\circ g(y),y)<A\qquad &(y\in X_2).
\end{align*}
Two finitely generated groups are QI if their Cayley graphs (with respect to some/any
finite sets of generators) are QI as metric spaces.
It is easy to see that finitely generated groups commensurable modulo finite groups 
are QI.

\medskip

Gromov observes that QI between groups can be characterized as \embf{Topological Equivalence}
(TE) defined in the the following statement.

\begin{theorem}[{Gromov \cite[Theorem 0.2.${\rm C}'_2$]{Gromov:1993asinv}}]
	\label{T:QI-TE}
Two finitely generated groups $\Gamma$ and $\Lambda$ are quasi-isometric iff
there exists a locally compact space $\Sigma$ with a continuous action of $\Gamma\times\Lambda$,
where both actions $\Gamma\acts \Sigma$ and $\Lambda\acts \Sigma$ are properly discontinuous
and cocompact.
\end{theorem}

The space $X$ in the above statement is called a \embf{TE-coupling}.
Here is an idea for the proof.
Given a TE-coupling $\Sigma$ one obtains a quasi-isometry from any point $p\in \Sigma$
by choosing $f:\Gamma\to\Lambda$, $g:\Lambda\to \Gamma$ so that
$\gamma p\in f(\gamma) X$ and $\lambda p\in g(\lambda) Y$,
where $X,Y\subset\Sigma$ are open sets with compact closures
and $\Sigma=\bigcup_{\gamma\in\Gamma} \gamma Y=\bigcup_{\lambda\in\Lambda} \lambda X$.
To construct a TE-coupling $\Sigma$ from a quasi-isometry $f:\Gamma\to\Lambda$,
consider the pointwise closure of the $\Gamma\times\Lambda$-orbit
of $f$ in the space of all maps $\Gamma\to\Lambda$ where $\Gamma$ acts 
by pre-composition on the domain and $\Lambda$ by post-composition on the image.
For more details see the guided exercise in \cite[p. 98]{delaHarpe:GGTbook}.

\medskip

A nice instance of QI between groups is a situation where
the groups admit a common \embf{geometric model}.
Here a \embf{geometric model} for a finitely generated group
$\Gamma$ is a (complete) separable metric space $(X,d)$ 
with a properly discontinuous and cocompact 
action of $\Gamma$ on $X$ by \emph{isometries}.
If $X$ is a common geometric model for $\Gamma_1$ and $\Gamma_2$, then
$\Gamma_1\qi X\qi \Gamma_2$.
For example, fundamental groups $\Gamma_i=\pi_1(M_i)$
of compact locally symmetric manifolds $M_1$ and $M_2$
with the same universal cover $\tilde{M}_1\cong\tilde{M}_2=X$
have $X$ as a common geometric model.
Notice that the common geometric model $X$ itself does not serve 
as a TE-coupling because the actions of the two groups do not commute.
However, a TE-coupling can be explicitly constructed
from the group $G=\Isom(X,d)$, which is locally compact
(in fact, compactly generated due to finite generation assumption
on $\Gamma_i$) second countable group.
Indeed, the isometric actions $\Gamma_i\acts (X,d)$
define homomorphisms $\Gamma_i\to G$ with finite kernels
and images being uniform lattices.
Moreover, the converse is also true: if $\Gamma_1,\Gamma_2$
admit homomorphisms with finite kernels and images being
uniform lattices in the same compactly generated second countable group
$G$, then they have a common geometric model -- take $G$ with a (pseudo-)metric
arising from an analogue of a word metric using compact sets.

Hence all uniform lattices in the same group $G$ are QI to each other.
Yet, typically, non-uniform lattices in $G$ are not QI to uniform ones
-- see Farb's survey \cite{Farb:1997qi} for the QI classification for lattices 
in semi-simple Lie groups.

\medskip

\emph{To summarize this discussion}: the notion of Measure Equivalence
is an equivalence relation between countable groups,
an important instance of which is given by groups which can imbedded
as lattices (uniform or not) in the same lcsc group.
It can be viewed as a measure-theoretic analogue of the equivalence
relation of being Quasi-Isometric (for finitely generated groups),
by taking Gromov's Topological Equivalence point of view.
An important instance of QI/TE is given by groups
which can be imbedded as uniform lattices in the same lcsc group.
In this situation one has both ME and QI.
However, we should emphasize that this is merely an \emph{analogy}:
the notions of QI and ME do not imply each other. 


\bigskip

\subsection{Orbit Equivalence in Ergodic Theory} 
\label{sub:ergodic_theory_orbit_equivalence}\hfill\\

Ergodic Theory investigates dynamical systems
from measure-theoretic point of view. 
Hereafter we shall be interested in measurable, measure preserving group actions
on a standard non-atomic probability measure space, and will refer to such actions
as \emph{probability measure preserving} (\embf{p.m.p.}).
It is often convenient to assume the action to be \embf{ergodic}, i.e.,
to require all measurable $\Gamma$-invariant sets to be \embf{null}
or \embf{co-null} (that is $\mu(E)=0$ or $\mu(X\setminus E)=0$).

A basic question in this context concerns possible \emph{orbit structures} of actions.
Equivalence of orbit structures is captured by the following notions of Orbit Equivalence 
(the notion of an orbit structure itself is discussed in \S~\ref{sub:rel_basic_definition}).
\begin{definition}\label{D:OE}
Two p.m.p. actions $\Gamma\acts(X,\mu)$ and $\Lambda\acts (Y,\nu)$ are 
\embf{orbit equivalent} (abbreviated OE, denoted $\Gamma\acts(X,\mu)\ore \Lambda\acts (Y,\nu)$) 
if there exists a measure space isomorphism $T:(X,\mu)\cong (Y,\nu)$ which takes $\Gamma$-orbits
onto $\Lambda$-orbits.
More precisely, an orbit equivalence is a Borel isomorphism 
$T:X'\cong Y'$ between co-null subsets 
$X'\subset X$ and $Y'\subset Y$ with $T_*\mu(E)=\mu(T^{-1}E)=\nu(E)$, $E\subset Y'$ and
$T(\Gamma.x\cap X')=\Lambda.T(x)\cap Y'$ for $x\in X'$.

A \embf{weak OE}, or \embf{Stable OE} (SOE) is a Borel isomorphism $T:X'\cong Y'$ between
\emph{positive measure} subsets $X'\subset X$ and $Y'\subset Y$
with $T_*\mu_{X'}=\nu_{Y'}$, where $\mu_{X'}=\mu(X')^{-1}\cdot \mu|_{X'}$, $\nu_{Y'}=\nu(Y')^{-1}\cdot \nu|_{Y'}$,
so that $T(\Gamma.x\cap X')=\Lambda.T(x)\cap Y'$ for all $x\in X'$.
The \embf{index} of such SOE-map $T$ is $\mu(Y')/\nu(X')$.
\end{definition} 

\medskip

In the study of orbit structure of dynamical systems in the topological
or smooth category, one often looks at such objects as fixed or periodic
points/orbits. 
Despite the important role these notions play in the underlying dynamical system,
periodic orbits have zero measure and therefore are invisible from the purely measure-theoretic
standpoint. 
Hence OE in Ergodic Theory is a study of the global orbit structure.
This point of view is consistent with the general philosophy 
of "non-commutative measure theory", i.e. von Neumann algebras.
Specifically OE in Ergodic Theory is closely related to the theory of
$\II_1$ factors as follows.

In the 1940s Murray and von Neumann introduced the so called "group-measure space"
construction to provide interesting examples of von-Neumann factors\footnote{ von Neumann algebras whose center consists only
of scalars.}: given a probability measure preserving (or more generally, non-singular) group 
action $\Gamma\acts (X,\mu)$ the associated von-Neumann algebra 
$M_{\Gamma\acts X}$ is a cross-product of $\Gamma$ with the Abelian algebra $L^\infty(X,\mu)$,
namely the weak closure in bounded operators on $L^2(\Gamma\times X)$
of the algebra generated by the operators $\setdef{f(g,x)\mapsto f(\gamma g,\gamma.x)}{\gamma\in\Gamma}$
and $\setdef{f(g,x)\mapsto \phi(x)f(g,x)}{\phi\in L^\infty(X,\mu)}$.
Ergodicity of $\Gamma\acts (X,\mu)$ is equivalent to $M_{\Gamma\acts X}$ being a factor.
It turns out that (for \emph{essentially free}) OE actions $\Gamma\acts X\ore\Lambda\acts Y$
the associated algebras are isomorphic $M_{\Gamma\acts X}\cong M_{\Lambda\acts Y}$,
with the isomorphism identifying the Abelian subalgebras $L^\infty(X)$ and $L^\infty(Y)$.
The converse is also true (one has to specify, in addition, an element in $H^1(\Gamma\acts X,\bbT)$)
-- see Feldman-Moore \cite{Feldman+Moore:77I, Feldman+Moore:77II}.
So Orbit Equivalence of (essentially free p.m.p. group actions) fits into the study
of $\II_1$ factors $M_{\Gamma\acts X}$ with a special focus on the so called \emph{Cartan subalgebra} 
given by $L^\infty(X,\mu)$. 
We refer the reader to Popa's 2006 ICM lecture \cite{Popa:2007ICM} and Vaes' Seminar Bourbaki  paper
\cite{Vaes:2007sb} for some more recent reports on this rapidly developing area. 

\medskip

The above mentioned assumption of \embf{essential freeness} of an action $\Gamma\acts (X,\mu)$
means that, up to a null set, the action is free; equivalently, for $\mu$-a.e. $x\in X$
the stabilizer $\setdef{\gamma\in\Gamma}{\gamma.x=x}$ is trivial.
This is a natural assumption, when one wants the acting group $\Gamma$
to "fully reveal itself" in a.e. orbit of the action.   
Let us now link the notions of OE and ME.

\begin{theorem}\label{T:ME-SOE}
Two countable groups $\Gamma$ and $\Lambda$ are Measure Equivalent iff
they admit essentially free (ergodic) probability measure preserving actions
$\Gamma\acts (X,\mu)$ and $\Lambda\acts (Y,\nu)$ which are Stably Orbit Equivalent.
\end{theorem}
(SOE)$\ \Longrightarrow\ $(ME) direction is more transparent in the special case of Orbit
Equivalence, i.e., index one. Let $\alpha:\Gamma\times X\to\Lambda$ be
the cocycle associated to an orbit equivalence $T:(X,\mu)\to(Y,\nu)$ 
defined by $T(g.x)=\alpha(g,x).T(x)$ (here freeness of $\Lambda\acts Y$ is used).
Consider $(\Omega,m)=(X\times\Lambda,\mu\times m_\Lambda)$ with the actions 
\begin{equation}\label{e:XLambda}
	g: (x,h)\mapsto (g x,\alpha(g,x)h),\qquad
	h: (x,k)\mapsto (x,hk^{-1})\qquad (g\in\Gamma,\ h\in\Lambda).
\end{equation}
Then $X\times\{1\}$ is a common fundamental domain for both actions
(note that here freeness of $\Gamma\acts X$ is used).
Of course, the same coupling $(\Omega,m)$ can be viewed as $(Y\times\Gamma,\nu\times m_\Gamma)$
with the $\Lambda$-action defined using $\beta:\Lambda\times Y\to\Gamma$ given by
$T^{-1}(h.y)=\beta(h,y).T^{-1}(y)$.
In the more general setting of \emph{Stable} OE one needs to adjust the definition
for the cocycles (see \cite{Furman:OE:99}) to carry out a similar construction.

Alternative packaging for the (OE)$\ \Longrightarrow\ $(ME) argument uses
the language of equivalence relations (see \S \ref{sub:rel_basic_definition}). 
Identifying $Y$ with $X$ via $T^{-1}$, one views $\Rel_{\Lambda\acts Y}$
and $\Rel_{\Gamma\acts X}$ as a single relation $\Rel$.
Taking $\Omega=\Rel$ equipped with the measure $\tilde{\mu}$ (\ref{e:tildemu}) consider the actions
\[
	g:(x,y)\mapsto (g.x,y),\qquad h:(x,y)\mapsto (x,h.y)\qquad (g\in\Gamma,\ h\in\Lambda).
\]
Here the diagonal embedding $X\mapsto\Rel$, $x\mapsto (x,x)$, gives the fundamental domain
for both actions.

(ME)$\ \Longrightarrow\ $(SOE).
Given an ergodic $(\Gamma,\Lambda)$ coupling $(\Omega,m)$, let $X,Y\subset\Omega$ be
fundamental domains for the $\Lambda$, $\Gamma$ actions;
these may be chosen so that $m(X\cap Y)>0$.
The finite measure preserving actions 
\begin{equation}\label{e:MESEO}
	\Gamma\acts X\cong \Omega/\Lambda,\qquad\Lambda\acts Y\cong \Omega/\Gamma.
\end{equation}
have weakly isomorphic orbit relations, since they appear as the restrictions to $X$ and $Y$ of the 
relation $\Rel_{\Gamma\times\Lambda\acts \Omega}$ (of type $\II_\infty$); these restrictions coincide on $X\cap Y$.
The index of this SOE coincides with the ME-index $[\Gamma:\Lambda]_\Omega$
(if $[\Gamma:\Lambda]_\Omega=1$ one can find a common fundamental domain $X=Y$).
The only remaining issue is that the actions $\Gamma\acts X\cong\Omega/\Lambda$, 
$\Lambda\acts Y\cong\Omega/\Gamma$ may not be essential free.
This can be fixed (see \cite{Gaboriau:2000survey}) by passing to an extension 
$\Phi:(\bar{\Omega},\bar{m})\to(\Omega,m)$ 
where $\Gamma\acts \bar\Omega/\Lambda$ and $\Lambda\acts\bar\Omega/\Gamma$
are essentially free. Indeed, take $\bar\Omega=\Omega\times Z\times W$,
where $\Lambda\acts Z$ and $\Lambda\acts W$ are \emph{free probability measure preserving actions}
and let
\[
	g:(\omega,z,w)\mapsto (g\omega,g z,w),\qquad
	h:(\omega,z,w)\mapsto (h\omega,z,h w)\qquad(g\in \Gamma,\ h\in\Lambda).
\]
%
\begin{remark}
Freeness of actions is mostly used in order to \emph{define} the rearrangement cocycles
for a (stable) orbit equivalence between actions. However, if SOE comes from a ME-coupling
the well defined ME-cocycles \emph{satisfy} the desired rearrangement property (such as
$T(g.x)=\alpha(g,x).T(x)$) and freeness becomes superfluous.  

If $\Phi:\bar\Omega\to\Omega$ is as above, and $\bar{X}, \bar{Y}$ denote the preimages of $X,Y$, 
then $\bar{X}, \bar{Y}$ are $\Lambda, \Gamma$ fundamental domains, the OE-cocycles 
$\Gamma\acts \bar{X}\soe\Lambda\acts\bar Y$ coincide with the ME-cocycles associated with 
$X,Y\subset\Omega$. 

Another, essentially equivalent, point of view is that ME-coupling defines a \emph{weak isomorphism}
between the \embf{groupoids} $\Gamma\acts \Omega/\Lambda$ and $\Lambda\acts \Omega/\Gamma$.
In case of free actions these groupoids reduce to their \emph{relations groupoids},
but in general the information about stabilizers is carried by the ME-cocycles.  
\end{remark} 


\bigskip
 
\subsection{Further comments on QI, ME and related topics} 
\label{sub:further comments}\hfill\\

Let $\Sigma$ be Gromov's Topological Equivalence between $\Gamma$ and $\Lambda$.
Then any point $x\in\Sigma$ defines a quasi-isometry $q_x:\Gamma\to\Lambda$
(see the sketch of proof of Theorem~\ref{T:QI-TE}).
In ME the maps $\alpha(-,x):\Gamma\to\Lambda$ defined for a.e. $x\in X$
play a similar role. 
However due to their measure-theoretic nature, such maps are insignificant taken individually,
and are studied as a measured family with the additional structure given by
the cocycle equation.

\medskip

Topological and Measure Equivalences are related to the following interesting
notion, introduced by Nicolas Monod in \cite{Monod:2006-ICM} under the appealing term \embf{randomorphisms}.
Consider the Polish space $\Lambda^\Gamma$ of all maps $f:\Gamma\to \Lambda$ with the product uniform topology,
and let 
\[
	[\Gamma,\Lambda]=\setdef{f:\Gamma\to\Lambda}{f(e_\Gamma)=e_\Lambda}.
\]
Then $\Gamma$ acts on $[\Gamma,\Lambda]$ by $g: f(x)\mapsto f(xg)f(g)^{-1}$, $x\in \Gamma$.
The basic observation is that homomorphisms $\Gamma\to\Lambda$ are precisely $\Gamma$-fixed points
of this action.
\begin{defn} \label{D:randomorphism}
A \embf{randomorphism} is a $\Gamma$-invariant probability measure on $[\Gamma,\Lambda]$.
\end{defn}
A measurable cocycle $c:\Gamma\times X\to\Lambda$ over a p.m.p. action $\Gamma\acts (X,\mu)$
defines a randomorphism by pushing forward the measure $\mu$ by the cocycle $x\mapsto c(-,x)$.
Thus Orbit Equivalence cocycles (see \ref{sub:oe_cocycles}) correspond to randomorphisms
supported on \emph{bijections} in $[\Gamma,\Lambda]$.
Also note that the natural composition operation for randomorphisms, given by the push-forward
of the measures under the natural map
\[
	[\Gamma_1,\Gamma_2]\times[\Gamma_2,\Gamma_3]\to[\Gamma_1,\Gamma_3],\qquad (f,g)\mapsto g\circ f
\]
corresponds to composition of couplings.
The view point of topological dynamics of the $\Gamma$-action on $[\Gamma,\Lambda]$
may be related to quasi-isometries and Topological Equivalence. 
For example, points in $[\Gamma,\Lambda]$ with precompact $\Gamma$-orbits
correspond to Lipschitz embeddings $\Gamma\to\Lambda$.

\medskip

\subsubsection{Using ME for QI} 
\label{ssub:using_me_for_qi}\hfill\\

Although Measure Equivalence and Quasi Isometry are parallel in many ways,
these concepts are different and neither one implies the other.
Yet, Yehuda Shalom has shown \cite{Shalom:2006:acta} how one can use ME ideas to study QI of \emph{amenable} groups.
The basic observation is that a topological coupling $\Sigma$ of amenable groups $\Gamma$ and $\Lambda$
carries a $\Gamma\times\Lambda$-invariant measure $m$ (coming from a $\Gamma$-invariant \emph{probability measure} 
on $\Sigma/\Lambda$), which gives a measure equivalence. 
It can be thought of as an invariant distribution on quasi-isometries $\Gamma\to\Lambda$,
and can be used to induce unitary representations and cohomology with unitary coefficients etc.
from $\Lambda$ to $\Gamma$.
Using such constructions, Shalom \cite{Shalom:2006:acta} was able to obtain a list of new QI invariants in the class of amenable groups,
such as (co)-homology over $\bbQ$, ordinary Betti numbers $\beta_i(\Gamma)$ among nilpotent groups and others.
Shalom also studied the notion of \embf{uniform embedding} (UE) between groups 
and obtained group invariants which are \emph{monotonic} with respect to UE.

In \cite{Sauer:2006:GAFA} Roman Sauer obtains further QI-invariants and UE-monotonic invariants using a combination of
QI, ME and homological methods. 
 
In another work \cite{Sauer:2006p69} Sauer used ME point of view
to attack problems of purely topological nature, related to the work of Gromov.


\medskip

\subsubsection{$\ell^p$-Measure Equivalence} 
\label{ssub:_ell_p_measure_equivalence}\hfill\\

Let $\Gamma$ and $\Lambda$ be finitely generated groups, equipped with some word metrics 
$|\cdot |_\Gamma$, $|\cdot |_\Lambda$.
We say that a $(\Gamma,\Lambda)$ coupling $(\Omega,m)$ is $\ell^p$ for some $1\le p\le\infty$
if there exist fundamental domains $X,Y\subset\Omega$ so that the associated ME-cocycles 
(see \ref{sub:me_cocycles})
$\alpha:\Gamma\times X\to\Lambda$ and $\beta:\Lambda\times Y\to\Gamma$ satisfy
\[
	\forall g\in\Gamma:\quad |\alpha(g,-)|_\Lambda\in L^p(X,\mu),\qquad
	\forall h\in\Lambda:\quad |\beta(h,-)|_\Gamma\in L^p(Y,\nu).
\]
If an $\ell^p$-ME-coupling exists, say that $\Gamma$ and $\Lambda$ are $\ell^p$-ME.
Clearly any $\ell^p$-ME-coupling is $\ell^q$ for all $q\le p$.
So $\ell^1$-ME is the weakest and $\ell^\infty$-ME is the most stringent among these relations.
One can check that $\ell^p$-ME is an equivalence relation on groups (the $\ell^p$ condition is preserved under 
composition of couplings), so we obtain a hierarchy of $\ell^p$-ME categories with $\ell^1$-ME being the weakest 
(largest classes) and at $p=\infty$ one arrives at ME+QI.
Thus $\ell^p$-ME amounts to Measure Equivalence with some geometric flavor. 

The setting of $\ell^1$-ME is considered in \cite{Bader+Furman+Sauer:2009, Bader+Furman+Sauer:2010} by Uri Bader, Roman Sauer and the author
to analyze rigidity of the least rigid family of lattices -- lattices in ${\rm SO}_{n,1}(\bbR)\simeq\Isom(\HH^n_\bbR)$,
$n\ge 3$, and fundamental groups of general negatively curved manifolds. 
It should be noted, that examples of non-amenable ME groups which are not $\ell^1$-ME 
seem to be rare (surface groups and free groups seem to be the main culprits).
In particular, it follows from Shalom's computations in \cite{Shalom:2000p101} that for $n\ge 3$
all lattices in ${\rm SO}_{n,1}(\bbR)$ are mutually $\ell^1$-ME.
We shall return to invariants and rigidity in $\ell^1$-ME framework in \S \ref{ssub:dimension_and_simplicial_volume} 
and \S \ref{ssub:hyperbolic_lattices}.


    
%
%


\newpage
\section{Measure Equivalence between groups} 
\label{sec:measure_equivalence}


This section is concerned with the notion of Measure Equivalence between 
countable groups $\Gamma\me\Lambda$ (Definition~\ref{D:ME}).
First recall the following deep result (extending previous work of Dye \cite{Dye:1959, Dye:1963}
on some amenable groups, and followed by Connes-Feldman-Weiss \cite{Connes+Feldman+Weiss:81} 
concerning all non-singular actions of all amenable groups)
\begin{theorem}[{Ornstein-Weiss \cite{Ornstein+Weiss:80rohlin}}]
\label{T:Ornstein-Weiss}
Any two ergodic probability measure preserving actions of any two infinite countable
amenable groups are Orbit Equivalent. 
\end{theorem}	
This result implies that all infinite countable amenable groups are ME; moreover for any two infinite amenable
groups $\Gamma$ and $\Lambda$ there exists an ergodic ME-coupling $\Omega$ with index $[\Gamma:\Lambda]_\Omega=1$
(hereafter we shall denote this situation by $\Gamma\ore\Lambda$).
Measure Equivalence of all amenable groups shows that many QI-invariants are not ME-invariants;
these include: growth type, being virtually nilpotent, (virtual) cohomological dimension, finite generations/presentation etc.

\medskip

The following are basic constructions and examples of Measure Equivalent groups:
\begin{enumerate}
	\item 
	If $\Gamma$ and $\Lambda$ can be embedded as lattices in the same lcsc group, then $\Gamma\me\Lambda$.
	\item\label{i:free}
	If $\Gamma_i\me\Lambda_i$ for $i=1,\dots,n$ then $\Gamma_1\times\cdots \times\Gamma_n\me \Lambda_1\times\cdots\times\Lambda_n$.
	\item
	If $\Gamma_i\ore\Lambda_i$ for $i\in I$ (i.e. the groups admit an ergodic ME-coupling with index one) then 
	$(*_{i\in I}\Gamma_i)\ore (*_{i\in I}\Lambda_i)$\footnote{ The appearance of the sharper condition 
	$\ore$ in (\ref{i:free}) is analogous to the one in the QI context:
	if groups $\Gamma_i$ and $\Lambda_i$ are \emph{bi-Lipschitz} then $*_{i\in I}\Gamma_i\qi *_{i\in I}\Lambda_i$.}. 
\end{enumerate}

\medskip

For $2\ne n,m<\infty$ the free groups $\bbF_n$ and $\bbF_m$ are commensurable, and therefore are ME
(however, $\bbF_\infty\not\me\bbF_2$).
The Measure Equivalence class $\ME{\bbF_{2\le n<\infty}}$ is very rich and remains mysterious (see \cite{Gaboriau:2005exmps}). 
For example it includes: surface groups $\pi_1(\Sigma_g)$, $g\ge 2$, non uniform (infinitely generated) lattices in $\SL_2(F_p[[X]])$, 
the automorphism group of a regular tree, free products $*_{i=1}^n A_i$ of arbitrary infinite amenable groups, more complicated free products such as $\bbF_2*\pi_1(\Sigma_g)*\bbQ$, etc. 
In the aforementioned paper by Gaboriau he constructs interesting geometric examples
of the form $*_{c}^n\bbF_{2g}$, which are fundamental groups of certain "branched surfaces".
Bridson, Tweedale and Wilton \cite{Bridson+Tweedale+Wilton:ETDS} prove that a large class of \emph{limit groups},
namely all \emph{elementarily free} groups, are ME to $\bbF_2$.
Notice that $\ME{\bbF_{2\le n<\infty}}$ contains uncountably many groups.

\medskip

The fact that some ME classes are so rich and complicated should emphasize the impressive list of ME invariants and rigidity results below.


\bigskip

\subsection{Measure Equivalence Invariants} 
\label{sub:me_invariants}\hfill\\

By ME-\embf{invariants} we mean properties of groups which are preserved
under Measure Equivalence, and numerical invariants which are preserved
or predictably transformed as a function of the ME index.

\medskip

\subsubsection{Amenability, Kazhdan's property (T), a-T-menability} 
\label{ssub:amenability_property_t_a_t_menability}\hfill\\

These properties are defined using the language of unitary representations.
Let $\pi:\Gamma\to U(\scrH)$ be a unitary representation of a (topological) group.
Given a finite (resp. compact) subset 
$K\subset G$ and $\epsilon>0$, we say that a unit vector $v\in \scrH$ is $(K,\epsilon)$-\embf{almost invariant} 
if $\|v-\pi(g)v\|<\epsilon$ for all $g\in K$.
A unitary $\Gamma$-representation $\pi$ which has $(K,\epsilon)$-almost invariant vectors for
all $K\subset\Gamma$ and $\epsilon>0$ is said to \embf{weakly contain} the trivial representation ${\bf 1}_\Gamma$,
denoted ${\bf 1}_\Gamma\prec\pi$.
The trivial representation ${\bf 1}_\Gamma$ is \embf{(strongly) contained} in $\pi$, denoted ${\bf 1}_\Gamma <\pi$, if
there exist non-zero $\pi(G)$-invariant vectors, i.e., $\scrH^{\pi(\Gamma)}\neq \{0\}$.
Of course ${\bf 1}_\Gamma<\pi$ trivially implies ${\bf 1}_\Gamma\prec\pi$.
We recall:
\begin{description}
	\item[{\sc Amenability}] 
	$\Gamma$ is \embf{amenable} if the trivial representation is weakly contained in the regular representation
	$\rho:\Gamma\to U(\ell^2(\Gamma))$,
	$\rho(g) f(x)=f(g^{-1}x)$.
	\item[{\sc Property (T)}]
	$\Gamma$ has \embf{property (T)} (Kazhdan \cite{Kazhdan:1967T})
	if for every unitary $\Gamma$-representation $\pi$:
	${\bf 1}_\Gamma\prec\pi$ implies ${\bf 1}_\Gamma<\pi$. 
	This is equivalent to an existence of a compact $K\subset \Gamma$ and $\epsilon>0$ so that 
	any unitary $\Gamma$-representation $\pi$
	with $(K,\epsilon)$-almost invariant vectors, has non-trivial invariant vectors.
	For compactly generated groups, another equivalent characterization (Delorme and Guichardet) 
	is that	any affine isometric $\Gamma$-action on a Hilbert space has a fixed point,
	i.e., if $H^1(\Gamma,\pi)=\{0\}$ for any (orthogonal) $\Gamma$-representation $\pi$.
	We refer to \cite{Bekka+deLaHarpe+Valette:2008T} for the details.
	\item[{\sc (HAP)}]
	$\Gamma$ is \embf{a-T-menable} (or has \embf{Haagerup Approximation Property}) if the following 
	equivalent conditions hold:  (i) $\Gamma$ has a mixing $\Gamma$-representation weakly containing the trivial one, 
	or (ii) $\Gamma$ has a proper affine isometric action on a (real) Hilbert space.
	The class of infinite a-T-menable groups contains amenable groups, free groups but is disjoint  
	from infinite groups with property (T).
	See \cite{Cherix+Cowling++:2001} as a reference.
\end{description}
Measure Equivalence allows to relate unitary representations of one group to another.
More concretely, let $(\Omega,m)$ be a $(\Gamma,\Lambda)$ coupling, and $\pi:\Lambda\to U(\scrH)$
be a unitary $\Lambda$-representation. Denote by $\tilde{\scrH}$ the Hilbert space
consisting of equivalence classes (mod null sets) of all measurable, $\Lambda$-equivariant
maps $\Omega\to\scrH$ with square-integrable norm over a $\Lambda$-fundamental domain:
\[
	\tilde\scrH=\setdef{f:\Omega\to \scrH}{ f(\lambda x)=\pi(\lambda)f(x),\ 
	\int_{\Omega/\Lambda}\|f\|^2<\infty} 
	\quad \mod \text{null sets}.
\]
The action of $\Gamma$ on such functions by translation of the argument, defines a unitary 
$\Gamma$-representation $\tilde{\pi}:\Gamma\to U(\tilde{\scrH})$.
This representation is said to be \embf{induced} from $\pi:\Lambda\to U(\scrH)$ via $\Omega$.
(In example~\ref{E:ME-lattices} this is precisely the usual Mackey induction of a unitary representations 
of a lattice to the ambient group, followed by a restriction to another lattice).

The ME invariance of the properties above (amenability, property (T), Haagerup approximation property)
can be deduced from the following observations.
Let $(\Omega,m)$ be a $(\Gamma,\Lambda)$ ME-coupling, $\pi:\Lambda\to U(\scrH)$ a unitary representation
and $\tilde{\pi}:\Gamma\to U(\tilde\scrH)$ the corresponding induced representation. Then:
\begin{enumerate}
	\item \label{i:regular}
	If $\pi$ is the regular $\Lambda$-representation on $\scrH=\ell^2(\Lambda)$, then 
	$\tilde\pi$ on $\tilde\scrH$ can be identified with the $\Gamma$-representation
	on $L^2(\Omega,m)\cong n\cdot \ell^2(\Gamma)$, where $n=\dim L^2(\Omega/\Lambda)\in\{1,2,\dots,\infty\}$.
	\item \label{i:alm-inv}
	If ${\bf 1}_\Lambda\prec \pi$ then ${\bf 1}_\Gamma\prec \tilde\pi$. 
	\item \label{i:wm}
	If $(\Omega,m)$ is $\Gamma\times\Lambda$ ergodic and $\pi$ is \emph{weakly mixing}
	(i.e. ${\rm 1}_\Lambda\not<\pi\otimes\pi^*$) then ${\bf 1}_\Gamma\not<\tilde\pi$.
	\item \label{i:mix}
	If $(\Omega,m)$ is $\Gamma\times\Lambda$ ergodic and $\pi$ is \emph{mixing} 
	(i.e. for all $v\in\scrH$: $\langle\pi(h)v,v\rangle\to 0$ as $h\to\infty$ in $\Lambda$)
	then $\tilde\pi$ is a mixing $\Gamma$-representation.
\end{enumerate}
Combining (\ref{i:regular}) and (\ref{i:alm-inv}) we obtain that being amenable is an ME-invariant.
The deep result of Ornstein-Weiss \cite{Ornstein+Weiss:80rohlin} and Theorem~\ref{T:ME-SOE}
imply that any two infinite countable amenable groups are ME. This gives: 
\begin{corollary}\label{C:MEZ-Amen}
The Measure Equivalence class of $\bbZ$ is the class of all infinite countable amenable groups
\[
	\ME{\bbZ}={\rm Amen}.
\] 	
\end{corollary}	
Bachir Bekka and Alain Valette \cite{Bekka+Valette:1993wmT} showed that if $\Lambda$ does not have property (T)
then it admits a \emph{weakly mixing} representation $\pi$ weakly containing the trivial one.
By (\ref{i:alm-inv}) and (\ref{i:wm}) this implies that property (T) is a ME-invariant 
(this is the argument in \cite[Corollary 1.4]{Furman:ME:99}, see also Zimmer \cite[Theorem 9.1.7 (b)]{Zimmer:book:84}).
The ME-invariance of amenability and Kazhdan's property for groups indicates that it should be possible
to define these properties for \emph{equivalence relations} and then relate them to groups.
This was done by Zimmer in \cite{Zimmer:Amenable:78, Zimmer:1981:cohom} and was recently further studied in the context 
of measured groupoids in \cite{Anantharaman-Delaroche+Renault:2000, Anantharaman-Delaroche:2005T}. 
We return to this discussion in \S \ref{ssub:amenability_strong_ergodicity_property_t_}. 
The ME-invariance of a-T-menability follows from (\ref{i:alm-inv}) and (\ref{i:mix}); 
see \cite{Jolissaint:2001, Cherix+Cowling++:2001}.


\medskip

\subsubsection{Cost of groups} 
\label{ssub:cost_of_groups}\hfill\\

The notion of the cost of an action/relation was introduced by Levitt \cite{Levitt:1995cost}
and developed by Damien Gaboriau \cite{Gaboriau:CRAS1998cost, Gaboriau:Inven2000cost, Gaboriau:2000survey};
the monographs \cite{Kechris+Miller:2004book} and \cite{Kechris:new} also contain an extensive
discussion of this topic.

The cost of an essentially free p.m.p. action $\Gamma\acts (X,\mu)$, denoted $\cost(\Gamma\acts X)$, 
is the cost of the corresponding orbit relations $\cost(\Rel_{\Gamma\acts X})$ 
as defined in \S~\ref{ssub:rel_cost} (it is the infimum of the weights of generating systems
for the groupoid where the "weight" is the sum of the measures of the domain/image sets of the generating system).
%
The cost of an action can be turned into a group invariant/s by setting
\[
	\gcostinf(\Gamma)=\inf_X\ \cost(\Gamma\acts X),\qquad \gcostsup(\Gamma)=\sup_X\ \cost(\Gamma\acts X)
\]
where the infimum/supremum are taken over all essentially free p.m.p. actions 
of $\Gamma$ (we drop ergodicity assumption here; in the definition of $\gcostinf(\Gamma)$ essential freeness 
is also superfluous). 
Groups $\Gamma$ for which $\gcostinf(\Gamma)=\gcostsup(\Gamma)$ are said to have fixed price, or 
\embf{prix fixe} (abbreviated P.F.). 
For general groups, Gaboriau defined \embf{the cost of a group} to be the lower one: 
\[
	\gcost(\Gamma)=\gcostinf(\Gamma).
\]
To avoid confusion, we shall use here the notation $\gcostinf(\Gamma)$ for general groups, and reserve 
$\gcost(\Gamma)$ for P.F. groups only.
\begin{question}\label{Q:fixed-price}
Do all countable groups have property P.F.?
\end{question}
The importance of this question will be illustrated in \S~\ref{ssub:rel_cost};
for example a positive answer to an apparently weaker question \ref{Q:cost-finind} 
would have applications to groups theory and 3-manifold (Abert-Nikolov \cite{Abert+Nikolov:2007cost}). 

The properties  $\gcostinf=1$, $1<\gcostinf<\infty$, and $\gcostinf=\infty$ are ME-invariants.
More precisely:
\begin{theorem}\label{T:gcostinf}
If $\Gamma\me\Lambda$ then 
$
	\gcostinf(\Lambda)-1=[\Gamma:\Lambda]_\Omega\cdot \left(\gcostinf(\Gamma)-1\right)
$
for some/any $(\Gamma,\Lambda)$-coupling $\Omega$.
\end{theorem}
We do not know whether the same holds for $\gcostsup$. 
Note that in \cite{Gaboriau:Inven2000cost} this ME-invariance is stated for P.F. groups only. 
\begin{proof}
Let $\Omega$ be a $(\Gamma,\Lambda)$-coupling with $\Gamma\acts X=\Omega/\Lambda$
and $\Lambda\acts Y=\Omega/\Gamma$ being free, where $X,Y\subset \Omega$
are $\Lambda-$, $\Gamma-$ fundamental domains.
Given any essentially free p.m.p. action $\Lambda\acts Z$, consider 
the $(\Gamma,\Lambda)$-coupling $\bar\Omega=\Omega\times Z$ with the actions 
\[
	g:(\omega,z)\mapsto (g\omega,z),\qquad h:(\omega,z)\mapsto (h\omega,hz)
	\qquad(g\in\Gamma,\ h\in\Lambda).
\]
The actions $\Gamma\acts \bar{X}=\bar{\Omega}/\Lambda$ and $\Lambda\acts \bar{Y}=\bar\Omega/\Gamma$
are Stably Orbit Equivalent with index $[\Gamma:\Lambda]_{\bar\Omega}=[\Gamma:\Lambda]_\Omega=c$.
Hence (using Theorem~\ref{T:cost-restriction} below) we have
\[
	c\cdot (\cost(\Rel_{\Gamma\acts \bar{X}})-1)=\cost(\Rel_{\Lambda\acts \bar{Y}})-1.
\] 
While $\Gamma\acts\bar{X}$ is a skew-product over $\Gamma\acts X$, the action $\Lambda\acts \bar{Y}$
is the diagonal action on $\bar{Y}=Y\times Z$. 
Since $\bar{Y}=Y\times Z$ has $Z$ as a $\Lambda$-equivariant quotient,
it follows (by considering preimages of any "graphing system") that
\[
	\cost(\Lambda\acts \bar{Y})\le \cost(\Lambda\acts Z).
\]
Since $\Lambda\acts Z$ was arbitrary, we deduce $\gcostinf(\Lambda)-1\ge c\cdot \left(\gcostinf(\Gamma)-1\right)$.
A symmetric argument completes the proof.
\end{proof}
\begin{theorem}[{Gaboriau \cite{Gaboriau:CRAS1998cost, Gaboriau:Inven2000cost, Gaboriau:2000survey}}]\label{T:PF}
The following classes of groups have P.F.:
\begin{enumerate}
	\item\label{i:finite}
	Any finite group $\Gamma$ has $\gcostinf(\Gamma)=\gcostsup(\Gamma)=1-\frac{1}{|\Gamma|}$.
	\item \label{i:amenable}
	Infinite amenable groups have $\gcostinf(\Gamma)=\gcostsup(\Gamma)=1$.
	\item 
	Free group $\bbF_n$, $1\le n\le \infty$, have $\gcostinf(\bbF_n)=\gcostsup(\bbF_n)=n$.
	\item 
	Surface groups $\Gamma=\pi_1(\Sigma_g)$ where $\Sigma_g$ is a closed orientable surface of genus $g\ge 2$ have 
	$\gcostinf(\Gamma)=\gcostsup(\Gamma)=2g-1$.
	\item\label{i:costoffree}
	Amalgamated products $\Gamma=A*_CB$ of finite groups have P.F. with
	\[
		\gcostinf(\Gamma)=\gcostsup(\Gamma)=1-(\frac{1}{|A|}+\frac{1}{|B|}-\frac{1}{|C|}).
	\]
	In particular $\gcostinf(\SL_2(\bbZ))=\gcostsup(\SL_2(\bbZ))=1+\frac{1}{12}$.
	\item\label{i:costofamalgam}
	Assume $\Gamma_1$, $\Gamma_2$ have P.F. then the free product $\Gamma_1*\Gamma_2$, and more general 
	amalgamated free products $\Lambda=\Gamma_1*_A\Gamma_2$ over
	an amenable group $A$, has P.F. with
	\[
		\gcost(\Gamma_1*\Gamma_2)=\gcost(\Gamma_1)+\gcost(\Gamma_2),\qquad
		\gcost(\Gamma_1*_A\Gamma_2)=\gcost(\Gamma_1)+\gcost(\Gamma_2)-\gcost(A).
	\] 
	\item\label{i:product}
	Products $\Gamma=\Gamma_1\times\Gamma_2$ of infinite non-torsion groups have $\gcostinf(\Gamma)=\gcostsup(\Gamma)=1$.
	\item\label{i:normal-amen}
	Finitely generated groups $\Gamma$ containing an infinite amenable normal subgroup have 
	$\gcostinf(\Gamma)=\gcostsup(\Gamma)=1$.
	\item\label{i:Qrank2}
	Arithmetic lattices $\Gamma$ of higher $\bbQ$-rank (e.g. $\SL_{n\ge 3}(\bbZ)$) have 
	$\gcostinf(\Gamma)=\gcostsup(\Gamma)=1$. 
\end{enumerate}
\end{theorem}
Note that for an infinite group $\gcostsup(\Gamma)=1$ iff $\Gamma$ has P.F. of cost one. 
So the content of cases (\ref{i:amenable}), (\ref{i:product}), (\ref{i:normal-amen}), (\ref{i:Qrank2}) 
is that $\gcostsup(\Gamma)=1$.
\begin{question}
Is it true that for all (irreducible) lattices $\Gamma$ in a (semi-)simple Lie group $G$ of higher rank have P.F. 
of $\gcostsup(\Gamma)=1$? \\
Is it true that any infinite group $\Gamma$ with Kazhdan's property (T) has P.F. with $\gcostsup(\Gamma)=1$?
\end{question}
Item (\ref{i:Qrank2}) in Theorem~\ref{T:PF} provides a positive answer to the first question for some \emph{non-uniform lattices} 
in higher rank Lie groups, but the proof relies on the internal structure of such lattices
(chains of pairwise commuting elements), rather than on its relation to the ambient Lie group $G$ 
(which also has a lot of commuting elements).
Note also that Theorem~\ref{T:gcostinf} implies that $\gcostinf(\Gamma)=1$ for all higher rank lattices.
The motivation for the second question is that property (T) implies vanishing of the 
first $\ell^2$-Betti number, $\beta^{(2)}_1(\Gamma)=0$; while for infinite groups it was shown by Gaboriau that
\begin{equation}\label{e:beta2cost}
	\beta^{(2)}_1(\Gamma)=\beta^{(2)}_1(\Rel_{\Gamma\acts X})\le \cost(\Rel_{\Gamma\acts X})-1.
\end{equation}
Furthermore, there are no known examples of strict inequality.
Lattices $\Gamma$ in higher rank semi-simple Lie groups without property (T) still 
satisfy $\beta^{(2)}_1(\Gamma)=0$ (an argument in the spirit of the current discussion is:
$\beta_1^{(2)}$ for ME groups are positively proportional by Gaboriau's Theorem \ref{T:Gaboriau:L2},
an irreducible lattice in a product is ME to a product of lattices and products
of infinite groups have $\beta^{(2)}_1=0$ by K\"uneth formula. 
Shalom's \cite{Shalom:2000:Inven} provides a completely geometric explanation).

\medskip

To give the flavor of the proofs let us indicate the argument for (\ref{i:normal-amen}) in Theorem~\ref{T:PF}.
Let $\Gamma$ be a group generated by a finite set $\{g_1,\dots,g_n\}$ and containing an 
infinite normal amenable subgroup $A$ and $\Gamma\acts (X,\mu)$ be an essentially free (ergodic) p.m.p. action.
Since $A$ is amenable, there is a $\bbZ$-action on $X$ with $\Rel_{A\acts X}=\Rel_{\bbZ\acts X}$ (mod null sets),
and we let $\phi_0:X\to X$ denote the action of the generator of $\bbZ$.
Given $\epsilon>0$ one can find a subset $E\subset X$ with $0<\mu(E)<\epsilon$
so that $\bigcup_{a\in A} aE=\bigcup\phi_0^n E=X$ mod null sets (if $A$-action is ergodic any positive measure set
works; in general, one uses the ergodic decomposition).
For $i=1,\dots,n$ let $\phi_i$ be the restriction of $g_i$ to $E$.
Now one easily checks that the normality assumption implies that $\Phi=\{\phi_0,\phi_1,\dots,\phi_n\}$
generates $\Rel_{\Gamma\acts X}$, while $\cost(\Phi)=1+n\epsilon$.

%

\medskip
For general (not necessarily P.F.) groups $\Gamma_i$ a version of (\ref{i:costofamalgam}) still holds:
\[
	\gcostinf(\Gamma_1*\Gamma_2)=\gcostinf(\Gamma_1)+\gcostinf(\Gamma_2),
	\qquad
	\gcostinf(\Gamma_1*_A\Gamma_2)=\gcostinf(\Gamma_1)+\gcostinf(\Gamma_2)-\gcost(A)
\]
where $A$ is finite or, more generally, amenable.
 
\medskip

Very recently Miklos Abert and Benjamin Weiss \cite{Abert+Weiss:2008} showed:
\begin{thm}[Abert-Weiss \cite{Abert+Weiss:2008}]\label{T:Abert+Weiss}
For any discrete countable group $\Gamma$, the highest cost $\gcostsup(\Gamma)$
is attained by non-trivial Bernoulli actions $\Gamma\acts (X_0,\mu_0)^\Gamma$
and their essentially free quotients.	
\end{thm}
Some comments are in order. Kechris \cite{Kechris:new} introduced the following notion:
for probability measure preserving actions of a fixed group $\Gamma$ say that 
$\Gamma\acts (X,\mu)$ \embf{weakly contains} $\Gamma\acts Y$ if given any finite measurable partition
$Y=\bigsqcup_{i=1}^n Y_i$, a finite set $F\subset \Gamma$ and an $\epsilon>0$,
there is a finite measurable partition $X=\bigsqcup_{i=1}^n X_i$ so that
\[
	\left|\mu(g X_i\cap X_j)-\nu(gY_i\cap Y_j)\right|<\epsilon\qquad (1\le i,j\le n,\ g\in F).
\]
The motivation for the terminology is the fact that weak containment of actions implies (but not equivalent to)
weak containment of the corresponding unitary representations: $L^2(Y)\preceq L^2(X)$.
It is clear that a quotient is (weakly) contained in the larger action.
It is also easy to see that the cost of a quotient action is no less than that of
the original (because one can lift any graphing from a quotient to the larger action maintaining
the cost of the graphing).
Kechris \cite{Kechris:new} proves that this (anti-)monotonicity still holds in 
the context of weak containment of essentially free actions of finitely generated groups, namely:
\[
	\Gamma\acts Y\preceq \Gamma\acts X\qquad\Longrightarrow\qquad
	\cost(\Gamma\acts Y)\ge \cost(\Gamma\acts X).
\] 
In fact, it follows from the more general fact that cost is upper semi-continuous in
the topology of actions.
Abert and Weiss prove that Bernoulli actions (and their quotients) are \emph{weakly contained}
in any essentially free action of a group.
Thus Theorem~\ref{T:Abert+Weiss} follows from the monotonicity of the cost.

\medskip


\subsubsection{$\ell^2$-Betti numbers} 
\label{ssub:_ell_2_betti_numbers}\hfill\\

The $\ell^2$-Betti numbers of (coverings of ) manifolds were introduced by Atiyah in \cite{Atiyah:1976}.
Cheeger and Gromov \cite{Cheeger+Gromov:1986L2} defined $\ell^2$-Betti numbers $\beta^{(2)}_i(\Gamma)\in[0,\infty]$, $i\in\bbN$,
for arbitrary countable group $\Gamma$ as dimensions (in the sense of Murray von-Neumann)
of certain homology groups (which are Hilbert $\Gamma$-modules).
For reference we suggest \cite{Eckmann:2000L2}, \cite{Luck:2002L2book}.
Here let us just point out the following facts:
\begin{enumerate}
	\item 
	If $\Gamma$ is infinite amenable, then $\beta^{(2)}_i(\Gamma)=0$, $i\in\bbN$;
	\item
	For free groups $\beta^{(2)}_1(\bbF_n)=n-1$ and $\beta^{(2)}_i(\bbF_n)=0$ for $i>1$.
	\item
	For groups with property (T), $\beta^{(2)}_1(\Gamma)=0$.
	\item\label{i:Kuneth}
	K\"uneth formula: $\beta^{(2)}_k(\Gamma_1\times\Gamma_2)=\sum_{i+j=k}\beta^{(2)}_i(\Gamma_1)\cdot \beta^{(2)}_j(\Gamma_2)$.
	\item
	Kazhdan's conjecture, proved by L\"uck, states that for residually finite groups satisfying appropriate finiteness properties
	(e.g. finite $K(pi,1)$) the $\ell^2$-Betti numbers are the stable limit of Betti numbers of finite index subgroups normalized by
	the index: $\beta^{(2)}_i(\Gamma)=\lim \frac{\beta_i(\Gamma_n)}{[\Gamma:\Gamma_n]}$ where $\Gamma>\Gamma_1>\dots$ 
	is a chain of normal subgroups of finite index.
	\item 
	The $\ell^2$ Euler characteristic $\chi^{(2)}(\Gamma)=\sum (-1)^i\cdot\beta^{(2)}_i(\Gamma)$ coincides
	with the usual Euler characteristic $\chi(\Gamma)=\sum (-1)^i\cdot\beta_i(\Gamma)$, provided both
	are defined, as is the case for fundamental group $\Gamma=\pi_1(M)$ of a compact aspherical manifold.
	\item\label{i:Atiya}
	According to the Hopf-Singer conjecture the $\ell^2$-Betti numbers for a fundamental group $\Gamma=\pi_1(M)$ 
	of a compact aspherical manifold $M$, vanish except, possibly, in the middle dimension $n$.
	Atiyah's conjecture states that $\ell^2$-Betti numbers are integers.
\end{enumerate}
The following remarkable result of Damien Gaboriau states that these intricate numeric invariants of groups are preserved
under Measure Equivalence, after a rescaling by the coupling index.
%
\begin{theorem}[Gaboriau \cite{Gaboriau:2000CRAS-L2}, \cite{Gaboriau:IHES2002L2}]
	\label{T:Gaboriau:L2}
Let $\Gamma\me\Lambda$ be ME countable groups. Then 
\[
	\beta^{(2)}_i(\Lambda)=c\cdot \beta^{(2)}_i(\Gamma)
	\qquad (i\in\bbN)
\]	
where $c=[\Gamma:\Lambda]_{\Omega}$ is a/the index of some/any $(\Gamma,\Lambda)$-coupling. 
\end{theorem}
In fact, Gaboriau introduced the notion of $\ell^2$-Betti numbers for $\II_1$-relations
and related them to $\ell^2$-Betti numbers of groups in case of the orbit relation
for an essentially free p.m.p. action -- see more comments in 
\S \ref{ssub:rel_ell_2_betti_numbers} below.

Thus the geometric information encoded in the $\ell^2$-Betti numbers for
fundamental groups of aspherical manifolds, such as Euler characteristic
and sometimes the dimension, pass through Measure Equivalence.
In particular, if lattices $\Gamma_i$ ($i=1,2$) (uniform or not)
in ${\rm SU}_{n_i,1}(\bbR)$ are ME then $n_1=n_2$;
the same applies to ${\rm Sp}_{n_i,1}(\bbR)$ and ${\rm SO}_{2n_i,1}(\bbR)$.
(The higher rank lattices are covered by stronger rigidity statements
-- see \S \ref{ssub:higher_rank_lattices} below).
Furthermore, it follows from Gaboriau's result that in general the set 
\[
	D^{(2)}(\Gamma)=\setdef{i\in\bbN}{0<\beta^{(2)}_i(\Gamma)<\infty}
\]
is a ME-invariant. Conjecture (\ref{i:Atiya}) relates this to the dimension
of a manifold $M$ in case of $\Gamma=\pi_1(M)$. 
One shouldn't expect $\dim(M)$ to be an ME-invariant of $\pi_1(M)$
as the examples of tori show; note also that for any manifold
$M$ one has $\pi_1(M\times \bbT^n)\me \pi_1(M\times\bbT^k)$.
However, among negatively curved manifolds Theorem~\ref{T:BFS-dim} 
below shows that $\dim(M)$ is invariant of $\ell^1$-ME.

For closed aspherical manifolds $M$ the dimension 
$\dim(M)$ is a QI invariant of $\pi_1(M)$.
Pansu proved that the whole set $D^{(2)}(\Gamma)$ is a QI invariant of $\Gamma$.
However, positive proportionality of $\ell^2$-Betti numbers for ME
fails under QI; in fact, there are QI groups whose Euler characteristics
have opposite signs. Yet

\begin{corollary}
For ME groups $\Gamma$ and $\Lambda$ with well defined Euler characteristic, 
say fundamental groups of compact manifolds, one has
\[
	\chi(\Lambda)=c\cdot \chi(\Gamma),\qquad  \text{where}\qquad
	c=[\Gamma:\Lambda]_\Omega \in (0,\infty). 
\]
In particular, the sign (positive, zero, negative) of the Euler characteristic is a ME-invariant.
\end{corollary}

\medskip

\subsubsection{Cowling-Haagerup $\Lambda$-invariant} 
\label{ssub:Lambda_invariant}\hfill\\

This numeric invariant $\Lambda_G$, taking values in $[1,\infty]$, is defined for any lcsc group 
$G$ in terms of norm bounds on unit approximation in the Fourier algebra $A(G)$ 
(see Cowling and Haagerup \cite{Cowling+Haagerup:1989}).
The $\Lambda$-invariant coincides for a lcsc group and its lattices. 
Moreover, Cowling and Zimmer \cite{Cowling+Zimmer:89sp} proved that 
$\Gamma_1\ore\Gamma_2$ implies $\Lambda_{\Gamma_1}=\Lambda_{\Gamma_2}$.
In fact, their proof implies the invariance under Measure Equivalence (see \cite{Jolissaint:2001}).
So $\Lambda_\Gamma$ is a ME-invariant.

Cowling and Haagerup \cite{Cowling+Haagerup:1989} computed the $\Lambda$-invariant for simple Lie groups
and their lattices:
in particular, proving that $\Lambda_G=1$ for $G\simeq {\rm SO}_{n,1}(\bbR)$ and ${\rm SU}_{n,1}(\bbR)$,
$\Lambda_G=2n-1$ for $G\simeq {\rm Sp}_{n,1}(\bbR)$, and $\Lambda_G=21$ for
the exceptional rank-one group $G=F_{4(-20)}$.   

One may observe that simple Lie groups split into two classes:
(1) ${\rm SO}_{n,1}(\bbR)$ and ${\rm SU}_{n,1}(\bbR)$ family, and (2) $G\simeq {\rm Sp}_{n,1}(\bbR)$, $F_{4(-20)}$
and higher rank. Groups in the first class have Haagerup Approximation Property (HAP, a.k.a. a-T-menability) and $\Lambda_G=1$,
while groups in the second class have Kazhdan's property (T) and $\Lambda_G>1$.
Cowling conjectured that $\Lambda_G=1$ and (HAP) might be equivalent. 
Recently one implication of this conjecture has been disproved: 
Cornulier, Stalder and Valette \cite{Cornulier+Stalder+Valette:2008HAP}
proved that the wreath product $H\wr \bbF_2$ of a finite group $H$ by the free group $\bbF_2$ 
has (HAP), while Ozawa and Popa \cite{Ozawa+Popa:2008Cartan1} prove that $\Lambda_{H\wr \bbF_2}>1$.
The question whether $\Lambda_\Gamma=1$ implies (HAP) is still open.

\medskip

One may deduce now that if $\Gamma$ is a lattice in $G\simeq {\rm Sp}_{n,1}(\bbR)$ or in $F_{4(-20)}$ 
and $\Lambda$ is a lattice in a simple Lie group $H$, then $\Gamma\me\Lambda$ iff $G\simeq H$.
Indeed, higher rank $H$ are ruled out by Zimmer's Theorem \ref{C:ME}; 
$H$ cannot be in the families ${\rm SO}_{n,1}(\bbR)$ and ${\rm SU}_{n,1}(\bbR)$ by property (T) 
or Haagerup property; and within the family of 
${\rm Sp}_{n,1}(\bbR)$ and $F_{4(-20)}$ the $\Lambda$-invariant detects $G$
($\ell^2$-Betti numbers can also be used for this purpose).


\medskip

\subsubsection{Treeability, anti-treeability, ergodic dimension} 
\label{ssub:treeability_anti_triability_ergodic_dimension}\hfill\\

In \cite{Adams:1990treeable} Scott Adams introduced the notion of \embf{treeable} equivalence relations
(see \S \ref{ssub:rel_treeability}). Following \cite{Kechris+Miller:2004book}, a group $\Gamma$ is
\begin{description}
	\item[\embf{Treeable}] if there \emph{exists} an essentially free p.m.p. $\Gamma$-action with a treeable orbit relation.
	\item[\embf{Strongly treeable}] if \emph{every} essentially free p.m.p. $\Gamma$-action gives a treeable orbit relation.
	\item[\embf{Anti-treeable}] if there \emph{are no} essentially free p.m.p. $\Gamma$-actions with a treeable 
	orbit relation.  
\end{description}
Amenable groups and free groups are strongly treeable. 
It seems to be still unknown whether there exist treeable but not strongly treeable groups,
in particular it is not clear whether surface groups (which are treeable) are strongly treeable.

The properties of being treeable or anti-treeable are ME-invariants. 
Moreover, $\Gamma$ is treeable iff $\Gamma$ is amenable (i.e. $\me \bbF_1=\bbZ$),
or is ME to either $\bbF_2$ or $\bbF_\infty$ (this fact uses Hjorth's \cite{Hjorth:06lemma}, see
\cite[Theorems 28.2 and 28.5]{Kechris+Miller:2004book}).
Groups with Kazhdan's property (T) are anti-treeable \cite{Adams+Spatzier:1990:T-tree}.
More generally, it follows from the recent work of Alvarez and Gaboriau \cite{Alvarez+Gaboriau:2008p1322} 
that a non-amenable group $\Gamma$ with $\beta^{(2)}_1(\Gamma)=0$ is anti-treeable
(in view of (\ref{e:beta2cost}) this also strengthens \cite[Corollaire VI.22]{Gaboriau:Inven2000cost}, 
where Gaboriau showed that a non-amenable $\Gamma$ with $\gcostinf(\Gamma)=1$ is anti-treeable).

\medskip

A treeing of a relation can be seen as a $\Gamma$-invariant assignment of pointed trees with $\Gamma$
as the set of vertices. 
One may view the relation \emph{acting} on this measurable family of pointed trees by moving
the marked point. 
More generally, one might define actions by relations, or measured groupoids,
on fields of simplicial complexes. 
Gaboriau defines (see \cite{Gaboriau:IHES2002L2}) the \embf{geometric dimension} of a relation $\Rel$ to be
the smallest possible dimension of such a field of \emph{contractible} simplicial complexes;
the \embf{ergodic dimension} of a group $\Gamma$ will be the minimal geometric dimension
over orbit relations $\Rel_{\Gamma\acts X}$ of all essentially free p.m.p. $\Gamma$-actions.
In this terminology $\Rel$ is treeable iff it has geometric dimension one, and a group $\Gamma$
is treeable if its ergodic dimension is one.
There is also a notion of an \embf{approximate geometric/ergodic dimension} \cite{Gaboriau:IHES2002L2}
describing the dimensions of a sequence of subrelations approximating a given orbit relation.

\begin{thm}[{Gaboriau \cite{Gaboriau:IHES2002L2}}] 
	Ergodic dimension and approximate ergodic dimension are ME-invariants.
\end{thm}

This notion can be used to obtain some information about ME of lattices in the family
of rank one groups ${\rm SO}_{n,1}(\bbR)$.
If $\Gamma_i<{\rm SO}_{n_i,1}(\bbR)$, $i=1,2$ are lattices
and $\Gamma_1\me\Gamma_2$, then Gaboriau's result on $\ell^2$-Betti numbers shows that
if one of $n_i$ is even, then $n_1=n_2$. However, for $n_i=2k_i+1$ all $\beta^{(2)}_i$ vanish.
In this case Gaboriau shows, using the above ergodic dimension, that $k_1\le k_2\le 2k_1$ or
$k_2\le k_1\le 2k_2$.


\medskip

\subsubsection{Free products} 
\label{ssub:me_free_products}\hfill\\

It was mentioned above that if $\Gamma_i\ore\Lambda$ then $*_{i\in I}\Gamma_i\ore *_{i\in I}\Lambda_i$ 
(Here $\Gamma\ore\Lambda$ means that the two groups admit an ergodic ME-coupling with index one,
equivalently admit essentially free actions which are \emph{orbit equivalent}).
To what extent does the converse hold? Namely when can one recognize the free factors on the level of Measure Equivalence?

This problem was extensively studied by Ioana, Peterson, and Popa in \cite{Ioana+Peterson+Popa:2008} where
strong rigidity results were obtained for orbit relations under certain assumptions on the actions (see \S \ref{ssub:free_decompositions}).
Here let us formulate a recent result from Alvarez and Gaboriau \cite{Alvarez+Gaboriau:2008p1322} which can be stated in purely group 
theoretic terms. 
In \cite{Alvarez+Gaboriau:2008p1322} a notion of \embf{measurably freely indecomposable} groups (MFI) is introduced,
and it is shown that this class includes all non-amenable group with $\beta_1^{(2)}=0$.
Thus infinite property (T) groups, non-amenable direct products, are examples of MFI groups.

\begin{theorem}[Alvarez-Gaboriau \cite{Alvarez+Gaboriau:2008p1322}]\label{T:Alvarez-Gaboriau:ME-rigidity}
Suppose that $*_{i=1}^n\Gamma_i\ \me \ *_{j=1}^m\Lambda_j$,
where $\{\Gamma_i\}_{i=1}^n$ and $\{\Lambda_j\}_{j=1}^m$ are two sets
of MFI groups with $\Gamma_i\not\me\Gamma_{i'}$ for $1\le i\neq i'\le n$,
and $\Lambda_{j}\not\me \Lambda{j'}$ for $1\le j\ne j'\le m$.\\
Then $n=m$ and, up to a permutation of indices, $\Gamma_i\me\Lambda_i$.
%
\end{theorem}
Another result from \cite{Alvarez+Gaboriau:2008p1322} concerning decompositions of equivalence relations 
as free products of sub-relations is discussed in \S \ref{ssub:free_decompositions}. 

Let us also mention recent works of Kida \cite{Kida:amalgamated} and Popa and Vaes \cite{Popa+Vaes:amalgamated}
which describe extremely strong rigidity properties for certain amalgamated products of various rigid groups.


\medskip

\subsubsection{The classes $\Creg$ and $\Cmix$} 
\label{ssub:classes_creg_and_cmix}\hfill\\

In  \S \ref{ssub:products_of_hyperbolic_like_groups} below we shall discuss rigidity results
obtained by Nicolas Monod and Yehuda Shalom in \cite{Monod+Shalom:OE:05} 
(see also \cite{Monod+Shalom:CRAS:03, Monod+Shalom:CO:04} 
and jointly with Mineyev \cite{Mineyev+Monod+Shalom:2004}).
These results involve \embf{second bounded cohomology} with unitary coefficients:
$H^2_b(\Gamma,\pi)$ -- a certain vector space associated to a countable group $\Gamma$
and a unitary representation $\pi:\Gamma\to U(\scr{H}_\pi)$. 
(Some background on bounded cohomology can be found in \cite[\S 3]{Monod+Shalom:OE:05} 
or \cite{Monod:2006-ICM}; for more details see \cite{Burger+Monod:GAFA:02, Monod:01bdd}).
Monod and Shalom define the class $\Creg$ of groups characterized by the property 
that 
\[
	H^2_b(\Gamma,\ell^2(\Gamma))\neq \{0\}
\]
and (potentially larger) class $\Cmix$ of groups $\Gamma$ with non-vanishing $H^2_b(\Gamma,\pi)$
for some \emph{mixing} $\Gamma$-representation $\pi$.
Known examples of groups in $\Creg\subset\Cmix$ include groups admitting "hyperbolic-like"
actions of the following types:
(see \cite{Monod+Shalom:CO:04}, \cite{Mineyev+Monod+Shalom:2004})
\begin{itemize}
	\item[(i)] non-elementary simplicial action on some simplicial tree, 
	proper on the set of edges; 
	\item[(ii)] non-elementary proper isometric action on some proper CAT(-1) space; 
	\item[(iii)] non-elementary proper isometric action on some Gromov-hyperbolic graph of bounded valency.
\end{itemize}
Hence $\Creg$ includes free groups, free products of arbitrary countable groups 
and free products amalgamated over a finite group (with the usual exceptions of order two),
fundamental groups of negatively curved manifolds, Gromov hyperbolic groups, and non-elementary subgroups of the 
above families.
Examples of groups not in $\Cmix$ include amenable groups, products of at least two infinite groups,
lattices in higher rank simple Lie groups (over any local field), irreducible lattices in products of general compactly generated
non-amenable groups (see \cite[\S 7]{Monod+Shalom:OE:05}).
\begin{thm}[Monod-Shalom \cite{Monod+Shalom:OE:05}]
$\quad$
\begin{enumerate}
	\item Membership in $\Creg$ or $\Cmix$ is a ME-invariant.
	\item 
	For direct products $\Gamma=\Gamma_1\times\cdots\times\Gamma_n$ where $\Gamma_i\in\Creg$
	are torsion free, the number of factors and their ME types are ME-invariants.
	\item
	For $\Gamma$ as above, if $\Lambda\me \Gamma$ then $\Lambda$ cannot be written as product of $m>n$
	infinite torsion free factors.
\end{enumerate}
\end{thm}


\medskip
 

\subsubsection{Dimension and simplicial volume ($\ell^1$-ME)} 
\label{ssub:dimension_and_simplicial_volume}\hfill\\

Geometric properties are hard to capture with the notion of Measure Equivalence.
The $\ell^2$-Betti numbers is an exception, but this invariant benefits from 
its Hilbert space nature.
In \cite{Bader+Furman+Sauer:2009, Bader+Furman+Sauer:2010} Uri Bader, Roman Sauer and the author consider a restricted version of
Measure Equivalence, namely $\ell^1$-ME (see \ref{ssub:_ell_p_measure_equivalence} for definition).
Being $\ell^1$-ME is an equivalence relation between finitely generated groups,
in which any two \embf{integrable} lattices in the same lcsc group are $\ell^1$-ME.
All uniform lattices are integrable, and so are all lattices in ${\rm SO}_{n,1}(\bbR)\simeq\Isom(\hypsp^n_\bbR)$
(see \ref{ssub:hyperbolic_lattices}).

\begin{theorem}[Bader-Furman-Sauer \cite{Bader+Furman+Sauer:2010}]\label{T:BFS-dim}
Let $\Gamma_i=\pi_1(M_i)$ where $M_i$ are closed manifolds which admit a Riemannian metric of negative sectional curvature.
Assume that $\Gamma_1$ and $\Gamma_2$ admit an $\ell^1$-ME-coupling $\Omega$.
Then
\[
	\dim(M_1)=\dim(M_2)\qquad\text{and}\qquad \|M_1\|=[\Gamma_2:\Gamma_1]_\Omega\cdot \|M_2\|,
\]
where $\|M_i\|$ denotes the simplicial volume of $M_i$.
\end{theorem}
The simplicial volume $\|M\|$ of a closed manifold $M$, introduced by Gromov in \cite{Gromov:1983volume},
is the norm of the image of the fundamental class under the comparison map
$H_n(M)\to H_n^{\ell^1}(M)$ into the $\ell^1$-homology, which is an $\ell^1$-completion of the usual homology.
This is a homotopy invariant of manifolds.
Manifolds carrying a Riemannian metric of negative curvature have $\|M\|>0$ (Gromov \cite{Gromov:1983volume}).



\subsection{Orbit/Measure Equivalence Rigidity} 
\label{sub:me_rigidity}\hfill\\

Let us now turn to Measure Equivalence rigidity results, i.e., classification results in the ME category.
In the introduction to this section we mentioned that 
the ME class $\ME{\bbZ}$ is precisely all infinite amenable groups.
The (distinct) classes $\ME{\bbF_{2\le n<\infty}}$ and $\ME{\bbF_\infty}$
are very rich and resist precise description.
However, much is known about more rigid families of groups.

\subsubsection{Higher rank lattices} 
\label{ssub:higher_rank_lattices}\hfill

%
\begin{theorem}[Zimmer \cite{Zimmer:cocyclesuper:80}]\label{T:ZimmerOE}
Let $G$ and $G'$ be center free simple Lie groups with ${\rm rk}_\bbR(G)\ge 2$, 
let $\Gamma<G$, $\Gamma'<G'$ be lattices and 
$\Gamma\acts (X,\mu)\ore\Gamma'\acts (X',\mu')$
be Orbit Equivalence between essentially free probability 
measure preserving actions. 
Then $G\cong G'$. Moreover the induced actions $G\acts (G\times_\Gamma X)$, 
$G'\acts (G'\times_{\Gamma'} Y)$ are isomorphic
up to a choice of the isomorphism $G\cong G'$.
\end{theorem}	
In other words ergodic (infinite) p.m.p. actions of lattices in distinct 
higher rank semi-simple Lie groups always have distinct orbit structures\footnote{ 
There is no need here to assume here that the actions are essentially free.
Stuck and Zimmer \cite{Stuck+Zimmer:1994} showed that all non-atomic ergodic p.m.p. actions of
higher rank lattices are essentially free; this is based on and generalizes the famous Factor Theorem of 
Margulis \cite{Margulis:1978factors},
see \cite{Margulis:book}.}, for example 
\[
	2\le n < m\qquad\Longrightarrow\qquad\SL_n(\bbZ)\acts\bbT^n\quad\not\ore\quad\SL_m(\bbZ)\acts\bbT^m.
\] 
This remarkable result (a contemporary of Ornstein-Weiss Theorem \ref{T:Ornstein-Weiss}) 
not only showed that the variety of orbit structures of \emph{non-amenable} groups is very rich,
but more importantly established a link between OE in Ergodic Theory and the theory of algebraic groups
and their lattices; in particular, introducing Margulis' \emph{superrigidity} phenomena
into Ergodic Theory. 
This seminal result can be considered as the birth of the subject discussed in this survey.
Let us record a ME conclusion of the above.
\begin{corollary}[Zimmer]\label{C:ME}
Let $G$, $G'$ be connected center free simple Lie groups with ${\rm rk}_\bbR(G)\ge 2$, 
$\Gamma<G$ and $\Gamma'<G'$ lattices. Then $\Gamma\me\Gamma'$ iff $G\cong G'$.
\end{corollary}

\medskip

The picture of ME classes of lattices in higher rank simple Lie groups
can be sharpened as follows.
\begin{theorem}[\cite{Furman:ME:99}]\label{T:Furman:ME}
Let $G$ be a center free simple Lie group with ${\rm rk}_\bbR(G)\ge 2$,
$\Gamma<G$ a lattice, $\Lambda$ some group Measure Equivalent to $\Gamma$.

Then $\Lambda$ is commensurable up to finite kernels to a lattice in $G$.
Moreover any ergodic $(\Gamma,\Lambda)$-coupling has a quotient 
which is either an atomic coupling (in which case $\Gamma$ and $\Lambda$ are commensurable),
or $G$, or $\Aut(G)$ with the Haar measure.
\end{theorem}
(Recall that $\Aut(G)$ contains $\Ad(G)\cong G$ as a finite index subgroup).
The main point of this result is a construction of a representation $\rho:\Lambda\to\Aut(G)$
for the \emph{unknown} group $\Lambda$ using ME to a higher rank lattice $\Gamma$. 
It uses Zimmer's cocycle superrigidity theorem 
and a construction involving 
a bi-$\Gamma$-equivariant measurable map $\Omega\times_\Lambda\check\Omega\to \Aut(G)$.
An updated version of this construction is stated in \S \ref{sub:constructing_representations}.
The by-product of this construction, is a map $\Phi:\Omega\to\Aut(G)$
satisfying
\[
	\Phi(\gamma \omega)=\gamma\, \Phi(\omega),\qquad\Phi(\lambda\omega)=\Phi(\omega)\rho(\lambda)^{-1}.
\]
It defines the above quotients (the push-forward measure $\Phi_*m$ is identified as either
atomic, or Haar measure on $G\cong\Ad(G)$ or on all of $\Aut(G)$,
using Ratner's theorem \cite{Ratner:1994:ICM}).
This additional information is useful to derive OE rigidity results (see Theorem~\ref{T:OE-rigidity-combined}).


\medskip

\subsubsection{Products of hyperbolic-like groups} 
\label{ssub:products_of_hyperbolic_like_groups}\hfill\\

The results above use in an essential way the cocycle superrigidity theorem of Zimmer,
which exploits \emph{higher rank phenomena} as in Margulis' superrigidity.
A particular situation where such phenomena take place are \emph{irreducible}
lattices in products of (semi)simple groups, starting from $\SL_2(\bbR)\times\SL_2(\bbR)$;
or cocycles over \embf{irreducible actions} of a product of $n\ge 2$ simple groups.
Here irreducibility of an action $G_1\times\cdots\times G_n\acts (X,\mu)$ 
means ergodicity of $G_i\acts (X,\mu)$ for each $1\le i\le n$\footnote{ Sometimes this can be relaxed
to ergodicity of $G_i'\acts (X,\mu)$ where $G_i'=\prod_{j\ne i}G_j$.}.
It recently became clear that higher rank phenomena occur
also for irreducible lattices in products of $n\ge 2$ of rather general lcsc groups;
and in the cocycle setting, for cocycles over irreducible actions of products of $n\ge 2$
of rather general groups (see the introduction to \cite{Monod+Shalom:OE:05}). 
This is to say that the product structure alone seems to provide sufficient \emph{"higher rank thrust"} to the situation.
The following break-through results of Nicolas Monod and Yehuda Shalom is an excellent
illustration of this fact (see \S \ref{sub:superrigidity_for_products}). 
Similar phenomena were independently discovered by Greg Hjorth and 
Alexander Kechris in \cite{Hjorth+Kechris:MAMS:05}. 
%
\begin{theorem}[{Monod-Shalom \cite[Theorem 1.16]{Monod+Shalom:OE:05}}]
	\label{T:Monod+Sahlom-ME}
Let $\Gamma=\Gamma_1\times\cdots\times\Gamma_n$ and $\Lambda=\Lambda_1\times\cdots\times\Lambda_m$
be products of torsion-free countable groups, where $\Gamma_i\in\Creg$.
Assume that $\Gamma\me\Lambda$.

Then $n\ge m$. If $n=m$ then, after a permutation of the indices, 
$\Gamma_i\me\Lambda_i$. In the latter case ($n=m$) any ergodic ME-coupling of $\Gamma\cong\Lambda$
has the trivial coupling as a quotient.
\end{theorem}
%
%
%
\begin{theorem}[{Monod-Shalom \cite{Monod+Shalom:OE:05}}]
	\label{T:Monod-Shalom-wMMIX}
Let $\Gamma=\Gamma_1\times\cdots\times\Gamma_n$ where $n\ge 2$ and $\Gamma_i$ are torsion free groups
in class $\Cmix$, and $\Gamma\acts (X,\mu)$ be an irreducible action (i.e., every $\Gamma_i\acts (X,\mu)$
is ergodic); let $\Lambda$ be a torsion free countable group and $\Lambda\acts (Y,\nu)$
be a mildly mixing action. 
If $\Gamma\acts X\soe \Lambda\acts Y$, then this SOE has index one, $\Lambda\cong\Gamma$
and the actions are isomorphic.
\end{theorem}
\begin{theorem}[{Monod-Shalom \cite{Monod+Shalom:OE:05}}]\label{T:Monod+Shalom:Am-rad}
For $i=1,2$ let $1\to A_i\to \bar\Gamma_i\to \Gamma_i\to 1$
be short exact sequence of groups with
$A_i$ amenable and $\Gamma_i$ are in $\Creg$ and are torsion free.
Then $\bar\Gamma_1\me \bar\Gamma_2$ implies $\Gamma_1\me \Gamma_2$. 
\end{theorem}
A key tool in the proofs of these results is a cocycle superrigidity theorem \ref{T:Monod-Shalom:CSR},
which involves \emph{second bounded cohomology} $H^2_b$ of groups.
In \cite{Bader+Furman:Weyl-hyplike} (see also \cite{Bader+Furman+Shaker:Weyl-circle}) 
Uri Bader and the author develop a different approach to higher rank phenomena, in particular showing 
an analogue of Monod-Shalom Theorem \ref{T:Monod-Shalom:CSR}, as stated in Theorem \ref{T:Bader+Furman:CSR}.
This result concerns a class of groups which admit \emph{convergence action} on
a compact metrizable space (i.e. a continuous action $H\acts M$ where the action 
$H\acts M^3\setminus {\rm Diag}$ on the locally compact 
space of distinct triples is proper).
Following Furstenberg \cite{Furstenberg:1967elvelop} we denote this class as $\Dyn$,
and distinguish a subclass $\Dynam$ of groups admitting convergent action $H\acts M$
with amenable stabilizers.
As a consequence of this superrigidity theorem it follows that
Theorems \ref{T:Monod+Sahlom-ME}--\ref{T:Monod+Shalom:Am-rad}
remain valid if class $\Creg$ is replaces by $\Dynam$.
Recently Hiroki Sako \cite{Sako:IMRN-classS, Sako:JFA-ME} has obtained similar results for groups in 
Ozawa's class $\mathcal{S}$ (see \cite{Ozawa:Kurosh}).

Let us point out that each of the classes $\Creg$, $\Dynam$, $\mathcal{S}$ include all Gromov hyperbolic groups 
(and many relatively hyperbolic ones), are closed under taking subgroups, and exclude direct products
of two infinite groups. These are key features of what one would like to call "hyperboli-like" group.

%


\subsubsection{Mapping Class Groups} 
\label{ssub:mapping_class_groups}\hfill\\

The following remarkable result of Yoshikata Kida concerns Mapping Class Groups of surfaces. 
Given a compact orientable surface $\Sigma_{g,p}$ of genus $g$ with $p$ boundary components
the \embf{extended mapping class group} $\Gamma(\Sigma_{g,p})^\diamond$ is the group of isotopy 
components of diffeomorphisms of $\Sigma_{g,p}$ (the \embf{mapping class group} itself
is the index two subgroup of isotopy classes of orientation preserving diffeomorphisms).
In the following assume $3g+p>0$, i.e., rule out the torus $\Sigma_{1,0}$, once punctured torus $\Sigma_{1,1}$, 
and spheres $\Sigma_{0,p}$ with $p\le 4$ punctures.


%
\begin{theorem}[Kida \cite{Kida:2006ME}]\label{T:Kida-MErigidity}
Let $\Gamma$ be a finite index subgroup in $\Gamma(\Sigma_{g,p})^\diamond$ with $3g+p-4>0$,
or in a finite product of such Mapping Class groups $\prod_{i=1}^n \Gamma(\Sigma_{g,p})^\diamond$. 
 
Then any group $\Lambda\me \Gamma$ is commensurable up to finite kernels to $\Gamma$,
and ergodic ME-coupling has a discrete $(\Gamma,\Lambda)$-coupling as a quotient.
\end{theorem}
This work (spanning \cite{Kida:2008-thesis, Kida:2006ME, Kida:2008OE}) is a real tour de force. 
Mapping Class Groups $\Gamma(\Sigma)$ are often compared to a lattice in a semi-simple Lie group $G$:
the Teichm\"{u}ller space $\mathcal{T}(\Sigma)$ is analogous to the symmetric space $G/K$,
Thurston boundary $\mathcal{PML}(\Sigma)$ analogous to Furstenberg boundary $B(G)=G/P$,
and the curve complex $C(\Sigma)$ to the spherical Tits building of $G$.
The MCG has been extensively studied as a geometric object, while Kida's work
provides a new ergodic-theoretic perspective.
For example, Kida proves that Thurston boundary $\mathcal{PML}(\Sigma)$ with the Lebesgue measure class 
is $\Gamma$-\embf{boundary} in the sense of Burger-Monod for the Mapping Class Group,
i.e., the action of the latter is amenable and doubly ergodic with unitary coefficients.
Properties of the MCG action on $\mathcal{PML}(\Sigma)$ allow Kida to characterize certain 
subrelations/subgroupoids arising in self Measure Equivalence of a MCG;
leading to the proof of a  cocycle (strong) rigidity theorem~\ref{T:Kida-coc},
which can be viewed as a groupoid version of Ivanov's rigidity theorem.
This strong rigidity theorem can be used with \S \ref{sub:constructing_representations}
to recognize arbitrary groups ME to a MCG.

Note that a Mapping Class Group behaves like a "lattice without ambient Lie group" -- all its ME-couplings
have discrete quotients.
Moreover, Kida's ME rigidity results extends to products of MCGs without any irreducibility assumptions. 
From this point of view MCGs are more ME rigid than higher rank lattices,
despite the fact that they lack many other rigidity attributes, such as property (T)
(see Andersen \cite{Andersen:MCG-no-T}). 


\medskip

{\bf Added in proof}. 
Very recently additional extremely strong ME-rigidity results were obtained in Kida \cite{Kida:amalgamated}
and Popa and Vaes \cite{Popa+Vaes:amalgamated} for certain amalgamated products of higher rank lattices, and the also 
mapping class groups. The latter paper also establishes $W^*$-rigidity.

\medskip

\subsubsection{Hyperbolic lattices and $\ell^1$-ME} 
\label{ssub:hyperbolic_lattices}\hfill\\

Measure Equivalence is motivated by the theory of lattices, with ME-couplings
generalizing the situation of groups imbedded as lattices in in the same ambient  
lcsc group. 
Thus, in the context of semi-simple groups, one wonders whether ME rigidity
results would parallel Mostow rigidity; and in particular would apply to
(lattices in) all simple groups with the usual exception of 
$\SL_2(\bbR)\simeq{\rm SO}_{2,1}(\bbR)\simeq{\rm SU}_{1,1}(\bbR)$.
The higher rank situation (at least that of simple groups) is well understood (\S \ref{ssub:higher_rank_lattices}).
In the rank one case (consisting of the families ${\rm SO}_{n,1}(\bbR)$, ${\rm SU}_{m,1}(\bbR)$, ${\rm Sp}_{k,1}(\bbR)$, 
and $F_{4(-20)}$) known ME-invariants discussed above (namely: property (T), $\ell^2$-Betti numbers, $\Lambda$-invariant, 
ergodic dimension) allow to distinguish lattices among most rank one groups.
This refers to statements of the form: if $\Gamma_i<G_i$ are lattices then $\Gamma_1\me\Gamma_2$ iff $G_1\simeq G_2$.
However ME classification such as in Theorems \ref{T:Furman:ME}, \ref{T:Monod+Sahlom-ME}, \ref{T:Kida-MErigidity}
are not known for rank one cases. 
The ingredient which is missing in the existing approach is an appropriate cocycle superrigidity theorem\footnote{
 For ${\rm Sp}_{n,1}(\bbR)$ and $F_{4(-20)}$ a cocycle superrigidity theorem was proved by
Corlette and Zimmer \cite{Corlette+Zimmer:1994} (see also Fisher and Hitchman \cite{Fisher+Hitchman:csr}), 
but these results requires boundness assumptions which
preclude them from being used for ME-cocycles. }. 

In a joint work with Uri Bader and Roman Sauer a cocycle strong rigidity theorem 
is proved for ME-cocycles for lattices in ${\rm SO}_{n,1}(\bbR)\simeq\Isom(\HH^n_\bbR)$, $n\ge 3$,
under a certain $\ell^1$-assumption (see \S \ref{ssub:_ell_p_measure_equivalence}).
It is used to obtain the following:
\begin{theorem}[Bader-Furman-Sauer \cite{Bader+Furman+Sauer:2009}]
	\label{T:BFSrigidity}
Let $\Gamma$ is a lattice in $G=\Isom( \hypsp^n)$, $n\ge 3$, and $\Lambda$ is some finitely generated
group $\ell^1$-ME to $\Gamma$ then $\Lambda$ is a lattice in $G$ modulo a finite normal subgroup.
Moreover any ergodic $(\Gamma,\Lambda)$-coupling has a quotient, which is ether discrete,
or $G=\Aut(G)$, or $G^0$ with the Haar measure. 
\end{theorem}			

\bigskip

Recently Sorin Popa has introduced a new set of ideas for studying Orbit Equivalence.
These results, rather than relying on rigidity of the acting groups alone, 
exploit rigidity aspects of \emph{groups actions} of certain type. 
We shall discuss them in \S\S \ref{sub:deformation_vs_rigidity_technique} 
and \ref{sub:local_rigidity}, \ref{ssub:cohomology_of_cocycles}.  


\medskip

\subsection{How many Orbit Structures does a given group have?} 
\label{sub:how_many_orbit_structures_does_a_given_group_have_}\hfill\\

Theorem \ref{T:Ornstein-Weiss} of Ornstein and Weiss \cite{Ornstein+Weiss:80rohlin} implies that for an infinite
amenable countable group $\Gamma$ all ergodic probability measure preserving actions $\Gamma\acts (X,\mu)$
define the same orbit structure, namely $\Rel_{amen}$.
What happens for non-amenable groups $\Gamma$? 
%
\begin{theorem}[{Epstein \cite{Epstein:2008}, after Ioana \cite{Ioana:2006contF2} 
	and Gabotiau-Lyons \cite{Gaboriau+Lyons:2007}}]\label{T:non-OE}
Any non-amenable countable group $\Gamma$ has a continuum of essentially free 
ergodic probability measure preserving actions $\Gamma\acts(X,\mu)$, no two of which
are stably Orbit Equivalent.
\end{theorem}
Let us briefly discuss the problem and its solution. 
Since ${\rm Card}(\Aut(X,\mu)^\Gamma)=\aleph=2^{\aleph_0}$ there are at most
continuum many \emph{actions} for a fixed countable group $\Gamma$.
The fact, this upper bound on the cardinality of isomorphism classes of actions
is achieved, using the corresponding fact about unitary representations and 
the Gaussian construction.
Hence one might expect at most $\aleph$-many non-OE actions for any given $\Gamma$.
OE rigidity results showed that some specific classes of  groups indeed have this 
many mutually non-OE actions; this includes: higher rank lattices \cite{Gefter+Golodets:1988:fg}, 
products of hyperbolic-like groups \cite[Theorem 1.7]{Monod+Shalom:OE:05},
and some other classes of groups \cite{Popa:2006p750, Popa:2007coc1}).
But the general question, regarding an arbitrary non-amenable $\Gamma$, remained open.

Most invariants of equivalence relations depend on the acting group rather than the action,
and thus could not be used to distinguish actions of the fixed group $\Gamma$. 
The notable exception to this meta-mathematical statement appears for non-amenable groups 
which do not have property (T). For such groups \emph{two} non-SOE actions can easily be constructed: 
(1) a \emph{strongly ergodic} action (using Schmidt's \cite{Schmidt:1981strerg}) and 
(2) an ergodic action which is \emph{not strongly ergodic} (using Connes-Weiss \cite{Connes+Weiss:1980T}).
Taking a product with an essentially free weakly mixing strongly ergodic $\Gamma$-actions 
(e.g. the Bernoulli action $(X_0,\mu_0)^\Gamma$) one makes the above two actions essentially free
and distinct.

In \cite{Hjorth:05Dye} Greg Hjorth showed that if $\Gamma$ has property (T) the set of 
isomorphism classes of orbit structures for essentially free $\Gamma$-actions
has cardinality $\aleph$, by proving that the natural map  from the
isomorphism classes of essentially free ergodic $\Gamma$-actions to 
the isomorphism classes of $\Gamma$-orbit structures is at most countable-to-one.
More precisely, the space of $\Gamma$-actions producing a fixed orbit structure
is equipped with a structure of a Polish space (separability) where any two nearby actions are
shown to be conjugate.
This is an example of proving \emph{rigidity up to countable classes} 
combining  \emph{separability} of the ambient space with a \emph{local rigidity}
phenomenon (stemming from property (T), see \S \ref{sub:local_rigidity} below). 
These ideas can be traced back to Connes \cite{Connes:80T}, Popa \cite{Popa:correspondences},
and play a central role in the most recent developments -- see \cite{Popa:2007ICM}\S 4.

The challenge now became to show that other non-amenable groups have infinitely, or even $\aleph$-many, 
non-OE essentially free ergodic actions. 
Damien Gaboriau and Sorin Popa \cite{Gaboriau+Popa:2005Fn} achieved this goal for the 
quintessential representative of a non-amenable group without property (T), namely for the free group $\bbF_2$.
Using a sophisticated \emph{rigidity vs. separability} argument they showed that within a certain rich family 
of $\bbF_2$-actions the map from isomorphism classes of actions to orbit structures is countable-to-one.
The \emph{rigidity} component of the argument was this time provided by Popa's notion of $w$-rigid actions such as
$\SL_2(\bbZ)\acts\bbT^2$, with the rigidity related to the \emph{relative property (T)}
for the semi-direct product $\SL_2(\bbZ)\ltimes \bbZ^2$ viewing $\bbZ^2$ as the Pontryagin
dual of $\bbT^2$. 

In \cite{Ioana:2006contF2} Adrian Ioana obtained a sweeping result showing that any $\Gamma$
containing a copy of $\bbF_2$ has $\aleph$-many mutually non-SOE essentially free actions.
The basic idea of the construction being to use a family of non-SOE of $\bbF_2$-actions $\bbF_2\acts X_t$ 
to construct co-induced $\Gamma$-actions $\Gamma\acts X_t^{\Gamma/{\bbF_2}}$ 
and pushing the solution of $\bbF_2$-problem to the analysis of the co-induced actions.
The class of groups containing $\bbF_2$ covers "most of" the class of non-amenable groups with few,
very hard to obtain, exceptions. 
The ultimate solution to the problem, covering all non-amenable groups,
was shortly obtained by Inessa Epstein \cite{Epstein:2008} using a result by Damien Gaboriau and Russel Lyons \cite{Gaboriau+Lyons:2007},
who proved that any non-amenable $\Gamma$ contains a $\bbF_2$ in a sort of measure-theoretical sense.
Epstein  was able to show that this sort of containment suffices to carry out an analogue of Ioana's co-induction argument 
\cite{Ioana:2006contF2} to prove Theorem~\ref{T:non-OE}.

\medskip

Furthermore, in \cite{Ioana+Kechris+Tsankov:2008subrel} Ioana, Kechris, Tsankov jointly with Epstein 
show that for any non-amenable $\Gamma$ the space of all ergodic free p.m.p. actions taken up to OE 
not only has cardinality of the continuum, but is also impossible to classify in a very strong sense.
One may also add, that most of the general results mentioned above show that within certain
families of actions the grouping into SOE-ones has countable classes, therefore 
giving only implicit families of non-SOE actions.
In \cite{Ioana:2007F2}  Ioana provided an explicit list of a continuum of mutually 
non-SOE actions of $\bbF_2$.

\newpage
\section{Measured equivalence relations} 
\label{sec:equivalence_relations}

\subsection{Basic definitions} 
\label{sub:rel_basic_definition}\hfill\\

We start with the notion of \embf{countable equivalence relations} in the Borel setting.
It consists of a  standard Borel space $(X,\mathcal{X})$ (cf. \cite{Feldman+Moore:77I} for definitions)
and a Borel subset $\Rel\subset X\times X$ which is an equivalence relation,
whose equivalence classes $\Rel[x]=\setdef{y\in X}{(x,y)\in\Rel}$ are all countable.

To construct such relations choose a countable collection $\Phi=\{\phi_i\}_{i\in I}$
of Borel bijections $\phi_i:A_i\to B_i$ between Borel subsets $A_i, B_i\in\mathcal{X}$,
$i\in I$; and let $\Rel_\Phi$ be the smallest equivalence relation including
the graphs of all $\phi_i$, $i\in I$. 
More precisely, $(x,y)\in\Rel_\Phi$ iff there exists a finite sequence $i_1,\dots, i_k\in I$
and $\epsilon_1,\dots,\epsilon_k\in\{-1,1\}$ so that
\[
	y=\phi_{i_k}^{\epsilon_k}\circ\cdots\circ\phi_{i_2}^{\epsilon_2}\circ\phi_{i_1}^{\epsilon_1}(x).
\]
We shall say that the family $\Phi$ \embf{generates} the relation $\Rel_\Phi$.
The particular case of a collection $\Phi=\{\phi_i\}$ of Borel isomorphisms
of the whole space $X$ generates a countable group $\Gamma=\langle\Phi\rangle$
and
\[
	\Rel_\Phi=\Rel_{\Gamma\acts X}=\setdef{(x,y)}{\Gamma x=\Gamma y}=\setdef{(x,\gamma.x)}{x\in X,\ \gamma\in\Gamma}. 
\]
Feldman and Moore \cite{Feldman+Moore:77I} proved that any countable Borel equivalence relation
admits a generating set whose elements are defined on all of $X$; in other words,
any equivalence relation appears as the orbit relation $\Rel_{\Gamma\acts X}$ of
a Borel action $\Gamma\acts X$ of some countable group $\Gamma$ (see \ref{ssub:FM_question}).

Given a countable Borel equivalence relation $\Rel$ the \embf{full group} $[\Rel]$ 
is defined by 
\[
	[\Rel]=\setdef{\phi\in\Aut(X,\mathcal{X})}{\forall x\in X:\ (x,\phi(x))\in\Rel}.
\]
The \embf{full pseudo-group} $[[\Rel]]$ consists of partially defined Borel isomorphisms
\[
	\psi:{\rm Dom}(\psi)\to {\rm Im}(\psi),\qquad \text{so that}\qquad 
	{\rm Graph}(\psi)=\setdef{(x,\psi(x))}{x\in{\rm Dom}(\psi)}\subset\Rel.
\]
If $\Rel$ is the orbit relation $\Rel_{\Gamma\acts X}$ of a group action $\Gamma\acts (X,\mathcal{X})$,
then any $\phi\in [\Rel]$ has the following "piece-wise $\Gamma$-structure": 
there exist countable partitions $\bigsqcup A_i=X=\bigsqcup B_i$ 
into Borel sets and elements $\gamma_i\in\Gamma$ with $\gamma_i(A_i)=B_i$
so that $\phi(x)=\gamma_i x$ for $x\in A_i$.
Elements $\psi$ of the full pseudo-group $[[\Rel_\Gamma]]$ have a similar "piece-wise $\Gamma$-structure"
with $\bigsqcup A_i={\rm Dom}(\psi)$ and $\bigsqcup B_i={\rm Im}(\psi)$.

\bigskip

Let $\Rel$ be a countable Borel equivalence relation on a standard Borel space $(X,\mathcal{X})$.
A measure $\mu$ on $(X,\mathcal{X})$ is $\Rel$-\embf{invariant} (resp. $\Rel$-\embf{quasi-invariant})
if for all $\phi\in[\Rel]$, $\phi_*\mu=\mu$ (resp. $\phi_*\mu\sim\mu$).
Note that if $\Phi=\{\phi_i:A_i\to B_i\}$ is a generating set for $\Rel$
then $\mu$ is $\Rel$-invariant iff $\mu$ is invariant 
under each $\phi_i$, i.e. $\mu(\phi_i^{-1}(E)\cap A_i)=\mu(E\cap B_i)$ for all $E\in\mathcal{X}$.
Similarly, quasi-invariance of a measure can be tested on a generating set.
The $\Rel$-\embf{saturation} of $E\in\mathcal{X}$ is $\Rel[E]=\setdef{x\in X}{\exists y\in E,\ (x,y)\in\Rel}$.
A $\Rel$ (quasi-) invariant measure $\mu$ is \embf{ergodic} if $\Rel[E]$ is either
$\mu$-null or $\mu$-conull for any $E\in\mathcal{X}$.
In this section we shall focus on countable Borel equivalence relations $\Rel$ 
on $(X,\mathcal{X})$ equipped with an ergodic, invariant, non-atomic,
probability measure $\mu$ on $(X,\mathcal{X})$. 
Such a quadruple $(X,\mathcal{X},\mu,\Rel)$ is called \embf{type} $\II_1$-relation.
These are precisely the orbit relations of ergodic measure preserving
actions of countable groups on non-atomic standard
probability measure spaces (the non-trivial implication follows from the above 
mentioned theorem of Feldman and Moore).
 
\medskip

Given a countable Borel relation $\Rel$ on $(X,\mathcal{X})$ and an $\Rel$-quasi-invariant
probability measure $\mu$, define infinite measures $\tilde\mu_L$, $\tilde\mu_R$ on $\Rel$ by
\begin{eqnarray*}
	\tilde\mu_L(E)&=&\int_X \# \setdef{y}{(x,y)\in E\cap \Rel}\,d\mu(x),\\
	\tilde\mu_R(E)&=&\int_X \# \setdef{x}{(x,y)\in E\cap \Rel}\,d\mu(y).
\end{eqnarray*}
These measures are equivalent, and coincide if
$\mu$ is $\Rel$-invariant, which is our main focus.
In this case we shall denote
\begin{equation}\label{e:tildemu}
	\tilde\mu=\tilde\mu_L=\tilde\mu_R
\end{equation}
Hereafter, saying that some property holds a.e. on $\Rel$ would refer to $\tilde\mu$-a.e.
(this makes sense even if $\mu$ is only $\Rel$-quasi-invariant).

\medskip

\begin{remark}\label{R:unique-erg}
In some situations a Borel Equivalence relation $\Rel$ on $(X,\mathcal{X})$
has only one (non-atomic) invariant probability measure.
For example, this is the case for the orbit relation of the standard action 
of a finite index subgroup\footnote{ Or just Zariski dense subgroup, see 
\cite{BFLM:2007CRAS}.} 
$\Gamma<\SL_n(\bbZ)$ on the torus $\bbT^n=\bbR^n/\bbZ^n$, or for a lattice $\Gamma$
in a simple center free Lie group $G$ acting on $H/\Lambda$, where 
$H$ is a simple Lie group, $\Lambda<H$ is  a lattice, and $\Gamma$ acts by
left translations via an embedding $j:G\to H$ with $j(G)$ having trivial centralizer
in $H$.  
In such situations one may gain understanding of the \emph{countable Borel equivalence relation}
$\Rel$ via the study of the $\II_1$-relation corresponding to the unique $\Rel$-invariant probability measure.
\end{remark}	

\medskip

As always in the \emph{measure-theoretic} setting null sets should be considered negligible. 
So an isomorphism $T$ between (complete) measure spaces $(X_i,\mathcal{X}_i,\mu_i)$, $i=1,2$,
is a Borel isomorphism between $\mu_i$-conull sets $T:X_1'\to X_2'$
with $T_*(\mu_1)=\mu_2$.
In the context of $\II_1$-relations, we declare two relations 
$(X_i,\mathcal{X}_i,\mu_i,\Rel_i)$, $i=1,2$ to be \embf{isomorphic}, if the exists
a measure space isomorphism $T:(X_1,\mu_1)\cong (X_2,\mu_2)$ so that
$T\times T:(\Rel_1,\tilde{\mu}_1)\to(\Rel_2,\tilde{\mu}_2)$ is an isomorphism.
In other words, after a restriction to conull sets, $T$ satisfies
\[
	(x,y)\in\Rel_1\qquad\Longleftrightarrow\qquad
	(T(x),T(y))\in \Rel_2.
\]
Let us also adapt the notions of the full group and the full pseudo-group
to the measure-theoretic setting, by passing to a quotient 
$\Aut(X,\mathcal{X})\to\Aut(X,\mathcal{X},\mu)$
where two Borel isomorphism $\phi$ and $\phi'$ which agree $\mu$-a.e.
are identified.
This allows us to focus on the essential measure-theoretic issues. 
The following easy, but useful Lemma illustrates the advantage of this framework.
\begin{lemma}\label{L:inn-orbits}
Let $(X,\mathcal{X},\mu,\Rel)$ be a $\II_1$-relation.
Then for $A,B\in\mathcal{X}$ one has $\mu(\phi(A)\triangle B)=0$ for some $\phi\in[\Rel]$
iff $\mu(A)=\mu(B)$.
\end{lemma}

\medskip

\subsubsection{Restriction and weak isomorphisms} 
\label{ssub:restriction_and_weak_isomorphisms}\hfill\\

Equivalence relations admit a natural operation of \embf{restriction}, 
sometimes called \embf{induction}, to a subset:
given a relation $\Rel$ on $X$ and a measurable subset $A\subset X$
the restriction $\Rel_A$ to $A$ is
\[
	\Rel_A=\Rel\cap (A\times A).
\]
In the presence of, say $\Rel$-invariant, measure $\mu$ on $(X,\mathcal{X})$
the restriction to a subset $A\subset X$ with $\mu(A)>0$ preserves the
restricted measure $\mu|_A$, defined by $\mu|_A(E)=\mu(A\cap E)$.
If $\mu$ is a probability measure, we shall denote by $\mu_A$ the normalized
restriction $\mu_A=\mu(A)^{-1}\cdot \mu|_A$.
It is easy to see that ergodicity is preserved, so a restriction of
a $\II_1$-relation $(X,\mu,\Rel)$ to a positive measure subset $A\subset X$ is a $\II_1$-relation
$(A,\mu_A,\Rel_A)$.

\begin{remark}\label{R:amp}
Note that it follows from Lemma~\ref{L:inn-orbits} that the isomorphism class
of $\Rel_A$ depends only on $\Rel$ and on the size $\mu(A)$,
so $\Rel_A$ may be denoted $\Rel^t$ where $t=\mu(A)$ is $0<t\le 1$.
One may also define $\Rel^t$ for $t>1$.
For an integer $k>1$ let $\Rel^k$ denote the product of $\Rel$
with the full relation on the finite set $\{1,\dots,k\}$, namely
the relation on $X\times\{1,\dots,k\}$
with $((x,i), (y,j))\in\Rel^k$ iff $(x,y)\in\Rel$.
So $(\Rel^k)^{1/k}\cong \Rel^1\cong\Rel$.
The definition of $\Rel^t$ can now be extended to all $0<t<\infty$
using an easily verified formula $(\Rel^t)^s\cong \Rel^{ts}$.
This construction is closely related to the notion of an \embf{amplification}
in von-Neumann algebras: the Murray von Neumann group-measure space construction
$M_\Rel$ satisfies $M_{\Rel^t}=(M_\Rel)^t$.
\end{remark}

The operation of restriction/induction allows one to relax
the notion of isomorphism of $\II_1$-relations as follows: 
\begin{defn}\label{D:weak-iso}
Two $\II_1$-relations $\Rel_1$ and $\Rel_2$ are weakly isomorphic
if $\Rel^1\cong \Rel_2^t$ for some $t\in\bbR_+^\times$.
Equivalently, if there exist positive measurable subsets $A_i\subset X_i$
with $\mu_2(A_2)=t\cdot \mu_1(A_1)$ and an isomorphism 
between the restrictions of $\Rel_i$ to $A_i$.
\end{defn}	

Observe that two ergodic probability measure-preserving actions
$\Gamma_i\acts (X_i,\mathcal{X}_i,\mu_i)$ of countable groups are 
orbit equivalent iff the corresponding orbit relations 
$\Rel_{\Gamma_i\acts X_i}$ are isomorphic.

 
\medskip

\subsection{Invariants of equivalence relations} 
\label{sub:rel-invariants}\hfill\\

Let us now discuss in some detail several qualitative and numerical
properties of $\II_1$ equivalence relations which are preserved
under isomorphisms and often preserved or rescaled by the index
under weak isomorphisms. 
We refer to such properties as \embf{invariants} of equivalence relations.
Many of these properties are motivated by properties of groups,
and often an orbit relation $\Rel_{\Gamma\acts X}$ of an essentially free
action of countable group would be a reflection of the corresponding
property of $\Gamma$.

\medskip

\subsubsection{Amenability, strong ergodicity, property (T)} 
\label{ssub:amenability_strong_ergodicity_property_t_}\hfill\\

Amenability of an equivalence relation can be defined in a number of ways.
In \cite{Zimmer:Amenable:78} Zimmer introduced the notion of \embf{amenability}
for a \embf{group action} on a space with quasi-invariant measure.
This notion plays a central role in the theory.
This definition is parallel to the fixed point characterization of amenability for groups.
For equivalence relation $\Rel$ on $(X,\mathcal{X})$ with a \emph{quasi-invariant} measure $\mu$
it reads as follows.

Let $E$ be a separable Banach space, and $c:\Rel\to \Isom(E)$ be a measurable $1$-\embf{cocycle},
i.e., a measurable (with respect to the weak topology on $E$) map, satisfying $\tilde\mu$-a.e.
\[
	c(x,z)=c(x,y)\circ c(y,z).
\]
Let $X\ni x\mapsto Q_x\subset E^*$ be a measurable family of non-empty convex compact subsets of 
the dual space $E^*$ taken with the $*$-topology, so that $c(x,y)^*(Q_x)=Q_y$.
The relation $\Rel$ is \embf{amenable} if any such family contains a measurable invariant
section, i.e., a measurable assignment $X\ni x\mapsto p(x)\in Q_x$, so that a.e.
\[
	c(x,y)^*p(x)=p(y).
\]
The (original) definition of amenability for group actions concerned general
cocycles $c:G\times X\to \Isom(E)$ rather than the ones depending only
on the orbit relation $\Rel_{\Gamma\acts X}$. The language of measured groupoids
provides a common framework for both settings (see \cite{Anantharaman-Delaroche+Renault:2000}).
 
Any non-singular action of an amenable group is amenable, because
any cocycle $c:\Gamma\times X\to\Isom(E)$ can be used to define an affine $\Gamma$-action 
on the closed convex subset of $L^\infty(X,E^*)=L^1(X,E)^*$ consisting
of all measurable sections $x\to p(x)\in Q_x$; the fixed point property of $\Gamma$ provides the desired 
$c^*$-invariant section.
The converse is not true: any (countable, or lcsc) group admits essentially free
amenable action with a quasi-invariant measure -- this is the main use of the notion of amenable actions.
However, for essentially \emph{free, probability measure preserving} actions,
amenability of the $\II_1$-relation $\Rel_{\Gamma\acts X}$ implies (hence is equivalent to)
amenability of $\Gamma$. Indeed, given an affine $\Gamma$ action $\alpha$ on a convex compact $Q\subset E^*$,
one can take $Q_x=Q$ for all $x\in X$ and set $c(gx,x)=\alpha(g)$; amenability of $\Rel_{\Gamma\acts X}$
provides an invariant section $p:X\to Q$ whose barycenter $q=\int_X p(x)\,d\mu(x)$ would
be an $\alpha(\Gamma)$-fixed point in $Q$.

\medskip

Connes, Feldman and Weiss \cite{Connes+Feldman+Weiss:81} proved that amenable relations 
are \embf{hyperfinite} in the sense that they can be viewed as an increasing union of finite subrelations; 
they also showed that such a relation can be generated by an action of $\bbZ$
(see also  \cite{Kaimanovich:1997:amen} by Kaimanovich for a nice exposition and several other
nice characterizations of amenability). It follows that there is only one amenable $\II_1$-relation,
which we denote hereafter by 
\[
	\Rel_{amen}.
\]
In \cite{Zimmer:1981:cohom} Zimmer introduced the notion of \embf{property (T)} for group actions on
measure spaces with quasi-invariant measure. The equivalence relation version can be stated as follows.
Let $\scrH$ be a separable Hilbert space and let $c:\Rel\to U(\scrH)$ be a measurable $1$-\emph{cocycle},
i.e., $c$ satisfies 
\[
	c(x,z)=c(x,y)\circ c(y,z)
\]
Then $\Rel$ has property (T) if any such cocycle for which there exists a sequence $v_n:X\to S(\scrH)$ 
of measurable maps into the unit sphere $S(\scrH)$ with 
\[
	\|v_n(y)-c(x,y)v_n\|\to 0\qquad [\tilde\mu]\text{-a.e.}
\] 
admits a measurable map $u:X\to S(\scrH)$ with $u(y)=c(x,y)u(x)$ for $\tilde\mu$-a.e. $(x,y)\in\Rel$.
For an essentially free probability measure preserving action $\Gamma\acts (X,\mu)$ 
the orbit relation $\Rel_{\Gamma\acts X}$ has property (T) if and only if the group $\Gamma$ has Kazhdan's property (T)
(in \cite{Zimmer:1981:cohom} weak mixing of the action was assumed for the "only if" implication, but
this can be removed as in \S \ref{ssub:amenability_property_t_a_t_menability} relying on 
Bekka - Valette \cite{Bekka+Valette:1993wmT}).
In \cite{Anantharaman-Delaroche:2005T} Anantharaman-Delaroche studied the notion of property (T) in the context of
general measured groupoids.

\bigskip

Let $\Rel$ be a ${\rm II}_1$-equivalence relation on $(X,\mu)$.
A sequence $\{A_n\}$ of measurable subsets of $X$ is 
\embf{asymptotically $\Rel$-invariant}, if 
$\mu(\phi(A_n)\triangle A_n)\to 0$ for every $\phi\in[\Rel]$.
This is satisfied trivially if $\mu(A_n)\cdot (1-\mu(A_n))\to0$. 
Relation $\Rel$ is \embf{strongly ergodic} if any 
asymptotically $\Rel$-invariant sequence of sets
is trivial in the above sense.
(Note that the condition of asymptotic invariance may be checked
on elements $\phi_i$ of any generating system $\Phi$ of $\Rel$).

The amenable relation $\Rel_{amen}$ is not strongly ergodic.
If an action $\Gamma\acts (X,\mu)$ has a \embf{spectral gap} (i.e.,
does not have almost invariant vectors) in the Koopman representation
on $L^2(X,\mu)\ominus \bbC$ then $\Rel_{\Gamma\acts X}$ is strongly ergodic.
Using the fact that the Koopman representation of a Bernoulli action $\Gamma\acts (X_0,\mu_0)^\Gamma$
is contained in a multiple of the regular representation $\infty\cdot \ell^2(\Gamma)$, 
Schmidt \cite{Schmidt:1980:SL2} characterized non-amenable groups by the property 
that they admit p.m.p. actions with strongly ergodic orbit relation.
If $\Rel$ is not strongly ergodic then it has an amenable relation as a non-singular quotient 
(Jones and Schmidt \cite{Jones+Schmidt:1987}). 
Connes and Weiss \cite{Connes+Weiss:1980T} showed that all p.m.p. actions of a group 
$\Gamma$ have strongly ergodic orbit relations if and \emph{only if} $\Gamma$
has Kazhdan's property (T). 
In this short elegant paper they introduced the idea of Gaussian actions 
as a way of constructing a \emph{p.m.p. action} from a given unitary \emph{representation}.

In general strong ergodicity of the orbit relation $\Rel_{\Gamma\acts X}$
does not imply a spectral gap for the action $\Gamma\acts (X,\mu)$  
(\cite{Schmidt:1980:SL2}, \cite{Hjorth+Kechris:MAMS:05}).
However, this implication does hold for generalized Bernoulli actions 
(Kechris and Tsankov \cite{Kechris+Tsankov:PAMS}), and when the action
has an ergodic centralizer (Chifan and Ioana \cite[Lemma 10]{Chifan+Ioana:Erg-subrels}).


\medskip

\subsubsection{Fundamental group - index values of self similarity} 
\label{ssub:fundamental_group}\hfill\\

The term \embf{fundamental group} of a $\II_1$-relation $\Rel$
refers to a subgroup of $\bbR^\times_+$ defined by
\[
	\scrF(\Rel)=\setdef{t\in\bbR^\times_+}{\Rel\cong \Rel^t}.
\]
Equivalently, for $\Rel$ on $(X,\mu)$, the fundamental group $\scrF(\Rel)$ 
consists of all ratios $\mu(A)/\mu(B)$
where $A, B\subset X$ are positive measure subsets with $\Rel_A\cong \Rel_B$
(here one can take one of the sets to be $X$ without loss of generality).
The notion is borrowed from similarly defined concept of the fundamental group
of a von Neumann algebra, introduced by Murray and von Neumann \cite{Murray+vonNeumann:1943:IV}:
$\scrF(M)=\setdef{t\in\bbR^\times_+}{ M^t\cong M}$.
However, the connection is not direct: even for group space construction 
$M=\Gamma\ltimes L^\infty(X)$ isomorphisms $M\cong M^t$ (or even automorphisms
of $M$) need not respect the Cartan subalgebra $L^\infty(X)$ in general. 

Since the restriction of the amenable relation $\Rel_{amen}$ to any positive measure subset 
$A\subset X$ is amenable, it follows
\[
	\scrF(\Rel_{amen})=\bbR^\times_+.
\] 
The same obviously applies to the product of any relation with an amenable one.

\medskip

On another extreme are orbit relations $\Rel_{\Gamma\acts X}$ of essentially free
ergodic action of ICC groups $\Gamma$ with property (T):
for such relations the fundamental group $\scrF(\Rel_{\Gamma\acts X})$ is at most countable
(Gefter and Golodets \cite[Corollary 1.8]{Gefter+Golodets:1988:fg}).

\medskip

Many relations have trivial fundamental group. This includes all $\II_1$
relations with a non-trivial numeric invariant which scales under restriction:
\begin{enumerate}
	\item 
	Relations with $1<\cost(\Rel)<\infty$; in particular, orbit relation $\Rel_{\Gamma\acts X}$
	for essentially free actions of $\bbF_n$, $1<n<\infty$, or surface groups.
	\item 
	Relations with some non-trivial $\ell^2$-Betti number $0<\beta^{(2)}_i(\Rel)<\infty$ for some $i\in\bbN$;
	in particular, orbit relation $\Rel_{\Gamma\acts X}$ for essentially free actions of a group $\Gamma$ 
	with $0<\beta^{(2)}_i(\Gamma)<\infty$ for some $i\in\bbN$, such as lattices in ${\rm SO}_{2n,1}(\bbR)$, 
	${\rm SU}_{m,1}(\bbR)$, ${\rm Sp}_{k,1}(\bbR)$.
\end{enumerate}
Triviality of the fundamental group often appears as a by-product of rigidity of groups
and group actions. For examples $\scrF(\Rel_{\Gamma\acts X})=\{1\}$ in the following situations:
\begin{enumerate}
	\item 
	Any (essentially free) action of a lattice $\Gamma$ in a simple Lie group of higher rank
	(\cite{Gefter+Golodets:1988:fg});
	\item 
	Any essentially free action of (finite index subgroups of products of) Mapping Class Groups 
	(\cite{Kida:2006ME});
	\item
	Actions of $\Gamma=\Gamma_1\times\dots\times\Gamma_n$, $n\ge 2$, of hyperbolic-like 
	groups $\Gamma_i$ where each of them acts ergodically (\cite{Monod+Shalom:OE:05});
	\item 
	$\Gdsc$-cocycle superrigid actions $\Gamma\acts X$ such as Bernoulli actions of groups
	with property (T) (\cite{Popa:2006RigidityI, Popa:2006RigidityII, Popa+Vaes:08Ber}). 
\end{enumerate}

\medskip

What are other possibilities for the fundamental group beyond the two extreme cases 
$\scrF(\Rel)=\bbR^\times_+$ and $\scrF(\Rel)=\{1\}$?
The most comprehensive answer (to date) to this question is contained in the following
result of S.Popa and S.Vaes (see \cite{Popa+Vaes:08Finf} for further references):
\begin{theorem}[{Popa-Vaes, \cite[Thm 1.1]{Popa+Vaes:08Finf}}]\label{T:Popa-Vaes:F+Out}
There exists a family $\mathcal{S}$ of additive subgroups of $\bbR$ which contains all 
countable groups, and (uncountable) groups of arbitrary Hausdorff dimension in $(0,1)$, so that 
for any $F\in\mathcal{S}$ and any totally disconnected locally compact unimodular group $G$
there exist uncountably many mutually non-SOE essentially free p.m.p. actions of $\bbF_\infty$
whose orbit relations $\Rel=\Rel_{{\bbF_\infty}\acts X}$ have $\scrF(\Rel)\cong \exp(F)$ and $\Out(\Rel)\cong G$.

Moreover, in these examples the Murray von Neumann group space factor $M=\Gamma\rtimes L^\infty(X)$
has $\scrF(M)\cong\scrF(\Rel)\cong \exp(F)$ and $\Out(M)\cong\Out(\Rel)\ltimes H^1(\Rel,\bbT)$,
where $H^1(\Rel,\bbT)$ is the first cohomology with coefficients in the $1$-torus.
\end{theorem}


\medskip

\subsubsection{Treeability} 
\label{ssub:rel_treeability}\hfill\\

An equivalence relation $\Rel$ is said \embf{treeable} (Adams \cite{Adams:1990treeable}) 
if it admits a generating set $\Phi=\{\phi_i\}$
so that the corresponding (non-oriented) graph on a.e. $\Rel$-class is a tree.
Basic examples of treeable relations include: $\Rel_{amen}$ viewing the amenable $\II_1$-relation 
as the orbit relation of some/any action of $\bbZ=\bbF_1$,
and more generally, $\Rel_{\bbF_n\acts X}$ where $\bbF_n\acts X$ is an essentially free action of the free group $\bbF_n$, $1\le n\le \infty$.
Any restriction of a treeable relation is treeable, and $\Rel$ is treeable iff $\Rel^t$ is.

If $\Rel_1\to\Rel_2$ is a (weak) \emph{injective relation morphism} 
and $\Rel_2$ is treeable, then so is $\Rel_1$ -- the idea is to lift a treeing graphing from $\Rel_2$ to $\Rel_1$
piece by piece. This way, one shows that if a group $\Lambda$ admits an essentially free action $\Lambda\acts Z$
with treeable $\Rel_{\Lambda\acts Z}$, and $\Gamma$ and $\Lambda$ admit (S)OE essentially free actions
$\Gamma\acts X$ and $\Lambda\acts Y$ then the $\Gamma$-action on $X\times Z$, $g:(x,z)\mapsto (gx,\alpha(g,x)z)$
via the (S)OE cocycle $\alpha:\Gamma\times X\to\Lambda$, has a treeable orbit structure $\Rel_{\Gamma\acts X\times Z}$.
Since surface groups $\Gamma=\pi_1(\Sigma_g)$, $g\ge2$, and $\bbF_2$ are lattices in ${\rm PSL}_2(\bbR)$, hence ME,
the former groups have free actions with treeable orbit relations. 
Are all orbit relations of free actions of a surface group treeable?


\medskip

\subsubsection{Cost} 
\label{ssub:rel_cost}

The notion of \embf{cost} for $\II_1$-relations corresponds to the notion of \embf{rank}
for discrete countable groups. The notion of cost was introduced 
by G. Levitt \cite{Levitt:1995cost} and extensively studied by D. Gaboriau 
\cite{Gaboriau:CRAS1998cost, Gaboriau:Inven2000cost, Gaboriau+Lyons:2007}.
\begin{defn}\label{D:cost} 
Given a generating system $\Phi=\{\phi_i:A_i\to B_i\}_{i\in\bbN}$
for a $\II_1$-equivalence relation $\Rel$ on $(X,\mu)$ the 
\embf{cost of the graphing} $\Phi$ is
\[
	\cost(\Phi)=\sum_{i} \mu(A_i)=\sum_{i} \mu(B_i)
\]
and the \embf{cost of the relation} is
\[
	\cost(\Rel)=\inf\setdef{\cost(\Phi)}{\Phi\text{ generates }\Rel}.
\]  
\end{defn}
A generating system $\Phi$ defines a graph structure on every $\Rel$-class and
$\cost(\Phi)$ is half of the average valency of this graph 
over the space $(X,\mu)$. 

The cost of a $\II_1$-relation takes values in $[1,\infty]$.
In the definition of the cost of a relation it is important that the relation is
probability measure preserving, but ergodicity is not essential.
The broader context includes relations with finite classes,
such relations can values less than one.
For instance, the orbit relation of a (non-ergodic) probability measure preserving
action of a finite group $\Gamma\acts (X,\mu)$ one gets 
\[
	\cost(\Rel_{\Gamma\acts X})=1-\frac{1}{|\Gamma|}.
\] 
If $\Rel$ is the orbit relation of some (not necessarily free) action $\Gamma\acts (X,\mu)$
then $\cost(\Rel)\le{\rm rank}(\Gamma)$, where the latter stands for the minimal number of generators
for $\Gamma$. Indeed, any generating set $\{g_1,\dots, g_k\}$ for $\Gamma$
gives a generating system $\Phi=\{\gamma_i:X\to X\}_{i=1}^k$ for $\Rel_{\Gamma\acts X}$. 
Recall that the amenable $\II_1$-relation $\Rel_{amen}$ can be generated by (any) action of $\bbZ$. 
Hence 
\[
	\cost(\Rel_{amen})=1. 
\]
The cost behaves nicely with respect to restriction:
\begin{theorem}[Gaboriau \cite{Gaboriau:Inven2000cost}]\label{T:cost-restriction}
For a $\II_1$-relation $\Rel$:
\[
	t\cdot(\cost(\Rel^t)-1)=\cost(\Rel)-1\qquad(t\in\bbR^\times_+).
\]
\end{theorem}
The following is a key tool for computations of the cost:
\begin{theorem}[Gaboriau \cite{Gaboriau:Inven2000cost}]\label{T:cost-treeable}
Let $\Rel$ be a treeable equivalence relation, and $\Phi$ be a graphing of $\Rel$ giving a tree structure
to $\Rel$-classes. Then 
\[
	\cost(\Rel)=\cost(\Phi).
\]
Conversely, for a relation $\Rel$ with $\cost(\Rel)<\infty$, if the cost is attained by some graphing $\Psi$ 
then $\Psi$ is a treeing of $\Rel$.	
\end{theorem}
	
The above result (the first part) implies that for any essentially free action $\bbF_n\acts (X,\mu)$
one has $\cost(\Rel_{\bbF_n\acts X})=n$. This allowed Gaboriau to prove the following fact,
answering a long standing question: 
\begin{cor}[Gaboriau \cite{Gaboriau:CRAS1998cost, Gaboriau:Inven2000cost}]
If essentially free probability measure preserving actions of $\bbF_n$ and $\bbF_m$
are Orbit Equivalent then $n=m$.
\end{cor}	
Note that $\bbF_n$ and $\bbF_m$ are commensurable for $2\le n,m<\infty$,
hence they have essentially free actions which are \emph{weakly isomorphic}.
The index of such weak isomorphism will necessarily be $\frac{n-1}{m-1}$, or $\frac{m-1}{n-1}$
(these free groups have P.F. - fixed price). 
It should be pointed out that one of the major open problems in the theory of
von Neumann algebras is whether it is possible for the factors $L(\bbF_n)$ and $L(\bbF_m)$ to 
be isomorphic for $n\neq m$ 
(it is known that either all $L(\bbF_n)$, $2\le n<\infty$, are isomorphic, or all distinct).

\medskip

The following powerful result of Greg Hjorth provides a link
from treeable relations back to actions of free groups:
\begin{theorem}[Hjorth \cite{Hjorth:06lemma}]
Let $\Rel$ be a treeable equivalence relation with $n=\cost(\Rel)$ in $\{1,2,\dots,\infty\}$.
Then $\Rel$ can be generated by an essentially free action of $\bbF_n$.
\end{theorem}	 
The point of the proof is to show that a relation $\Rel$ which has a treeing graphing 
with \emph{average valency} $2n$ admits a (treeing) graphing with a.e. constant
valency $2n$.

\medskip

The behavior of the cost under passing to a subrelation of finite index
is quite subtle -- the following question is still open (to the best of author's knowledge).
\begin{question}[{Gaboriau}]\label{Q:cost-finind}
Let $\Gamma'$ be a subgroup of finite index in $\Gamma$,
and $\Gamma\acts (X,\mu)$ be an essentially free p.m.p. action.
Is it true that the costs of the orbit relations of $\Gamma$ and $\Gamma'$ are related by the index $[\Gamma:\Gamma']$:
\[
	\cost(\Rel_{\Gamma'\acts X})-1 = [\Gamma:\Gamma']\cdot (\cost(\Rel_{\Gamma\acts X})-1) \quad ?
\]
\end{question}	
In general $\Gamma'$ has at most $[\Gamma:\Gamma']$-many ergodic components.
The extreme case where the number of $\Gamma'$-ergodic components is maximal: $[\Gamma:\Gamma']$,
corresponds to $\Gamma\acts (X,\mu)$ being co-induction from an ergodic $\Gamma'$-action.
In this case the above formula easily holds.
The real question lies in the other extreme where $\Gamma'$ is ergodic.

\medskip

Recall that the notion of the cost is analogous to the notion of \emph{rank}
for groups, where $\rank(\Gamma)=\inf\setdef{n\in\bbN}{\exists \text{ epimorphism } \bbF_n\to\Gamma}$.
Schreier's theorem states that for $n\in\bbN$ any subgroup $F<\bbF_n$ of finite index
$[\bbF_n:F]=k$ is itself free: $F\cong \bbF_{k(n-1)+1}$.
This implies that for any finitely generated $\Gamma$ and any finite index subgroup of $\Gamma'<\Gamma$ 
one has
\[
	\rank(\Gamma')-1\le [\Gamma:\Gamma']\cdot (\rank(\Gamma)-1)
\]
with equality in the case of free groups.
Let $\Gamma>\Gamma_1>\dots$ be a chain of subgroups of finite index.
One defines the \embf{rank gradient} (Lackenby \cite{Lackenby:2002Haken})
of the chain $\{\Gamma_n\}$ as the limit of the monotonic (!) sequence:
\[
	\srank(\Gamma,\{\Gamma_n\})=\lim_{n\to\infty}\frac{\rank(\Gamma_n)-1}{[\Gamma:\Gamma_n]}.
\]
It is an intriguing question whether (or when) is it true that $\srank(\Gamma,\{\Gamma_n\})$
depends only on $\Gamma$ and not on a particular chain of finite index subgroups. 
One should, of course, assume that the chains in question have trivial intersection,
and one might require the chains to consists of normal subgroups in the original group.
In the case of free groups $\srank$ is indeed independent of the chain.

%
In \cite{Abert+Nikolov:2007cost} Abert and Nikolov prove that the rank gradient of a chain of 
finite index subgroups of $\Gamma$ is given by the cost of a certain associated ergodic p.m.p. $\Gamma$-action.
Let us formulate a special case of this relation where the chain $\{\Gamma_n\}$ consists of \emph{normal} subgroups
$\Gamma_n$ with $\bigcap\Gamma_n=\{1\}$.
Let $K=\varprojlim \Gamma/\Gamma_n$ denote the 
profinite completion corresponding to the chain. 
The $\Gamma$-action by left translations on the compact totally disconnected group $K$
preserves the Haar measure $m_K$ and $\Gamma\acts(K,m_K)$ is a free ergodic p.m.p. action.
(let us point out in passing that this action has a spectral gap, implying strong ergodicity,
iff the chain has property $(\tau)$ introduced by Lubotzky and Zimmer \cite{Lubotzky+Zimmer:1989tau}).
\begin{theorem}[Abert-Nikolov \cite{Abert+Nikolov:2007cost}]
With the above notations:
\[
	\srank(\Gamma,\{\Gamma_n\})=\cost(\Rel_{\Gamma\acts K})-1.
\] 
\end{theorem}	
One direction ($\ge$) is easy to explain. Let $K_n$ be the closure of $\Gamma_n$ in $K$.
Then $K_n$ is an open normal subgroup of $K$ of index $m=[\Gamma:\Gamma_n]$.
Let $1=g_1,g_2,\dots,g_n\in\Gamma$ be representatives of $\Gamma_n$-cosets,
and $h_1,\dots,h_k$ generators of $\Gamma_n$ with $k=\rank(\Gamma_n)$.
Consider the graphing $\Phi=\{\phi_2,\dots,\phi_m,\psi_1,\dots,\psi_k\}$,
where $\phi_i:K_n\to g_i K_n$ are restrictions of $g_i$ ($2\le i\le m$),
and $\psi_j:K_n\to K_n$ are restrictions of $h_j$ ($1\le j\le k$).
These maps are easily seen to generate $\Rel_{\Gamma\acts K}$, with the cost of
\[
	\cost(\Phi)=k\cdot m_K(K_n)+(m-1)\cdot m_K(K_n)=\frac{k-1}{m}+1
	=1+\frac{\rank(\Gamma_n)-1}{[\Gamma:\Gamma_n]}.
\]

Abert and Nikolov observed that a positive answer to Question~\ref{Q:cost-finind}
combined with the above result shows that $\srank(\Gamma)$ is independent of
the choice of a (normal) chain, and therefore is a numeric invariant associated
to any residually finite finitely generated groups. 
Surprisingly, this turns out to be related to a problem 
in the theory of compact hyperbolic $3$-manifolds concerning 
rank versus Heegard genus \cite{Lackenby:2002Haken}  -- see 
\cite{Abert+Nikolov:2007cost} for the connection and further discussions. 

The above result has an application, independent of Question~\ref{Q:cost-finind}.
Since amenable groups have P.F. with $\gcost=1$, it follows that a finitely generated, residually finite
amenable group $\Gamma$ has sub-linear rank growth for finite index normal subgroups with trivial intersection,
i.e., $\srank(\Gamma)=0$ for any such chain.


\medskip

\subsubsection{$\ell^2$-Betti numbers} 
\label{ssub:rel_ell_2_betti_numbers}\hfill\\

We have already mentioned the $\ell^2$-Betti numbers $\beta^{(2)}_i(\Gamma)$ associated with 
a discrete group $\Gamma$ and Gaboriau's proportionality result \ref{T:Gaboriau:L2} 
for Measure Equivalence between groups. 
In fact, rather than relating the $\ell^2$-Betti numbers of groups via ME, 
in \cite{Gaboriau:IHES2002L2} Gaboriau 
\begin{itemize}
	\item[-] \emph{defines} the notion of $\ell^2$-Betti numbers 
	$\beta^{(2)}_i(\Rel)$ for a $\II_1$-equivalence relation $\Rel$;
	\item[-] proves that  $\beta^{(2)}_i(\Gamma)=\beta^{(2)}_i(\Rel_{\Gamma\acts X})$
	for essentially free ergodic action $\Gamma\acts (X,\mu)$;
	\item[-] observes that $\beta^{(2)}_i(\Rel^t)=t\cdot\beta^{(2)}_i(\Rel)$ for any 
	$\II_1$-relation.
\end{itemize}
The definition of $\beta^{(2)}_i(\Rel)$ is inspired by the definition of $\beta^{(2)}_i(\Gamma)$ 
by Cheeger and Gromov \cite{Cheeger+Gromov:1986L2}: it uses $\Rel$-action (or groupoid action) 
on pointed contractible simplicial complexes, corresponding complexes of Hilbert modules with
$\Rel$-action, and von-Neumann dimension with respect to the algebra $M_\Rel$.   

In the late 1990s Wolgang L\"uck developed an algebraic notion of dimension for arbitrary 
modules over ven Neumann algebras, in particular giving an alternative approach to $\ell^2$-Betti
numbers for groups (see \cite{Luck:2002L2book}).  
In \cite{Sauer:2005L2Betti} Roman Sauer used L\"uck's notion of dimension to define
$\ell^2$-Betti numbers of equivalence relations, and more general measured groupoids,
providing an alternative approach to Gaboriau's results. 
In \cite{Sauer+Thom:2007p196} Sauer and Thom  develop further homological tools (including a spectral sequence
for associated to strongly normal subrelations) to study $\ell^2$-Betti numbers for
groups, relations, and measured groupoids.
 

\medskip

\subsubsection{Outer automorphism group} 
\label{ssub:outer_automorphism_group}\hfill\\

Given an equivalence relation $\Rel$ on $(X,\mu)$ define the corresponding
\embf{automorphism group} as the group of self isomorphisms: 
\[
	\Aut(\Rel)=\setdef{T\in\Aut(X,\mu)}{ T\times T(\Rel)=\Rel\quad \text{(modulo null sets)}}.
\]
The subgroup $\Inn(\Rel)$ of inner automorphisms is
\[
	\Inn(\Rel)=\setdef{T\in\Aut(X,\mu)}{ (x,T(x))\in\Rel\ \text{for a.e. } x\in X}
\]
This is just the full group $[\Rel]$, but the above notation emphasizes the fact that
it is normal in $\Aut(\Rel)$ and suggest to consider the \embf{outer automorphism group} 
\[
	\Out(\Rel)=\Aut(\Rel)/\Inn(\Rel).
\]
One might think of $\Out(\Rel)$ as the group of all measurable permutations of
the $\Rel$-classes on $X$. 
Recall (Lemma~\ref{L:inn-orbits}) that $\Inn(\Rel)$ is a huge group as
it acts transitively on (classes mod null sets of) measurable subsets of any given size in $X$.
Yet the quotient $\Out(\Rel)$ might be small (even finite or trivial), and can sometimes be
explicitly computed.
\begin{remark}
As an abstract group $H=\Inn(\Rel)$ is simple, and its automorphisms come from automorphisms of $\Rel$;
in particular $\Out(H)=\Out(\Rel)$.
Moreover, Dye's reconstruction theorem states that (the isomorphism type of) $\Rel$ is 
determined by the structure of $\Inn(\Rel)$ as an abstract group 
(see \cite[\S I.4]{Kechris:new} for proofs and further facts).	
\end{remark}
Let us also note that the operation of restriction/amplification of the relation does not alter the
outer automorphism group (cf. \cite[Lemma 2.2]{Furman:Outer:05}): 
\[
	\Out(\Rel^t)\cong\Out(\Rel) \qquad (t\in\bbR^\times_+).
\]	
  
\medskip

The group $\Aut(\Rel)$ has a natural structure of a Polish group (\cite{Gefter:1996:out2},
\cite{Gefter+Golodets:1988:fg}).
First, recall that if $(Y,\nu)$ is a finite or infinite measure Lebesgue space then $\Aut(Y,\nu)$
is a Polish group with respect to the \embf{weak topology} induced from the weak (=strong) 
operator topology of the unitary group of $L^2(Y,\nu)$.
This defines a Polish topology on $\Aut(\Rel)$ when the latter is viewed
as acting on the infinite measure space $(\Rel,\tilde{\mu})$ 
\footnote{ This topology coincides with the restriction to $\Aut(\Rel)$ of the 
\embf{uniform} topology on $\Aut(X,\mu)$ given by the
metric $d(T,S)=\mu\setdef{x\in X}{T(x)\neq S(x)}$. On all of $\Aut(X,\mu)$ the uniform topology
is complete but not separable; but its restriction to $\Aut(\Rel)$ is separable.}.
However, $\Inn(\Rel)$ is not always closed in $\Aut(\Rel)$, so the topology on $\Out(\Rel)$
might be complicated. Alexander Kechris recently found the following surprising connection: 
\begin{theorem}[{Kechris \cite[Theorem 8.1]{Kechris:new}}]
If $\Out(\Rel)$	fails to be a Polish group, then $\cost(\Rel)=1$.
\end{theorem}
Now assume that $\Rel$ can be presented as the orbit relation of an essentially free
action $\Gamma\acts (X,\mu)$, so $\Aut(\Rel)$ is the group of self orbit equivalences 
of $\Gamma\acts X$. The centralizer $\Aut_\Gamma(X,\mu)$ of $\Gamma$ in $\Aut(X,\mu)$
embeds in $\Aut(\Rel)$, and if $\Gamma$ is ICC (i.e., has Infinite Conjugacy Classes) 
then the quotient map $\Aut(\Rel)\overto{\rm out}\Out(\Rel)$ is \emph{injective} on $\Aut_\Gamma(X,\mu)$
(cf. \cite[Lemma 2.6]{Gefter:1996:out2}).
So $\Out(\Rel)$ has a copy of $\Aut_\Gamma(X,\mu)$, and the latter might be very big. 
For example in the Bernoulli action $\Gamma\acts (X,\mu)=(X_0,\mu_0)^\Gamma$, 
it contains $\Aut(X_0,\mu_0)$ acting diagonally on the factors.
Yet, if $\Gamma$ has property (T) then $\Aut_\Gamma(X,\mu)\cdot \Inn(\Rel_{\Gamma\acts X})$ 
is open in the Polish group $\Aut(\Rel_{\Gamma\acts X})$. 
In this case the image of $\Aut_\Gamma(X,\mu)$ has finite or countable 
index in $\Out(\Rel_{\Gamma\acts X})$. 
This fact was observed by Gefter and Golodets in \cite[\S 2]{Gefter+Golodets:1988:fg},
and can be deduced from Proposition~\ref{P:local-rigidity}.

To get a handle on $\Out(\Rel_{\Gamma\acts X})$ one looks at OE-cocycles
$c_T:\Gamma\times X\to\Gamma$ corresponding to elements $T\in\Aut(\Rel_{\Gamma\acts X})$.
It is not difficult to see that $c_T$ is conjugate in $\Gamma$ to the identity (i.e., $c_T(g,x)=f(gx)^{-1}g f(x)$
for $f:X\to\Gamma$) iff $T$ is in $\Aut_\Gamma(X,\mu)\cdot\Inn(\Rel)$.
Thus, starting from a group $\Gamma$ or action $\Gamma\acts X$ with strong rigidity properties for
cocycles, one controls $\Out(\Rel_{\Gamma\acts X})$ via $\Aut_\Gamma(X,\mu)$.
This general scheme (somewhat implicitly) is behind the first example of an equivalence relation
with trivial $\Out(\Rel)$ constructed by Gefter \cite{Gefter:1993:out1, Gefter:1996:out2}.
Here is a variant of this construction:
\begin{thm}\label{T:Out1}
Let $\Gamma$ be a torsion free group with property (T), $K$ a compact connected Lie group without outer automorphisms, 
and $\tau:\Gamma\to K$ a dense imbedding. Let $L<K$ be a closed subgroup
and consider the ergodic actions $\Gamma\acts (K,m_K)$ and $\Gamma\acts (K/L,m_{K/L})$
by left translations. Then 
\[
	\Out(\Rel_{\Gamma\acts K})\cong K,\qquad \Out(\Rel_{\Gamma\acts K/L})\cong N_K(L)/L. 
\]
\end{thm}	
In particular, taking $K={\rm PO}_n(\bbR)$ and $L\cong {\rm PO}_{n-1}(\bbR)<K$ to be the stabilizer of a line in $\bbR^n$,
the space $K/L$ is the projective space $P^{n-1}$, and we get
\[
	\Out(\Rel_{\Gamma\acts P^{n-1}})=\{1\}
\]
for any property (T) dense subgroup $\Gamma<{\rm PO}_n(\bbR)$. Such a group $\Gamma$ exists iff $n\ge 5$,
Zimmer \cite[Theorem 7]{Zimmer:1984:Kazhdan-on-mfld}.
The preceding discussion, combined with the cocycles superrigidity theorem~\ref{T:GammaonK}
below, and an easy observation that $\Aut_\Gamma(K/L,m_{K/L})$ 
is naturally isomorphic to $N_K(L)/L$, provide a self contained sketch of the proof 
of the theorem.

In the above statement $\Out(K)$ is assumed to be trivial and $\Gamma$ to be torsion free
just to simplify the statement.
However the assumption that $K$ is connected is essential.
Indeed, the dense embedding of $\Gamma={\rm PSL}_n(\bbZ)$ in the compact profinite
group $K={\rm PSL}_n(\bbZ_p)$ where $p$ is a prime, gives
\[
	\Out(\Rel_{{\rm PSL}_n(\bbZ)\acts {\rm PSL}_n(\bbZ_p)})\cong {\rm PSL}_n(\bbQ_p)\rtimes \bbZ/2\bbZ
\]
where the $\bbZ/2$-extension is given by the transpose $g\mapsto g^{tr}$.
The inclusion $\supset$ was found in \cite{Gefter:1996:out2}, and the equality 
is proved in \cite[Theorem 1.6]{Furman:Outer:05},
where many other computations of $\Out(\Rel_{\Gamma\acts X})$ are carried out for
actions of lattices in higher rank Lie groups.

Finally, we recall that the recent preprint \cite{Popa+Vaes:08Ber} of Popa and Vaes quoted above
(Theorem~\ref{T:Popa-Vaes:F+Out}) shows that an arbitrary totally disconnected lcsc group $G$ can arise as $\Out(\Rel_{\Gamma\acts X})$
for an essentially free action of a \emph{free} group $\bbF_\infty$.


\medskip

\subsubsection{Cohomology} 
\label{subs:rel_cohomology}\hfill\\

Equivalence relations have groups of cohomology associated to them similar
to cohomology of groups. These were introduced by Singer \cite{Singer:1955}
and largely emphasized by Feldman and Moore \cite{Feldman+Moore:77II}.
Given, say a type $\II_1$, equivalence relation $\Rel$ on $(X,\mu)$ consider
\[
	\Rel^{(n)}=\setdef{(x_0,\dots,x_n)\in X^{n+1}}{(x_i,x_{i+1})\in\Rel}
\] 
equipped with the infinite Lebesgue measure $\tilde\mu^{(n)}$ defined by
\[
	\tilde\mu^{(n)}(A)=\int_X \#\setdef{(x_1,\dots,x_n)}{(x_0,\dots,x_n)\in\Rel^{(n)}}\,d\mu(x_0).
\]
Take $(\Rel^{(0)},\tilde\mu^{0})$ to be $(X,\mu)$.
Note that $(\Rel^{(1)},\mu^{(1)})$ is just $(\Rel,\tilde\mu)$ from \S \ref{sub:rel_basic_definition}.
Since $\mu$ is assumed to be $\Rel$-invariant, the above formula is invariant under permutations
of $x_0,\dots,x_n$.  

Fix a Polish Abelian group $A$ written multiplicatively (usually $A=\bbT$). 
The space $C^n(\Rel,A)$ of $n$-cochains consists of equivalence classes (modulo $\tilde\mu^{(n)}$-null sets)
of measurable maps $\Rel^{(n)}\to A$, linked by the operators $d_n:C^n(\Rel,A)\to C^{n+1}(\Rel,A)$ 
\[
	d_n(f)(x_0,\dots,x_{n+1})=\prod_{i=0}^{n+1}f(x_0,\dots,\hat{x}_i,\dots,x_0)^{(-1)^i}.
\]
Call $Z^n(\Rel)={\rm Ker} (d_n)$ the $n$-cocycles, and $B^n(\Rel)={\rm Im}(d_{n-1})$ the $n$-coboundaries;
the cohomology groups are defined by $H^n(\Rel)=Z^n(\Rel)/B^n(\Rel)$.
In degree $n=1$ the $1$-cocycles are measurable maps $c:(\Rel,\mu)\to A$ such that
\[
	c(x,y)c(y,x)=c(x,z)
\] 
and $1$-coboundaries have the form $b(x,y)=f(x)/f(y)$ for some measurable $f:X\to A$.

\medskip

If $A$ is a compact Abelian group, such as $\bbT$, then $C^{1}(\Rel,A)$ is a Polish group
(with respect to convergence in measure). Being closed subgroup in $C^{1}(\Rel,A)$, the 1-cocycles 
$Z^1(\Rel,A)$ form a Polish group. 
Schmidt \cite{Schmidt:1981strerg} showed that $B^1(\Rel,A)$ is closed in $Z^1(\Rel,A)$ iff $\Rel$
is strongly ergodic. 

There are only few cases where $H^1(\Rel,\bbT)$ were computed:
C.C. Moore \cite{Moore:1982} constructed a relation with trivial $H^1(\Rel,\bbT)$.
Gefter \cite{Gefter:1987:coh} considered $H^1(\Rel_{\Gamma\acts G},\bbT)$ for 
actions of property (T) group $\Gamma$ densely imbedded in a semi-simple Lie group $G$.
More recently Popa and Sasyk \cite{Popa+Sasyk:2007} studied $H^1(\Rel_{\Gamma\acts X},\bbT)$
for property (T) groups $\Gamma$ with Bernoulli actions $(X,\mu)=(X_0,\mu_0)^\Gamma$.
In both cases $H^1(\Rel_{\Gamma\acts X},\bbT)$ is shown to coincide with 
group of characters ${\rm Hom}(\Gamma,\bbT)$.
Higher cohomology groups remain mysterious. 

\bigskip

The fact that $A$ is Abelian is essential to the definition of $H^n(\Rel,A)$ for $n>1$.
However in degree $n=1$ one can define $H^1(\Rel,\Lambda)$ as a \emph{set} for a general 
target group $\Lambda$\footnote{\ In order to define the notion of measurability $\Lambda$ should
have a Borel structure, and better be a Polish group; often it is a discrete countable groups,
or a Lie group.}.
In fact, this notion is commonly used in this theory under the name of
\embf{measurable cocycles} (see \S \ref{sec:cocycles} and \ref{ssub:superrigidity} below).
For the definition in terms of equivalence relations 
let $Z^1(\Rel,\Lambda)$ denote the \emph{set} of all measurable maps (mod $\tilde\mu$-null sets)
\[
	c:(\Rel,\tilde\mu)\to\Lambda\qquad{s.t.}\qquad
	c(x,z)=c(x,y)c(y,z)
\]
and let $H^1(\Rel,\Lambda)=Z^1(\Rel,\Lambda)/\sim$ where the equivalence $\sim$ between
$c, c'\in Z^1(\Rel,\Lambda)$ is declared if $c(x,y)=f(x)^{-1}c'(x,y)f(y)$ for
some measurable $f:(X,\mu)\to\Lambda$.

If $\Rel=\Rel_{\Gamma\acts X}$ is the orbit relation of an essentially free
action, then $Z^1(\Rel_{\Gamma\acts X},\Lambda)$ coincides with the set
of measurable cocycles $\alpha:\Gamma\times X\to\Lambda$ by
$\alpha(g,x)=c(x,gx)$. 
Note that ${\rm Hom}(\Gamma,\Lambda)/\Lambda$ maps into $H^1(\Rel,\Lambda)$,
via $c_\pi(x,y)=\pi(g)$ for the unique $g\in\Gamma$ with $x=gy$.
The point of cocycles superrigidity theorems is to
show that under favorable conditions this map is surjective.


\medskip

\subsubsection{Free decompositions} 
\label{ssub:free_decompositions}\hfill\\

Group theoretic notions such as free products, amalgamated products, and HNN-extensions
can be defined in the context of equivalence relations -- see Gaboriau \cite[Section IV]{Gaboriau:Inven2000cost}.
For example, a $\II_1$-relation $\Rel$ is said to split as a free product of sub-relations $\{\Rel_i\}_{i\in I}$, denoted
$\Rel=*_{i\in I}\Rel_i$, if 
\begin{enumerate}
	\item 
	$\Rel$ is generated by $\{\Rel_i\}_{i\in I}$, i.e., $\Rel$ is the smallest equivalence
	relation containing the latter family;
	\item 
	Almost every chain $x=x_0,\dots,x_n=y$, where  
	$x_{j-1}\neq x_j$, $(x_{j-1},x_j)\in\Rel_{i(j)}$ and $i(j+1)\neq i(j)$, has $x\neq y$.
\end{enumerate}
If $\Srel$ is yet another subrelation, one says that $\Rel$ splits as a \embf{free product} of $\Rel_i$ \embf{amalgamated over} $\Srel$,
$\Rel=*_\Srel \Rel_i$, if in condition (2) one replaces $x_{j-1}\neq x_j$ by $(x_{j-1},x_j)\not\in\Srel$.

The obvious example of the situation above is an essentially free action of a free product of groups 
$\Gamma_3=\Gamma_1*\Gamma_2$ (resp. amalgamated product $\Gamma_5=\Gamma_1*_{\Gamma_4}\Gamma_2)$ on a
probability space $(X,\mu)$; in this case the orbit relations 
$\Rel_i=\Rel_{\Gamma_i\acts X}$ satisfy $\Rel_3=\Rel_1*\Rel_2$ (resp. $\Rel_5=\Rel_1*_{\Rel_4}\Rel_2$).

Another useful construction (see Ioana, Peterson, Popa \cite{Ioana+Peterson+Popa:2008}) is as follows. 
Given measure preserving (possibly ergodic) relations $\Rel_1$, $\Rel_2$ on
a probability space $(X,\mu)$ for $T\in\Aut(X,\mu)$ consider the relation
generated by $\Rel_1$ and $T(\Rel_2)$. It can be shown that for a residual set of $T\in\Aut(X,\mu)$
the resulting relation is a free product of $\Rel_1$ and $T(\Rel_2)$. 
A similar construction can be carried out for amalgamated products. 
Note that in contrast with the category of groups the isomorphism type of the free product is
not determined by the free factors alone.
 
\medskip

Ioana, Peterson and Popa \cite{Ioana+Peterson+Popa:2008} obtained strong rigidity results for 
free and amalgamated products of ergodic measured equivalence relations and $\II_1$-factors with 
certain rigidity properties.
Here let us describe some results obtained by Alvarez and Gaboriau \cite{Alvarez+Gaboriau:2008p1322},
which are easier to state; they may be viewed as an analogue of Bass-Serre theory in the context
of equivalence relations. 
Say that a $\II_1$-relation $\Rel$ \embf{freely indecomposable} (FI)
if $\Rel$ is not a free product of its subrelations.
A group $\Gamma$ is said to be \embf{measurably freely indecomposable} (MFI)
if all its essentially free action give freely indecomposable orbit relations.
A group may fail to be MFI even if it is freely indecomposable in the group theoretic sense
(surface groups provide an example).
Not surprisingly, groups with property (T) are MFI
(cf. Adams Spatzier \cite{Adams+Spatzier:1990:T-tree}); but more generally 
\begin{theorem}[Alvarez-Gaboriau \cite{Alvarez+Gaboriau:2008p1322}]\label{T:Alvarez-Gaboriau:MFI}
If $\Gamma$ is non-amenable and $\beta^{(2)}_1(\Gamma)=0$ then $\Gamma$ is MFI.	
\end{theorem}
\begin{theorem}[Alvarez-Gaboriau \cite{Alvarez+Gaboriau:2008p1322}]\label{T:Alvarez-Gaboriau:irr-rigidity}
Let $I$, $J$ be two finite or countable index sets, $\{\Gamma_i\}_{i\in I}$ and $\{\Lambda_j\}_{j\in J}$
be two families of MFI groups, $\Gamma=*_{i\in I}\Gamma_i$, $\Lambda=*_{j\in J}\Lambda_j$,
and $\Gamma\acts (X,\mu)$, $\Lambda\acts (Y,\nu)$ be essentially free p.m.p. actions where
each $\Gamma_i\acts (X,\mu)$ and $\Lambda_j\acts (Y,\nu)$ are ergodic. Assume that 
$\Gamma\acts X\soe\Lambda\acts Y$.

Then $|I|=|J|$ and there is a bijection $\theta:I\to J$
so that $\Gamma_i\acts X\soe \Lambda_{\theta(i)}\acts Y$.
\end{theorem}
The assumption that each free factor is ergodic is important here;
Alvarez and Gaboriau also give an analysis of the general situation (where this assumption is dropped).
%



\bigskip

\subsection{Rigidity of equivalence relations} 
\label{sub:er_rigidity}\hfill\\

The close relation between ME and SOE allows to deduce
that certain orbit relations $\Rel_{\Gamma\acts X}$ remember  
the acting group $\Gamma$ and the action $\Gamma\acts (X,\mu)$
up to isomorphism, or up to a \embf{virtual isomorphism}.
This slightly technical concept is described in the following:   
\begin{lemma}\label{L:v-iso}
Suppose an ergodic ME-coupling $(\Omega,m)$ of $\Gamma$ with $\Lambda$
corresponds to a SOE between ergodic actions $T:\Gamma\acts (X,\mu)\soe\Lambda\acts (Y,\nu)$.
Then the following are equivalent:
\begin{enumerate}
	\item 
	There exist short exact sequences 
	\[
		1\to \Gamma_0\to \Gamma\to \Gamma_1\to 1,\qquad 1\to \Lambda_0\to \Lambda\to\Lambda_1\to 1
	\] 
	where $\Gamma_0$ and $\Lambda_0$ are finite, a discrete $(\Gamma_1,\Lambda_1)$-coupling
	$(\Omega_1,m_1)$ and an equivariant map $\Phi:(\Omega,m)\to(\Omega_1,m_1)$;
	\item
	There exist isomorphism between finite index subgroups 
	\[
		\Gamma_1>\Gamma_2\ \cong\ \Lambda_2<\Lambda_1,
	\] 
	so that $\Gamma_1\acts X_1=X/\Gamma_0$ and $\Lambda_1\acts Y_1=Y/\Lambda_0$ 
	are induced from some isomorphic ergodic actions $\Gamma_2\acts X_2 \cong \Lambda_2\acts Y_2$.
	\item
	The SOE (or ME) cocycle $\Gamma\times X\to\Lambda$ is conjugate in $\Lambda$ to a cocycle
	whose restriction to some finite index subgroup $\Gamma_1$ is a homomorphism $\Gamma_1\to \Lambda$
	(the image is necessarily of finite index).
\end{enumerate}	
\end{lemma}	

Let us now state two general forms of relation rigidity. Here is one form
\begin{thm}
	\label{T:OE-rigidity-combined}
Let $\Gamma\acts (X,\mu)$ be an ergodic essentially free action of one of the types below,
$\Lambda$ an arbitrary group and $\Lambda\acts (Y,\nu)$ as essentially free p.m.p. action
whose orbit relation $\Rel_{\Lambda\acts Y}$ is weakly isomorphic to $\Rel_{\Gamma\acts X}$.

Then $\Lambda$ is commensurable up to finite kernels to $\Gamma$ and the actions $\Gamma\acts X$ 
and $\Lambda\acts Y$ are virtually isomorphic; in particular, the SEO-index is necessarily rational. 

\medskip

The list of actions $\Gamma\acts X$ with this SOE-rigidity property include:
\begin{enumerate}
	\item \label{i:higherrank}
	$\Gamma$ is a lattice in a connected, center free, simple, Lie group $G$ 
	of higher rank, and $\Gamma\acts X$ has no equivariant quotients of
	the form $\Gamma\acts G/\Gamma'$ where $\Gamma'<G$ is a lattice
	(\cite[Theorem A]{Furman:OE:99});
	\item \label{i:products}
	$\Gamma=\Gamma_1\times\cdots\times\Gamma_n$ where $n\ge 2$, $\Gamma_i\in\Creg$,
	and $\Gamma_i\acts (X,\mu)$ are ergodic; in addition assume that $\Lambda\acts (Y,\nu)$
	is mildly mixing (Monod-Shalom \cite{Monod+Shalom:OE:05});
	\item \label{i:MCG}
	$\Gamma$ is a finite index subgroup in a (product of) Mapping Class Groups as
	in Theorem~\ref{T:Kida-MErigidity} (Kida \cite{Kida:2008OE}).
\end{enumerate}
\end{thm}	

\embf{(a)} For a concrete example for (\ref{i:higherrank})--(\ref{i:MCG}) one might take Bernoulli actions
$\Gamma\acts (X_0,\mu_0)^\Gamma$. 
In (\ref{i:higherrank}) one might also take $\SL_n(\bbZ)\acts \bbT^n$ or $\SL_n(\bbZ)\acts \SL_n(\bbZ_p)$ with $n\ge 3$.
In (\ref{i:products}) one might look at $\bbF_n\times \bbF_m$ acting on a compact Lie group $K$, e.g. ${\rm SO}_3(\bbR)$,
by $(g,h):k\mapsto gkh^{-1}$ where $\bbF_n$, $\bbF_m$ are imbedded densely in $K$.

\medskip

\embf{(b)}
In (1) the assumption that there are no $\Gamma$-equivariant quotient maps $X\to G/\Gamma'$ 
is necessary, since given such a quotient there is a $\Gamma'$-action on some $(X',\mu')$
with $\Gamma'\acts X'\soe \Gamma\acts X$. The rigidity statement in this case is that
this is a complete list of groups and their essentially free actions up to virtual isomorphism
(\cite[Theorem C]{Furman:OE:99}). The appearance of these factors has to do with
$(G,m_G)$ appearing as a quotient of a $(\Gamma,\Lambda)$-coupling. 

\medskip

\embf{(c)}
The basic technique for establishing the stated rigidity in cases (\ref{i:higherrank}) -- (\ref{i:MCG})
is to establish condition (1) in Lemma~\ref{L:v-iso}. This is done by analyzing a self $\Gamma$-coupling
of the form $\Omega\times_\Lambda\check\Omega$ (where $X=\Omega/\Lambda$ and $Y=\Omega/\Gamma$) 
and invoking an analogue of the construction in
\S \ref{sub:constructing_representations}.

\medskip

\embf{(d)}
In all cases one can sharpen the results (eliminate the "virtual") by imposing some benign additional
assumptions: rule out torsion in the acting groups, and impose ergodicity for actions of finite index subgroups.

\bigskip

The second \emph{stronger} form of relation rigidity refers to 
\embf{rigidity of relation morphisms} which are obtained from
$\Gdsc$-\embf{cocycle superrigid actions} 
discovered by Sorin Popa (see \S \ref{sub:deformation_vs_rigidity_technique}). 
We illustrate this framework by the following particular statement
(see \cite[Theorem 0.4]{Popa:2007coc1} and \cite[Theorem 1.8]{Furman:2007popa}).
\begin{theorem}\label{T:cons-of-csr}
Let $\Gamma\acts (X,\mu)$ be a mixing $\Gdsc$-cocycle superrigid action, such as:
\begin{enumerate}
	\item A Bernoulli $\Gamma$-action on $(X_0,\mu_0)^\Gamma$, 
	where $\Gamma$ has property (T), or $\Gamma=\Gamma_1\times\Gamma_2$ with $\Gamma_1$ non-amenable
	and $\Gamma_2$ being infinite;
	\item $\Gamma\acts K/L$ where $\Gamma\to K$ is a homomorphism with dense image
	in a simple compact Lie group $K$ with trivial $\pi_1(K)$, $L<K$ is a closed subgroup, and $\Gamma$ has (T). 
\end{enumerate}
Let $\Lambda$ be some group with an ergodic essentially free measure preserving action $\Lambda\acts (Y,\nu)$\footnote{ The 
space $(Y,\nu)$ might be finite or infinite Lebesgue measure space.}, $X'\subset X$ a positive measure subset
and $T:X'\to Y$ a measurable map with $T_*\mu\prec \nu$ and
\[
	(x_1,x_2)\in \Rel_{\Gamma\acts X}\cap (X'\times X')\qquad\Longrightarrow\qquad
	(T(x_1),T(x_2))\in \Rel_{\Lambda\acts Y}.
\]
Then there exist
\begin{itemize}
	\item 
	an exact sequence  $\Gamma_0\overto{}\Gamma\overto{\rho}\Lambda_1$ 
	with finite $\Gamma_0$ and $\Lambda_1<\Lambda$;
	\item
	A $\Lambda_1$-ergodic subset $Y_1\subset Y$ with $0<\nu(Y_1)<\infty$;
	\item
	Denoting $(X_1,\mu_1)=(X,\mu)/\Gamma_0$ and $\nu_1=\nu(Y_1)^{-1}\cdot \nu|_{Y_1}$, there is an isomorphism 
	$T_1:(X_1,\mu_1)\cong (Y_1,\nu_1)$ 
	of $\Lambda_1$-actions.
\end{itemize}
Moreover, $\mu$-a.e. $T(x)$ and $T_1(\Gamma_0x)$ are in the same $\Lambda$-orbit.
\end{theorem}
   

\medskip

\subsubsection{A question of Feldman and Moore} 
\label{ssub:FM_question}\hfill\\

Feldman and Moore showed \cite{Feldman+Moore:77I} that any countable Borel equivalence relation 
$\Rel$ can be generated by a Borel action of a countable group. 
They asked whether one can find a \emph{free} action of some group, so that $\Rel$-classes would be in one-to-one 
correspondence with the acting group. 
This question was answered in the negative by Adams \cite{Adams:88FM}.
In the context of measured relations, say of type $\II_1$, the question is whether it is possible 
to generate $\Rel$ (up to null sets) by an \emph{essentially free} action of some group.
This question was also settled in the negative in \cite[Theorem D]{Furman:OE:99}, using the following basic constructions: 
\begin{enumerate}
	\item 
	Start with an essentially free action $\Gamma\acts (X,\mu)$ which is rigid
	as in Theorem~\ref{T:OE-rigidity-combined} or \ref{T:cons-of-csr},
	and let $\Rel=(\Rel_{\Gamma\acts X})^t$ with an \emph{irrational} $t$. 
	\item
	Consider a proper imbedding $G\hookrightarrow H$ of higher rank simple Lie groups choose a lattice
	$\Gamma<H$, say $G=\SL_3(\bbR)\subset H=\SL_4(\bbR)$ with $\Gamma=\SL_4(\bbZ)$.
	Such actions always admit a Borel cross-section $X\subset H/\Gamma$ for the $G$-action, 
	equipped with a holonomy invariant probability measure $\mu$.
	Take $\Rel$ on $(X,\mu)$ to be the relation of being in the same $G$-orbit.
\end{enumerate}
In case (1) one argues as follows: if some group $\Lambda$ has an essentially free action $\Lambda\acts (Y,\nu)$
with $(\Rel_{\Gamma\acts X})^t=\Rel\cong\Rel_{\Lambda\acts Y}$ then the rigidity implies 
that $\Gamma$ and $\Lambda$ are commensurable up to finite kernel, and $\Gamma\acts X$ is virtually 
isomorphic to $\Lambda\acts Y$. 
But this would imply that the index $t$ is rational, contrary to the assumption.
This strategy can be carried out in other cases of very rigid actions as in 
\cite{Monod+Shalom:OE:05, Popa:2007coc1, Kida:2008OE, Ioana:2008profinite}. 
Theorem~\ref{T:GammaonK} provides an example of this type $\Rel=(\Rel_{\Gamma\acts K})^t$
where $\Gamma$ is a Kazhdan group densely imbedded in a compact connected Lie group $K$.
So the reader has a sketch of the full proof for a $\II_1$-relation which cannot be generated 
by an essentially free action of any group.

Example of type (2) was introduced by Zimmer in \cite{Zimmer:1991:trans}, where it was proved that
the relation $\Rel$ on such a cross-section cannot be essentially freely generated by a group $\Lambda$
which admits a linear representation with an infinite image. The linearity assumption was removed
in \cite{Furman:OE:99}. This example is particularly interesting since it cannot be "repaired"
by restriction/amplification; as any $\Rel^t$ can be realized as a cross-section of the same $G$-flow
on $H/\Gamma$.

\begin{ques*}[Vershik]
Let $\Rel$ on $(X,\mu)$ be a $\II_1$-relation which cannot be generated by an essentially 
free action of a group; and let $\Gamma\acts (X,\mu)$ be some action producing $\Rel$.
One may assume that the action is faithful, i.e., $\Gamma\to\Aut(X,\mu)$ is an embedding.
What can be said about $\Gamma$ and the structure of the measurable family $\{\Gamma_x\}_{x\in X}$
of the stabilizers of points in $X$? 
\end{ques*}
In \cite{Popa+Vaes:08Rn} Sorin Popa and Stefaan Vaes give an example of a $\II_1$-relation 
$\Rel$ (which is a restriction of the $\II_\infty$-relation $\Rel_{\SL_5(\bbZ)\acts \bbR^5}$ to
a subset $A\subset \bbR^5$ of positive finite measure)
which has property (T) but cannot be generated by an action (not necessarily free)
of \emph{any} group with property (T).
 

\newpage
\section{Techniques} 
\label{sec:techniques}

\subsection{Superrigidity in semi-simple Lie groups} 
\label{ssub:superrigidity}\hfill\\

The term \emph{superrigidity} refers to a number of phenomena originated 
and inspired by the following celebrated discovery of G. A. Margulis.
\begin{theorem}[Margulis \cite{Margulis:1974:ICM}]\label{T:Margulis-sr}
Let $G$ and $G'$ be (semi-)simple connected real center free Lie groups without compact factors
with ${\rm rk}(G)\ge 2$, $\Gamma<G$ be an irreducible lattice and $\pi:\Gamma\to G'$ a homomorphism
with $\pi(\Gamma)$ being Zariski dense in $G'$ and not precompact.
Then $\pi$ extends to a (rational) epimorphism $\bar\pi:G\to G'$.
\end{theorem}
The actual result is more general than stated, as it applies to products of 
semi-simple algebraic groups over general local fields.
We refer the reader to the comprehensive  monograph (Margulis \cite{Margulis:book}) for the general statements,
proofs and further results and applications, including the famous Arithmeticity Theorem.


The core of (some of the available) proofs of Margulis' superrigidity Theorem
is a combination of the theory of algebraic groups and purely ergodic-theoretic arguments.
The result applies to uniform and non-uniform lattices alike,
it also covers irreducible lattices in higher rank Lie groups, such as $\SL_2(\bbR)\times\SL_2(\bbR)$.
Let us also note that the assumption that $\pi(\Gamma)$ is not precompact in $G'$ is redundant if 
$\pi(\Gamma)$ is Zariski dense in a real Lie group $G'$ (since compact groups over $\bbR$ are algebraic), 
but is important in general (cf. $\SL_n(\bbZ)< \SL_n(\bbQ_p)$ is Zariski dense
but precompact).

\medskip

In \cite{Zimmer:cocyclesuper:80} R. J. Zimmer has obtained a far reaching generalization 
of Margulis' superrigidity, passing from the context of representations of lattices
to the framework of measurable cocycles over probability measure preserving actions 
(representations of "virtual subgroups" in Mackey's terminology). 
The connection can be briefly summarized as follows:   
given a transitive action $G\acts X=G/\Gamma$ and some topological group
$H$; there is a bijection between measurable cocycles $G\times G/\Gamma\to H$ modulo
cocycle conjugation and homomorphisms $\Gamma\to H$ modulo conjugation in $H$
\[
	H^1(G\acts G/\Gamma, H)\qquad\cong\qquad {\rm Hom}(\Gamma,H)/H
\]
(see \S \ref{sub:canonical_class_of_a_lattice_induction}, and \cite{Zimmer:book:84, Furstenberg:afterMargulis-Zimmer}).
In this correspondence, a representation $\pi:\Gamma\to H$ extends to a homomorphism $G\to H$
iff the corresponding cocycle $\pi\circ c:G\times G/\Gamma\to H$ is conjugate to a homomorphism
$G\to H$.
Zimmer's Cocycle Superrigidity Theorem states that under appropriate non-degeneracy assumptions 
a measurable cocycle over an \emph{arbitrary} p.m.p. ergodic action $G\acts (X,\mu)$ is conjugate
to a homomorphism.
\begin{theorem}[Zimmer \cite{Zimmer:cocyclesuper:80}, see also \cite{Zimmer:book:84}]
	\label{T:Zimmer-csr}
Let $G$, $G'$ be semi-simple Lie group as in Theorem~\ref{T:Margulis-sr}, in particular ${\rm rk}_\bbR(G)\ge 2$,
let $G\acts (X,\mu)$ be an irreducible probability measure preserving action
and $c:G\times X\to G'$ be a measurable cocycle which is Zariski dense and not compact.
Then there exist a (rational) epimorphism $\pi:G\to G'$ and a measurable $f:X\to G'$
so that $c(g,x)=f(gx)^{-1}\pi(g)f(x)$.
\end{theorem}
In the above statement \embf{irreducibility} of $G\acts(X,\mu)$ means 
mere ergodicity if $G$ is a simple group,
and ergodicity of the action $G_i\acts (X,\mu)$ for each factor $G_i$ in the case
of a semi-simple group $G=\prod_{i=1}^n G_i$ with $n\ge 2$ factors.
For a lattice $\Gamma<G=\prod G_i$ in a semi-simple group
the transitive action $G\acts G/\Gamma$ is irreducible precisely iff $\Gamma$ is an irreducible
lattice in $G$.
The notions of being \embf{Zariski dense} (resp. \embf{not compact}) for a cocycle $c:G\times X\to H$
mean that $c$ is not conjugate to a cocycle $c^f$ taking values in a proper algebraic 
(resp. compact) subgroup of $H$.

The setting of cocycles over p.m.p. actions adds a great deal of generality to the 
superrigidity phenomena. 
First illustration of this is the fact that once cocycle superrigidity is known
for actions of $G$ it passes to actions of lattices in $G$:
given an action $\Gamma\acts (X,\mu)$ of a lattice $\Gamma<G$ one obtains
a $G$-action on $\bar{X}=G\times_\Gamma X$ by acting on the first coordinate
(just like the composition operation of ME-coupling \S \ref{sub:groups}).
A cocycle $c:\Gamma\times X\to H$ has a natural lift to $\bar{c}:G\times \bar{X}\to H$
and its cohomology is directly related to that of the original cocycle.
So cocycle superrigidity theorems have an almost automatic bootstrap from lcsc groups to their lattices.
The induced action $G\acts \bar{X}$ is ergodic iff $\Gamma\acts X$ is ergodic;
however irreducibility is more subtle. Yet, if $\Gamma\acts (X,\mu)$ is mixing
then $G\acts \bar{X}$ is mixing and therefore is irreducible.

Theorem~\ref{T:ZimmerOE} was the first application of Zimmer's cocycle superrigidity \ref{T:Zimmer-csr} 
(see \cite{Zimmer:cocyclesuper:80}).
Indeed, if $\alpha:\Gamma\times X\to \Gamma'$ is the rearrangement cocycle associated to an Orbit Equivalence
$T:\Gamma\acts(X,\mu)\ore \Gamma'\acts (X',\mu')$ where $\Gamma<G$, $\Gamma'<G'$ are lattices,
then, viewing $\alpha$ as taking values in $G'$, Zimmer observes that $\alpha$ is Zariski dense
using a form of Borel's density theorem and deduces that $G\cong G'$ (here for simplicity 
the ambient groups are assumed to be simple, connected, center-free and  ${\rm rk}_\bbR(G)\ge2$).
Moreover there is a homomorphism $\pi:\Gamma\to G'$ and $f:X\to G'$ so that 
$\alpha(\gamma,x)=f(\gamma x)\pi(\gamma)f(x)^{-1}$ with $\pi:\Gamma\to \pi(\Gamma)<G'$ being isomorphism
of lattices. 

\begin{remark}\label{R:dichotomy}
At this point it is not clear whether $\pi(\Gamma)$ should be (conjugate to) $\Gamma'$, and 
even assuming $\pi(\Gamma)=\Gamma'$ whether $f$ takes values in $\Gamma'$. 
In fact, the self orbit equivalence of the $\Gamma$ action on $G/\Gamma$ given by
$g\Gamma\mapsto g^{-1}\Gamma$ gives a rearrangement cocycle
$c:\Gamma\times G/\Gamma\to\Gamma$ which is conjugate to the identity $\Gamma\to\Gamma$
by a unique map $f:G/\Gamma\to G$ with $f_*(m_{G/\Gamma})\prec m_G$.
However, if $\pi(\Gamma)=\Gamma'$ and $f$ takes values in $\Gamma'$ it follows
that the original actions $\Gamma\acts (X,\mu)$ and $\Gamma'\acts (X',\mu')$ are isomorphic
via the identification $\pi:\Gamma\cong\Gamma'$. We return to this point below.
\end{remark}

\medskip

\subsubsection{Superrigidity and ME-couplings} 
\label{ssub:Superrigidity_and_ME_couplings}\hfill\\

Zimmer's cocycle superrigidity theorem applied to OE or ME-cocycles 
(see \S \ref{sub:oe_cocycles}, \ref{sub:me_cocycles}) has a natural interpretation in terms of ME-couplings.
Let $G$ be a higher rank simple Lie group (hereafter implicitly, connected, and center free),
denote by $i:G\to\Aut(G)$  the adjoint homomorphism (which is an embedding since $G$ is center free).
\begin{theorem}[{\cite[Theorem 4.1]{Furman:ME:99}}]\label{T:canonical-quotients}
Let $G$ be a higher rank simple Lie group, $\Gamma_1,\Gamma_2<G$ lattices,
and $(\Omega,m)$ an ergodic $(\Gamma_1,\Gamma_2)$-coupling. 
Then there exists a unique measurable map $\Phi:\Omega\to \Aut(G)$ 
so that $m$-a.e. on $\Omega$
\[
	\Phi(\gamma_1 \omega)=i(\gamma_1) \Phi(\omega),\qquad \Phi(\gamma_2\omega)=\Phi(\omega)i(\gamma_2)^{-1}
	\qquad (\gamma_i\in\Gamma_i).
\]
Moreover, $\Phi_*m$ is either the Haar measure on a group $G\cong \Ad(G)\le G'\le \Aut(G)$, or is atomic
in which case $\Gamma_1$ and $\Gamma_2$ are commensurable.
\end{theorem}
\begin{proof}[Sketch of the proof]
To construct such a $\Phi$, choose a fundamental domain $X\subset \Omega$ for the $\Gamma_2$-action
and look at the ME-cocycle $c:\Gamma_1\times X\to\Gamma_2<G$.
Apply Zimmer's cocycle superrigidity theorem to find $\pi:\Gamma_1\to G$ and $\phi:X\to G$.
Viewing $G$ as a subgroup in $\Aut(G)$, one may adjust $\pi$ and $\phi:X\to\Aut(G)$ by some $\alpha\in\Aut(G)$,
so that $\pi$ is the isomorphism $i:\Gamma_1\to\Gamma_1$, to get
\[
	c(\gamma_1,x)=\phi(\gamma_1.x)^{-1}i(\gamma_1) \phi(x).
\]
Define $\Phi:\Omega\to\Aut(G)$ by $\Phi(\gamma_2 x)=\phi(x)i(\gamma_2)^{-1}$ and check that it satisfies
the required relation. 
To identify the measure $\Phi_*m$ on $\Aut(G)$ one uses Ratner's theorem, which provides the classification of 
$\Gamma_1$-ergodic \emph{finite} measures on $\bar{G}/\Gamma_2$.
\end{proof}

Theorem~\ref{T:Furman:ME} is then proved using this fact with $\Gamma_1=\Gamma_2$ plugged into
the construction in \ref{sub:constructing_representations} which describes an unknown group $\Lambda$
essentially as a lattice in $G$. 

Note that there are two distinct cases in Theorem~\ref{T:canonical-quotients}: 
either $\Phi_*m$ is atomic, in which case
$(\Omega,m)$ has a discrete ME-coupling as a quotient, or $\Phi_*m$ is a Haar measure
on a Lie group.
The former case leads to a virtual isomorphism between the groups and the actions 
(this is case (1) in Theorem~\ref{T:OE-rigidity-combined}); in the latter
$\Gamma_1\acts X\cong\Omega/\Gamma_2$ has a quotient of the form $\Gamma_1\acts \bar{G}/\Gamma_2$
(which is \cite[Theorem C]{Furman:OE:99}).
This dichotomy clarifies the situation in Remark~\ref{R:dichotomy} above.


\bigskip

\subsection{Superrigidity for product groups} 
\label{sub:superrigidity_for_products}\hfill\\

Let us now turn to a brief discussion of Monod-Shalom rigidity (see \S\S \ref{ssub:classes_creg_and_cmix}, \ref{ssub:products_of_hyperbolic_like_groups}). 
Consider a special case of Margulis-Zimmer superrigidity results where the target group $G'$ has rank one 
(say $G'={\rm PSL}_2(\bbR)$),
while $G$ has higher rank. 
The conclusion of the superrigidity Theorems \ref{T:Margulis-sr} (resp. \ref{T:Zimmer-csr})
is that either a representation (resp. cocycle) is degenerate\footnote{ Here precompact, or contain in a parabilic}, 
or there is an epimorphism $\pi:G\to G'$.
The latter case occurs \emph{if and only if} $G$ is semi-simple $G=\prod G_i$, 
with one of the factors $G_i\simeq G'$,
and $\pi:G\to G'$ factoring through the projection $\pi:G\overto{{\rm pr}_i} G_i\simeq G'$. 
In this case the given representation of the lattice extends to $\pi$ (resp. the cocycle 
is conjugate to the epimorphism $\pi$).

This special case of Margulis-Zimmer superrigidity, i.e., from higher rank $G$ to rank one $G'$,
was generalized by a number of authors 
\cite{Burger+Mozes:CAT-1, Adams:1988tree, Adams+Spatzier:1990:T-tree}
replacing the assumption that the target group $G'$ has rank one, by
more geometric notions, such as $G'=\Isom(X)$ where $X$ is a proper CAT(-1) space.
In the setting considered by Monod and Shalom the target group is "hyperbilic-like"
in a very general way, while the source group $G$ rather than being higher rank semi-simple Lie group,
is just a product $G=G_1\times\cdots\times G_n$ of $n\ge 2$ \emph{arbitrary} 
compactly generated (in fact, just lcsc) groups.
The philosophy is that the number $n\ge 2$ of direct factors 
provides enough \emph{higher rank} properties for such statements.
\begin{theorem}[Monod - Shalom \cite{Monod+Shalom:CO:04}]\label{T:Monod-Shalom:CSR}
Let $G=G_1\times\cdots\times G_n$ be a product of $n\ge 2$
lcsc groups, $G\acts (X,\mu)$ an irreducible p.m.p. action, $H$ is hyperbolic-like group, 
and $c:G\times X\to H$ is a non-elementary measurable cocycle.

Then there is a non-elementary closed subgroup $H_1<H$, a compact normal subgroup $K\triangleleft H_1$,
a measurable $f:X\to H$, and a homomorphism $\rho:G_i\to H_1/K$ from one of the factors $G_i$ of $G$,
so that the conjugate cocycle $c^f$ takes values in $H_1$, and $G\times X\to H_1\to H_1/K$
is the homomorphism $\pi:G\overto{pr_i} G_i\overto{\rho} H_1/K$.
\end{theorem}
This beautiful theorem is proved using the technology of \emph{second bounded cohomology}
(developed in \cite{Burger+Monod:JEMS:99, Burger+Monod:GAFA:02, Monod:01bdd} and applied
in this setting in \cite{Monod+Shalom:CO:04, Mineyev+Monod+Shalom:2004}) with the 
notions of \emph{hyperbolic-like} and \emph{non-elementary} interpreted in the context 
of the class $\Creg$.

\medskip

Suppose $\Gamma=\Gamma_1\times\cdots\times\Gamma_n$, $n\ge 2$, is a product of
"hyperbolic-like" groups. Let $(\Omega,m)$ be a self ME-coupling of $\Gamma$.
Consider a ME-cocycle $\Gamma\times X\to\Gamma$ which can be viewed as a combination
of $n$ cocycles 
\[
	c_i:\Gamma\times X\overto{c}\Gamma\overto{pr_i}\Gamma_i\qquad (i=1,\dots,n)
\]
and assume that $\Gamma\acts \Omega/\Gamma$ is an \emph{irreducible} action.
Viewing the source group $\Gamma$ as a product of $n\ge 2$ factors acting irreducibly,
and recalling that the target groups $\Gamma_i$ are "hyperbolic-like" Monod and Shalom apply
Theorem~\ref{T:Monod-Shalom:CSR}.
The cocycles arising from ME coupling turn out to be non-elementary, leading to the 
conclusion that each cocycle $c_i$ is conjugate to a homomorphism $\rho_i:\Gamma_{j(i)}\to \Gamma_i'$,
modulo some reductions and finite kernels,.
Since $\Gamma_i$ commute, the conjugations can be performed independently and \emph{simultaneously} 
on all the cocycles $c_i$.
After some intricate analysis of the map $i\to j(i)$,
kernels and co-kernels of $\rho_i$, Monod and Shalom show that in the setting of ME couplings 
as above the map $i\to j(i)$ is a permutation
and $\rho_i$ are isomorphisms.
Thus the original cocycle $c$ can be conjugate to an \emph{automorphism} of $\Gamma$. 

This ME-cocycle superrigidity can now be plugged into an analogue of Theorem~\ref{T:canonical-quotients}
to give a measurable bi-$\Gamma$-equivariant map $\Omega\to\Gamma$, 
which can be used as an input to a construction like Theorem~\ref{T:reps}.
This allows to identify \emph{unknown} groups $\Lambda$ Measure
Equivalent to $\Gamma=\Gamma_1\times\cdots\times \Gamma_n$.
The only delicate point is that starting from a $\Gamma\acts X\soe\Lambda\acts Y$
and the corresponding $(\Gamma,\Lambda)$-coupling $\Omega$
one needs to look at the self $\Gamma$-coupling $\Sigma=\Omega\times_\Lambda\check\Omega$
and apply the cocycle superrigidity result to $\Gamma\acts \Sigma/\Gamma$.
In order to guarantee that the latter action is irreducible, Monod and Shalom
require $\Gamma\acts X$ to be irreducible and $\Lambda\acts Y$ to be \emph{mildly mixing}.
They also show that the assumption on mild mixing is necessary for the result. 

\medskip

In \cite{Bader+Furman:Weyl-hyplike} Uri Bader and the author proposed to
study higher rank superrigidity phenomena using a notion of a (generalized) Weyl group,
which works well for higher rank simple Lie groups, arbitrary products $G=G_1\times\cdots\times G_n$
of $n\ge 2$ factors, and exotic $\tilde{A}_2$ groups, which are close relatives to lattices in $\SL_3(\bbQ_p)$.
In particular:
\begin{theorem}[Bader - Furman \cite{Bader+Furman:Weyl-hyplike}]\label{T:Bader+Furman:CSR}
Theorem \ref{T:Monod-Shalom:CSR} holds for target groups from class $\Dynam$.
\end{theorem}
Here $\Dynam$ is a class of \emph{hyperbolic-like} groups which includes many of the 
examples in $\Creg$. Plugging this into Monod-Shalom machine one obtains the same
results of products of groups in class $\Dynam$.

 
\medskip

\subsection{Strong rigidity for cocycles} 
\label{sub:strong_rigidity_for_cocycles}\hfill\\

In the proof of Theorem~\ref{T:canonical-quotients} Zimmer's cocycle superrigidity
was applied to a Measure Equivalence cocycle. This is a rather special class of cocycles
(see \S \ref{sub:me_cocycles}). If cocycles are analogous to representations of lattices
then ME-cocycles are analogous to isomorphisms between lattices, in particular,
they have an "inverse". 
Kida's work on ME for Mapping Class Groups focuses on rigidity results 
for such cocycles. We shall not attempt to explain the ingredients used in this work,
but will just formulate the main technical result analogous to Theorem~\ref{T:canonical-quotients}. 
Let $\Gamma$ be a subgroup of finite index in $\Gamma(\Sigma_{g,p})^\diamond$
with $3g+p-4>0$, $C=C(\Sigma_{g,p})$ denote its curve complex, and $\Aut(C)$
the group of its automorphisms; this is a countable group commensurable to $\Gamma$.
\begin{theorem}[Kida \cite{Kida:2006ME}]\label{T:Kida-coc}
Let $(\Omega,m)$ be a self ME-coupling of $\Gamma$. Then there exists a measurable
map $\Gamma\times\Gamma$-equivariant map $\Phi:\Omega\to \Aut(C)$.
\end{theorem}	

\medskip

Returning to the point that ME-cocycles are analogous to isomorphism between lattices,
one might wonder whether Theorem~\ref{T:canonical-quotients} holds in cases where Mostow rigidity applies,
specifically for $G$ of rank one with ${\rm PSL}_2(\bbR)$ excluded.
In \cite{Bader+Furman+Sauer:2009} this is proved for $G\simeq \Isom(\hypsp^n_\bbR)$, $n\ge 3$, 
and a restricted ME.
\begin{theorem}[Bader - Furman - Sauer \cite{Bader+Furman+Sauer:2009}]\label{T:BFS-mostow}
Theorem~\ref{T:canonical-quotients} applies to $\ell^1$-ME-couplings of lattices in $G={\rm SO}_{n,1}(\bbR)$, $n\ge 3$.
\end{theorem}	
The proof of this result uses homological methods ($\ell^1$ and other completions of the usual homology)
combined with a version of Gromov-Thurston proof of Mostow rigidity (for $\Isom(\hypsp^n_\bbR)$, $n\ge 3$)
adapted to this setting.


\medskip

\subsection{Cocycle superrigid actions} 
\label{sub:deformation_vs_rigidity_technique}\hfill\\

In all the previous examples the structure of the acting group was the sole
source for (super)rigidity. 
Recently Sorin Popa has developed a number of remarkable cocycle superrigidity results 
of a completely different nature \cite{Popa:2006:Betti, Popa:2006p750, Popa:2006RigidityI, Popa:2006RigidityII, Popa:2007coc1,
Popa:2007ICM, Popa:2008spec-gap}.
These results exhibit an extreme form of cocycle superrigidity, and rather than relying only on 
the properties of the acting group $\Gamma$ take the advantage of the \emph{action} $\Gamma\acts (X,\mu)$.
\begin{defn}\label{D:GCSR}
An action $\Gamma\acts (X,\mu)$ is $\Grps$-\embf{cocycle superrigid}, where
$\Grps$ is some class of topological groups, if for every $\Lambda\in\Grps$
every measurable cocycle $c:\Gamma\acts X\to \Lambda$
has the form $c(g,x)=f(gx)^{-1}\rho(g)f(x)$ for some homomorphism
$\rho:\Gamma\to\Lambda$ and some measurable $f:X\to\Lambda$.  
\end{defn}
Here we shall focus on the class $\Gdsc$ of all countable groups;
however the following results hold for all cocycles taking values in a 
broader class $\Ufin$ which contains $\Gdsc$ and $\Gcpt$ -- separable compact groups.
Note that the concept of $\Gdsc$-cocycle superrigidity is unprecedentedly strong:
there is no assumption on the cocycle, the assumption on the target group
is extremely weak, the "untwisting" takes place in the same target group.
\begin{theorem}[Popa \cite{Popa:2007coc1}]\label{T:Popa-Bernoulli:CSR}
Let $\Gamma$ be a group with property (T), $\Gamma\acts (X,\mu)=(X_0,\mu_0)^\Gamma$
be the Bernoulli action. Then $\Gamma\acts (X,\mu)$ is $\Gdsc$-cocycle superrigid.
\end{theorem}
In fact, the result is stronger: it suffices to assume that $\Gamma$ has \embf{relative property (T)} with respect to 
a \embf{w-normal} subgroup $\Gamma_0$, and $\Gamma\acts (X,\mu)$ has a \embf{relatively weakly mixing} extension
$\Gamma\acts (\bar{X},\bar{\mu})$ which is \embf{s-malleable}, while $\Gamma_0\acts (\bar{X},\bar{\mu})$ 
is weakly mixing. Under these conditions $\Gamma\acts (X,\mu)$ is $\Ufin$-cocycle superrigid.
See \cite{Popa:2007coc1} and \cite{Furman:2007popa} for the relevant definitions and more details. 	
We indicate the proof (of the special case above) in \S \ref{sub:proofs}. 
Vaguely speaking Popa's approach exploits the tension between certain (local) \emph{rigidity}  
provided by the acting group and \emph{deformations} supplied by the action.
In the following remarkable result, Popa further relaxed the property (T) assumption. 
\begin{theorem}[Popa \cite{Popa:2008spec-gap}]\label{T:Popa-specgap:CSR}
Let $\Gamma$ be a group containing a product $\Gamma_1\times\Gamma_2$
where $\Gamma_1$ is non-amenable, $\Gamma_2$ is infinite, and $\Gamma_1\times\Gamma_2$
is w-normal in $\Gamma$. 
Then any Bernoulli action $\Gamma\acts (X,\mu)$ is $\Ufin$-cocycle superrigid.
\end{theorem}	
The deformations alluded above take place for the diagonal $\Gamma$-action
on the square $(X\times X,\mu\times\mu)$.
This action is supposed to be ergodic, equivalently the original action
should be weak \emph{mixing mixing} and satisfy addition properties.
\emph{Isometric actions} or, staying in the ergodic-theoretic terminology, 
actions with \emph{discrete spectrum}, provide the opposite type of dynamics.
These actions have the form $\Gamma\acts K/L$ where $L<K$ are compact groups, $\Gamma\to K$
a homomorphism with dense image, and $\Gamma$ acts by left translations.
Totally disconnected $K$ corresponds to profinite completion $\varprojlim\Gamma/\Gamma_n$
with respect to a chain of normal subgroups of finite index.
Isometric actions $\Gamma\acts K/L$ with profinite $K$, can be called
\embf{profinite ergodic} actions of $\Gamma$ -- these are precisely inverse
limits $X=\varprojlim X_n$ of transitive $\Gamma$-actions on finite spaces.
Adrian Ioana found the following "virtually $\Gdsc$-cocycle superrigidity"
phenomenon for profinite actions of Kazhdan groups.
\begin{theorem}[Ioana \cite{Ioana:2008profinite}]\label{T:Ioana-profinite:CSR}
Let $\Gamma\acts X=K/L$ be an ergodic profinite action. Assume that $\Gamma$
has property (T), or a relative property (T) with respect to a normal subgroup
$\Gamma_0$ which acts ergodically on $X$. Then any measurable cocycle 
$c:\Gamma\acts X\to\Lambda$ into a discrete group, is conjugate
to a cocycle coming from a finite quotient $X\to X_n$,
i.e., $c$ is conjugate to a cocycle induced from a homomorphism $\Gamma_n\to\Lambda$
of a finite index subgroup.
\end{theorem}	
In \S \ref{ssub:a_cocycle_superrigidity_theorem} a similar result is proven for all discrete 
spectrum actions (not necessarily profinite ones). 


\bigskip

\subsection{Constructing representations} 
\label{sub:constructing_representations}\hfill\\

In Geometric Group Theory many QI rigidity results are proved using the following trick.
Given a metric space $X$ one declares self-quasi-isometries $f, g:X\to X$
to be equivalent if \[\sup_{x\in X} d(f(x),g(x))<\infty.\]
Then equivalence classes of q.i. form a \emph{group}, denoted $\QI(X)$.
This group contains (a quotient of) $\Isom(X)$, which can sometimes be identified 
within $\QI(X)$ in coarse-geometric terms.
If $\Gamma$ is a group with well understood $\QI(\Gamma)$ 
and $\Lambda$ is an \emph{unknown} group q.i. to $\Gamma$,
then one gets a \emph{homomorphism}
\[
	\rho:\Lambda\to \Isom(\Lambda)\to\QI(\Lambda)\cong \QI(\Gamma)
\]
whose kernel and image can then be analyzed. 

Facing a similar problem in the Measure Equivalence category, there is a difficulty in defining
an analogue for $\QI(\Gamma)$. 
Let us describe a construction which allows to analyze the class of all groups 
ME to a given group $\Gamma$ from an information about self ME couplings of $\Gamma$.

Let $G$ be a lcsc unimodular group. Let us assume that $G$ has the \embf{strong ICC} property,
by which we mean that the only regular Borel conjugation invariant probability measure on $G$ 
is the trivial one, namely the Dirac mass $\delta_e$ at the origin.
For countable groups this is equivalent to the condition that all non-trivial conjugacy classes
are infinite, i.e., the usual ICC property.
Connected, (semi) simple Lie groups with trivial center and no compact factors provide other examples 
of strongly ICC groups.

Theorems~\ref{T:canonical-quotients}, \ref{T:Kida-coc}, \ref{T:BFS-mostow} are instances
where a strongly ICC group $G$ has the property that for any\footnote{ In the case of $G=\Isom({\bf H}^n)$
we restrict to all $\ell^1$-ME couplings.}  ME self coupling $(\Omega,m)$ of a lattice $\Gamma$ in $G$ there 
exists a bi-$\Gamma$-equivariant measurable map to $G$, i.e. a Borel map $\Phi:\Omega\to G$
satisfying $m$-a.e.
\[ 
	\Phi((\gamma_1,\gamma_2)\omega)=\gamma_1\Phi(\omega)\gamma_2^{-1}\qquad (\gamma_1,\gamma_2\in\Gamma).
\] 
It is not difficult to see that the strong ICC property implies that such a map is also unique.
(It should also be pointed out that the existence of such maps for self couplings of lattices is equivalent to 
the same property for self couplings of the lcsc group $G$ itself; but here we shall stay in the framework of
countable groups).
The following general tool shows how these properties of $G$ can be used to classify all groups ME to a lattice $\Gamma<G$; up to finite kernels these turn out to be lattices in $G$.
\begin{theorem}[Bader-Furman-Sauer \cite{Bader+Furman+Sauer:2009}]	\label{T:reps}
Let $G$ be a strongly ICC lcsc unimodular group, $\Gamma<G$ a lattice,
and $\Lambda$ some group ME to $\Gamma$ and $(\Omega,m)$ be a $(\Gamma,\Lambda)$-coupling.
Assume that the self ME-coupling $\Sigma=\Omega\times_\Lambda\check{\Omega}$ of $\Gamma$
admits a Borel map $\Phi:\Sigma\to G$, satisfying a.e.
\[
	\Phi([\gamma_1 x,\gamma_2 y])=\gamma_1\cdot \Phi([x,y])\cdot \gamma_2^{-1}
	\qquad (\gamma_1,\gamma_2\in\Gamma).
\]	
Then there exists a short exact sequence $K\overto{} \Lambda\overto{} \bar{\Lambda}$ 
with $K$ finite and $\bar{\Lambda}$ being a lattice in $G$, and
a Borel map $\Psi:\Omega\to G$ so that a.e.
\[
	\Phi([x,y])=\Psi(x)\cdot\Psi(y)^{-1},\quad
	\Psi(\gamma z)=\gamma\cdot\Psi(z),\quad
	\Psi(\lambda z)=\Psi(z)\cdot \bar\lambda^{-1}.
\]
Moreover, the pushforward of $\Psi_*m$ is a Radon measure on $G$ invariant under
the maps 
\[
	g\mapsto \gamma g\bar\lambda,\qquad (\gamma\in\Gamma,\ \bar\lambda\in\bar\Lambda).
\]
\end{theorem}
If $G$ is a (semi)-simple Lie group the last condition on the pushforward measure
can be analyzed using Ratner's theorem (as in Theorem~\ref{T:canonical-quotients})
to deduce that assuming ergodicity $\Psi_*m$ is either a Haar measure on $G$, 
or on a coset of its finite index subgroup,
or it is (proportional to) counting measure on a coset of a lattice $\Gamma'$
containing $\Gamma$ and a conjugate of $\bar\Lambda$ as finite index subgroups.

Theorem~\ref{T:reps} is a streamlined and improved version of
similar statements obtained in \cite{Furman:ME:99} for higher rank lattices, 
in \cite{Monod+Shalom:OE:05} for products, and in \cite{Kida:2006ME} for mapping class groups. 
%


\bigskip

\subsection{Local rigidity for measurable cocycles} 
\label{sub:local_rigidity}\hfill\\

The \emph{rigidity vs. deformations}  approach to rigidity results developed by Sorin Popa 
led to number of striking results in von Neumann algebras and in Ergodic theory (some been mentioned
in \S \ref{sub:deformation_vs_rigidity_technique}). 
Let us illustrate the \emph{rigidity} side of this approach by the following simple 
purely ergodic-theoretic statement, which is a variant of Hjorth's \cite[Lemma 2.5]{Hjorth:05Dye}.

\medskip

Recall that one of the several equivalent forms of property (T) is the following statement:
a lcsc group $G$ has (T) if there exist a compact $K\subset G$ and $\epsilon>0$
so that for any unitary $G$-representation $\pi$
and any $(K,\epsilon)$-almost invariant unit vector $v$
there exists a $G$-invariant unit vector $w$ with $\|v-w\|<1/4$.
%
\begin{prop}\label{P:local-rigidity}
Let $G$ be a group with property (T) and $(K,\epsilon)$ as above.
Then for any ergodic probability measure preserving action $G\acts (X,\mu)$,
any countable group $\Lambda$ and any pair of cocycles
$\alpha,\beta: G\times X\to\Lambda$ with
\[
	 \mu\setdef{x\in X}{\alpha(g,x)= \beta(g,x)}>1-\frac{\epsilon^2}{2}
	\qquad (\forall g\in K)
\]
there exists a measurable map $f:X\to \Lambda$ so that $\beta=\alpha^f$. 
Moreover, one can assume that 
\[
	\mu\setdef{x}{f(x)=e}>3/4.
\]
\end{prop}	
\begin{proof}
Let $\tilde{X}=X\times\Lambda$ be equipped with the infinite measure $\tilde{\mu}=\mu\times m_\Lambda$
where $m_\Lambda$ stands for the counting measure on $\Lambda$.
Then $G$ acts on $(\tilde{X},\tilde{\mu})$ by
\[
	g:(x,\lambda)\mapsto (g.x, \alpha(g,x)\lambda\beta(g,x)^{-1}).
\]
This action preserves $\tilde{\mu}$ and we denote by $\pi$ the corresponding
unitary $G$-representation on $L^2(\tilde{X},\tilde{\mu})$.
The characteristic function $v={\bf 1}_{X\times\{e\}}$ satisfies 
\[
	\|v-\pi(g)v\|^2=2-2{\rm Re}\langle\pi(g)v,v\rangle<2-2(1-\frac{\epsilon^2}{2})=\epsilon^2
	\qquad(g\in\Gamma)
\]
and therefore there exists a $\pi(G)$-invariant unit vector $w\in L^2(\tilde{X},\tilde{\Lambda})$
with $\|v-w\|<1/4$. Since $1=\|w\|^2=\int_X \sum_\lambda |w(x,\lambda)|^2$
we may define 
\[
	p(x)=\max_{\lambda}|w(x,\lambda)|,\qquad \Lambda(x)=\setdef{\lambda}{|w(x,\lambda)|=p(x)}
\]
and observe that $p(x)$ and the cardinality $k(x)$ of the finite set $\Lambda(x)$ are
measurable $\Gamma$-invariant functions on $(X,\mu)$; hence are a.e. constants $p(x)=p\in(0,1]$, $k(x)=k\in\{1,2,\dots\}$.
Since 
$1/16>\|v-w\|^2\ge (1-p)^2$ we have $p>3/4$. 
It follows that $k=1$ because $1=\|w\|^2\ge kp^2$.
Therefore  $\Lambda(x)=\{f(x)\}$ for some measurable map $f:X\to\Lambda$. 
The $\pi(G)$-invariance of $w$ gives $\pi(G)$-invariance of the characteristic 
function of $\setdef{(x,f(x))\in\tilde{X}}{x\in X}$, which is equivalent to
\begin{equation}\label{e:fgx-fx}
	f(gx)=\alpha(g,x)f(x)\beta(g,x)^{-1}\qquad \text{and}\qquad \beta=\alpha^f.
\end{equation}
Let $A=f^{-1}(\{e\})$ and $a=\mu(A)$. Since $\sum_{\lambda}|w(x,\lambda)|^2$ is a $G$-invariant
function it is a.e. constant $\|w\|^2=1$. Hence for $x\notin A$ we have
$|w(x,e)|^2\le 1-|w(x,f(x))|^2=1-p^2$, and 
\[
	\frac{1}{16}>\|v-w\|^2\ge a\cdot (1-p^2)+(1-a)\cdot (1-(1-p^2))\ge (1-a)\cdot p^2>\frac{9(1-a)}{16}.
\]
Thus $a=\mu\setdef{x\in X}{f(x)=e}>8/9>3/4$ as required. 
\end{proof}	

\bigskip

\subsection{Cohomology of cocycles} 
\label{ssub:cohomology_of_cocycles}\hfill\\

Let us fix two groups $\Gamma$ and $\Lambda$. There is no real assumption on $\Gamma$,
it may be any lcsc group, but we shall impose an assumption on $\Lambda$. 
One might focus on the case where $\Lambda$ is a countable group (class $\Gdsc$), 
but versions of the statements below would apply also to separable compact groups, or groups
in a larger class $\Ufin$ of all Polish groups which imbed in the unitary group of a
von-Neumann algebra with finite faithful trace\footnote{ This class, introduced by Popa
contains both discrete countable groups and separable compact ones.}, or
a potentially even larger class $\Ginv$ of groups with a bi-invariant metric,
and the class $\Galg$ of connected algebraic groups over local fields, say of zero characteristic. 

Given a (not necessarily free) p.m.p. action $\Gamma\acts (X,\mu)$ let $Z^1(X,\Lambda)$, 
or $Z^1(\Gamma\acts X,\Lambda)$, denote the space of all measurable cocycles $c:\Gamma\times X\to\Lambda$
and by $H^1(X,\Lambda)$, or $H^1(\Gamma\acts X,\Lambda)$, 
the space of equivalence classes of cocycles up to conjugation
by measurable maps $f:X\to\Lambda$. 
If $\Lambda\in\Galg$ we shall focus on a subset $H_{ss}^{1}(X,\Lambda)$
of (classes of) cocycles whose algebraic hull is connected, semi-simple, center free
and has no compact factors.
 
Any $\Gamma$-equivariant quotient map $\pi:X\to Y$ defines a pull-back
$Z^1(Y,\Lambda)\to Z^1(X,\Lambda)$ by $c^\pi(g,x)=c(g,\pi(x))$, which 
descends to 
\[
	H^1(Y,\Lambda)\overto{\pi^*} H^1(X,\Lambda).
\]
Group inclusions $i:\Lambda<\bar\Lambda$, and $j:\Gamma'<\Gamma$ give rise to push-forward maps
\[
	H^1(X,\Lambda)\overto{i_*} H^1(X,\bar\Lambda),\qquad H^1(\Gamma\acts X,\Lambda)\overto{j_*} H^1(\Gamma'\acts X,\Lambda).
\]

\begin{ques*} 
	What can be said about these maps of the cohomology ?
\end{ques*}

\medskip

The discussion here is inspired and informed by Popa's \cite{Popa:2007coc1}. In particular,
the following statements \ref{P:H1-inj}(2), \ref{P:H1-tightness}, \ref{P:H1-wnormal}(1), \ref{C:YY}
are variations on Popa's original \cite[Lemma 2.11, Proposition 3.5, Lemma 3.6, Theorem 3.1]{Popa:2007coc1}.
Working with class $\Ginv$ makes the proofs more transparent than in $\Ufin$ -- 
this was done in \cite[\S 3]{Furman:2007popa}.
Proposition \ref{P:H1-inj} for semi-simple target (3) is implicit in \cite[Lemma 3.5]{Furman:OE:99}.
The full treatment of the statements below, including Theorem~\ref{T:H1-pb-pu}, will appear in \cite{Furman:cohom}.

\medskip

\begin{prop}\label{P:H1-inj}
Let $\pi:X\to Y$ be a $\Gamma$-equivariant quotient map. Then 
\[
	H^1(Y,\Lambda)\overto{\pi^*} H^1(X,\Lambda) 
\]
is injective in the following cases:
\begin{enumerate}
	\item $\Lambda$ is discrete and torsion free.
	\item $\Lambda\in\Ginv$ and $\pi:X\to Y$ is relatively weakly mixing.
	\item $\Lambda\in\Galg$ and $H^1(-,\Lambda)$ is replaced by $H_{ss}^{1}(-,\Lambda)$.
\end{enumerate}
\end{prop}	
The notion of \embf{relative weakly mixing} was introduced independently by Zimmer \cite{Zimmer:1976:wm} 
and Furstenberg \cite{Furstenberg:1977:Sz}: a $\Gamma$-equivariant map $\pi:X\to Y$ is relatively weakly mixing
if the $\Gamma$-action on the fibered product $X\times_Y X$ is ergodic (or ergodic relatively to $Y$);
this turns out to be equivalent to the condition that $\Gamma\acts X$ contains no intermediate isometric 
extensions of $\Gamma\acts Y$.

\begin{prop}\label{P:H1-tightness}
	Let $i:\Lambda<\bar\Lambda\in\Ginv$ be a closed subgroup, and $\Gamma\acts(X,\mu)$ some p.m.p. action.
	Then 
	\[
		H^1(X,\Lambda)\overto{i_*} H^1(X,\bar\Lambda)
	\] 
	is injective.
\end{prop}
This useful property fails in $\Galg$ setting: if $\Gamma<G$ is a lattice in a (semi-) simple Lie group
and $c:\Gamma\times G/\Gamma\to\Gamma$ in the canonical class then viewed as cocycle into $G>\Gamma$,
$c$ is conjugate to the identity imbedding $\Gamma\cong \Gamma<G$, but as a $\Gamma$-valued
cocycle it cannot be "untwisted".

\begin{prop}\label{P:H1-wnormal}
	Let $\pi:X\to Y$ be a quotient map of ergodic actions, and $j:\Gamma'<\Gamma$ be a normal, or sub-normal, or w-normal 
	closed subgroup acting ergodically on $X$. Assume that either
	\begin{enumerate}
		\item $\Lambda\in\Ginv$ and $\pi$ is relatively weakly mixing, or
		\item $\Lambda\in\Galg$ and one considers $H^1_{ss}(-,\Lambda)$.
	\end{enumerate} 
	Then $H^1(\Gamma\acts Y,\Lambda)$ is the push-out of the rest of the following diagram:
	\[ 
		\xymatrix{ 
		  H^{1}(\Gamma\acts X,\Lambda) \ar[r]^{j_*} & H^1(\Gamma'\acts X,\Lambda)\\
	 	  H^{1}(\Gamma\acts Y,\Lambda) \ar[u]^{\pi^{*}} \ar[r]^{j_*}  & H^{1}(\Gamma'\acts Y,\Lambda)\ar[u]^{\pi^*}
		}
	\]
	In other words, if the restriction to $\Gamma'\acts X$ of a cocycle $c:\Gamma\times X\to\Lambda$ is conjugate to
	one descending to $\Gamma'\times Y\to\Lambda$, then $c$ has a conjugate
	that descends to $\Gamma\times X\to\Lambda$.
\end{prop}
The condition $\Gamma'<\Gamma$ is \embf{w-normal} (weakly normal) means that there exists a well ordered chain $\Gamma_i$
of subgroups starting from $\Gamma'$ and ending with $\Gamma$, so that $\Gamma_i\triangleleft\Gamma_{i+1}$ and 
for limit ordinals $\Gamma_j=\bigcup_{i<j} \Gamma_i$ (Popa).  

\medskip

Let $\pi_i:X\to Y_i$ is a collection of $\Gamma$-equivariant quotient maps.
Then $X$ has a unique $\Gamma$-equivariant quotient $p:X\to Z=\bigwedge Y_i$, which is maximal among
all common quotients $p_i:Y_i\to Z$. 
Identifying $\Gamma$-equivariant quotients with $\Gamma$-equivariant complete sub $\sigma$-algebras of $\mathcal{X}$,
one has $p^{-1}(\mathcal{Z})=\bigcap_i \pi_i^{-1}(\mathcal{Y}_i)$; or in the operator algebra 
formalism $p^{-1}(L^\infty(Z))=\bigcap_i \pi_i^{-1}(L^\infty(Y_i))$.
\begin{thm}\label{T:H1-pb-pu}
Let $\pi_i:X\to Y_i$, $1\le i\le n$, be a finite collection of $\Gamma$-equivariant quotients, and $Z=\bigwedge_{i=1}^n Y_i$. 
Then $H^1(Z,\Lambda)$ is the push-out of $H^1(Y_i,\Lambda)$ under conditions (1)-(3) of Proposition~\ref{P:H1-inj}:
\[ 
	\xymatrix{ 
	 & H^{1}(X,\Lambda) & \\
	H^{1}(Y_1,\Lambda) \ar[ur]^{\pi_1^{*}}   & \cdots H^{1}(Y_i,\Lambda) \ar[u]^{\pi_i^*} \cdots & \ar[ul]_{\pi_n^{*}} H^{1}(Y_n,\Lambda)\\ 
	 & H^{1}(Z,\Lambda) \ar[ur]^{p_1^{*}} \ar[u]^{p_i^*}\ar[ul]_{p_n^{*}} & 
	}
\]
More precisely, if $c_i:\Gamma\times X\to\Lambda$ are cocycles (in case (3) assume $[c_i]\in H^1_{ss}(Y_i,\Lambda)$), whose
pullbacks $c_i(g,\pi_i(x))$ are conjugate over $X$, then there exists a unique class $[c]\in H^1(Z,\Lambda)$,
so that $c(g,p_i(y))\sim c_i(g,y)$ in $Z^1(Y_i,\Lambda)$ for all $1\le i\le n$.
\end{thm}
The proof of this Theorem relies on Proposition~\ref{P:H1-inj} and contains it as a special case $n=1$.

This result can be useful to push cocycles to deeper and deeper quotients; if $\pi:X\to Y$ is 
\emph{a minimal} quotient to which a cocycle or a family of cocycles can descend up to conjugacy,
then it is \embf{the minimal} or \embf{characteristic} quotient for these cocycles: if they descend
to any quotient $X\to Y'$ then necessarily $X\to Y'\to Y$.
For example if $\Gamma<G$ is a higher rank lattice, $\Lambda$ a discrete group and $c:\Gamma\times X\to \Lambda$
is an OE (or ME) cocycle, then either $c$ descends to a $\Gamma$-action on a finite set (virtual isomorphism case),
or to $X\overto{\pi} G/\Lambda'$ with $\Lambda\simeq \Lambda'$ lattice in $G$, where $\pi$ is \emph{uniquely defined}
by $c$ (initial OE or ME).

An important special (and motivating) case of Theorem~\ref{T:H1-pb-pu} is that of $X=Y\times Y$ where $\Gamma\acts Y$ is a
weakly mixing action. Then the projections $\pi_i:X\to Y_i=Y$, $i=1,2$, give $Z=Y_1\wedge Y_2=\{pt\}$
and $H^1(\Gamma\acts \{pt\},\Lambda)={\rm Hom}(\Gamma,\Lambda)$. So
\begin{cor}[{Popa \cite[Theorem 3.1]{Popa:2007coc1}, see also \cite[Theorem 3.4]{Furman:2007popa}}]\label{C:YY}
Let $\Gamma\acts Y$ be a weakly mixing action and $c:\Gamma\times Y\to\Lambda$ a cocycle into $\Lambda\in\Ginv$.
Let $X=Y\times Y$ with the diagonal $\Gamma$-action, $c_1, c_2:\Gamma\times X\to\Lambda$ the cocycles
$c_i(g,(y_1,y_2))=c(g,y_i)$. If $c_1\sim c_2$ over $X$ then there exists homomorphism $\rho:\Gamma\to\Lambda$
and a measurable $f:Y\to\Lambda$, so that $c(g,y)=f(gy)^{-1}\rho(g)\,f(y)$.
\end{cor}

\bigskip

\subsection{Proofs of some results} 
\label{sub:proofs}\hfill\\

In this section we shall give a relatively self contained proofs 
of some of the results mentioned above. 

\subsubsection{Sketch of a proof for Popa's cocycle superrigidity theorem \ref{T:Popa-Bernoulli:CSR}}\hfill\\

First note that without loss of generality the base space $(X_0,\mu_0)$
of the Bernoulli action may be assumed to be non-atomic.
Indeed, Proposition~\ref{P:H1-inj}(2) implies that for each of
the classes $\Gdsc\subset \Ufin\subseteq \Ginv$ the corresponding 
cocycle superrigidity descends through relatively weakly mixing quotients,
and $([0,1],dx)^\Gamma\to (X_0,\mu_0)^\Gamma$ is such.

Given any action $\Gamma\acts (X,\mu)$ consider the diagonal $\Gamma$-action on $(X\times X,\mu\times\mu)$
and its centralizer $\Aut_\Gamma(X\times X)$ in the Polish group $\Aut(X\times X,\mu\times\mu)$.
It always contain the flip $F:(x,y)\mapsto (y,x)$. 
Bernoulli actions $\Gamma\acts X=[0,1]^\Gamma$ have the special property (called \embf{s-malleability} by Popa)
that there is a path 
\[
	p:[1,2]\to \Aut_\Gamma(X\times X),\qquad\text{with}\qquad p_1={\rm Id},\qquad p_2=F.
\]
Indeed, the diagonal component-wise action of $\Aut([0,1]\times[0,1])$ on $X\times X=([0,1]\times[0,1])^\Gamma$ 
embeds into $\Aut_\Gamma(X\times X)$ and can be used to connect ${\rm Id}$ to $F$.

Fix a cocycle $c:\Gamma\acts X\to \Lambda$. Consider the two lifts to $X\times X\to X$: 
\[
	c_i:\Gamma\acts X\times X\to\Lambda,\qquad c_i(g,(x_1,x_2))=c(g,x_i),\qquad (i=1,2).
\]
Observe that they are connected by the continuous path of cocycles $c_t(g,(x,y))=c_1(g,p_t(x,y))$,
$1\le t\le 2$.
Local rigidity \ref{P:local-rigidity} implies that $c_1$ and $c_2$ are conjugate over $X\times X$, and
the proof is completed invoking Corollary \ref{C:YY}.
Under the weaker assumption of relative property (T) with respect to a w-normal subgroup,
Popa uses Proposition~\ref{P:H1-wnormal}.
 
\subsubsection{A cocycle superrigidity theorem} 
\label{ssub:a_cocycle_superrigidity_theorem}

We state and prove a cocycle superrigidity theorem, inspired and generalizing 
Adrian Ioana's Theorem~\ref{T:Ioana-profinite:CSR}. 
Thus a number of statements (Theorems~\ref{T:Out1}, \ref{T:cons-of-csr}(2), \S \ref{ssub:FM_question}) 
in this survey get a relatively full treatment.
The proof is a good illustration of Popa's \emph{deformation vs. rigidity} approach.

Recall that an ergodic p.m.p. action $\Gamma\acts (X,\mu)$ is said to have a \embf{discrete spectrum}
if the Koopman $\Gamma$-representation on $L^2(X,\mu)$ is a Hilbert sum of finite dimensional subrepresentations.
Mackey proved (generalizing Halmos - von Neumann theorem for $\bbZ$, and using Peter-Weyl ideas) 
that discrete spectrum action
is measurably isomorphic to the isometric $\Gamma$-action on $(K/L,m_{K/L})$, $g:kL\mapsto \tau(g)kL$,
where $L<K$ are compact separable groups and $\tau:\Gamma\to K$ is a homomorphism
with dense image.
 
\begin{theorem}[after Ioana's Theorem \ref{T:Ioana-profinite:CSR}, \cite{Ioana:2008profinite}]
\label{T:GammaonK}
Let $\Gamma\acts (X,\mu)$ be an ergodic p.m.p. action with discrete spectrum.
Assume that $\Gamma$ has property (T), or contains a w-normal subgroup $\Gamma_0$
with property (T) acting ergodically on $(X,\mu)$. 
Let $\Lambda$ be an arbitrary torsion free discrete countable group and 
$c:\Gamma\times X\to \Lambda$ be a measurable cocycle.

Then there is a finite index subgroup $\Gamma_1<\Gamma$, a $\Gamma_1$-ergodic component $X_1\subset X$
(of measure $\mu(X_1)=[\Gamma:\Gamma_1]^{-1}$), a homomorphism $\rho:\Gamma_1\to \Lambda$ and a measurable map $\phi:X\to \Lambda$,
so that the conjugate cocycle $c^\phi$ restricted to $\Gamma_1\acts X_1\to\Lambda$,
is the homomorphism $\rho:\Gamma_1\to\Lambda$. The cocycle $c^\phi:\Gamma\times X\to \Lambda$
is induced from $\rho$.
\end{theorem}
The assumption that $\Lambda$ is torsion free is not essential; in general, one might need
to lift the action to a finite cover $\hat{X}_1\to X_1$ via a finite group which imbeds in $\Lambda$.
If $K$ is a connected Lie group, then $\Gamma_1=\Gamma$ and $X_1=X=K/L$.
The stated result is deduced from the case where $L$ is trivial, i.e. $X=K$, 
using Proposition~\ref{P:H1-inj}(1).
We shall make this simplification and assume $\Gamma$ has property (T)
(the modification for the more general case uses an appropriate version of Proposition~\ref{P:local-rigidity} and Proposition~\ref{P:H1-tightness}). 
An appropriate modification of the result handles compact groups
as possible target group $\Lambda$ for the cocycle.
\begin{proof}
The $K$-action by right translations: $t:x\mapsto xt^{-1}$, commutes with the $\Gamma$-action on $K$; in fact, $K$ is precisely
the centralizer of $\Gamma$ in $\Aut(K,m_K)$. 
This allows us to deform the initial cocycle $c:\Gamma\times X\to\Lambda$ be setting
\[
	c_t(g,x)=c(g,xt^{-1})\qquad (t\in K).
\]
Let $F\subset \Gamma$ and $\epsilon>0$ be as in the "local rigidity" Proposition~\ref{P:local-rigidity}.
Then for some open neighborhood
$U$ of $e\in K$ for every $t\in U$ there is a unique measurable $f_t:K\to \Lambda$ with
\[
	c_t(g,x)=c(g,xt^{-1})=f_t(gx)c(g,x)f_t(x)^{-1}\qquad \mu\setdef{x}{f_t(x)=e}>\frac{3}{4}.
\] 
Suppose $t,s\in U$ and $ts\in U$. Then
\begin{eqnarray*}
	f_{ts}(gx)\, c(g,x)\, f_{ts}(x)^{-1} &=&c_{ts}(g,x)=c(g,xs^{-1}t^{-1})\\
	&=&f_t(gxs^{-1})\,c(g,xs^{-1})\,f_t(xs^{-1})^{-1}\\
	&=&f_t(gxs^{-1})f_s(gx)\, c(g,x)\, [f_t(xs^{-1})f_s(x)^{-1}]^{-1}.
\end{eqnarray*}
This can be rewritten as 
\[
	F(gx)=c(g,x)\, F(x)\, c(g,x)^{-1},\qquad \text{where}\qquad F(x)=f_{ts}(x)^{-1}f_t(xs^{-1})f_s(x).
\]
Since $f_t, f_s, f_{ts}$ take value $e$ with probability $>3/4$, it follows that $A=F^{-1}(\{e\})$
has $\mu(A)>0$. The equation implies $\Gamma$-invariance of $A$. Thus $\mu(A)=1$ by ergodicity. 
Hence whenever $t,s,ts \in U$
\begin{equation}\label{e:f-coc}
	f_{ts}(x)=f_t(xs^{-1})f_s(x).
\end{equation}
If $K$ is a totally disconnected group, i.e., a profinite completion of $\Gamma$ as in Ioana's 
Theorem~\ref{T:Ioana-profinite:CSR}, then $U$ contains an open subgroup $K_1<K$. 
In this case one can skip the following paragraph.

In general, let $V$ be a symmetric neighborhood of $e\in K$ so that $V^2\subset U$,
and let $K_1=\bigcup_{n=1}^\infty V^n$. Then $K_1$ is an open (hence also closed) subgroup of $K$;
in the connected case $K_1=K$. 
We shall extend the family $\{f_t:K\to \Lambda\}_{t\in V}$ to be defined for all $t\in K_1$
while satisfying (\ref{e:f-coc}), using "cocycle continuation" procedure akin to analytic continuation.
For $t,t'\in K_1$ a $V$-\embf{quasi-path} $p_{t\to t'}$ from $t$ to $t'$ is a sequences $t=t_0,t_1,\dots,t_n=t'$
where $t_{i}\in t_{i-1}V$. 
Two $V$-quasi-paths from $t$ to $t'$ are $V$-\embf{close} if they are within $V$-neighborhoods
from each other. Two $V$-quasi-paths $p_{t\to t'}$ and $q_{t\to t'}$ are $V$-\embf{homotopic} if there is a chain
$p_{t\to t'}=p^{(0)}_{t\to t'},\dots,p^{(k)}_{t\to t'}=q_{t\to t'}$ of $V$-quasi-paths
where $p^{(i-1)}$ and $p^{(i)}$ are $V$-close, $1\le i\le k$. 
Iterating (\ref{e:f-coc}) one may continue the definition of $f_\cdot$ from $t$ to $t'$ along a $V$-quasi-path;
the continuation being the same for $V$-close quasi-paths, and therefore for $V$-homotopic 
quasi-paths as well (all from $t$ to $t'$). The possible ambiguity of this cocycle continuation procedure
is encoded in the \embf{homotopy group} $\pi^{(V)}_1(K_1)$ consisting of equivalence classes of
$V$-quasi-paths from $e\to e$ modulo $V$-homotopy. 
We claim that this group is finite. In the case of a connected Lie group $K_1$, $\pi^{(V)}_1(K_1)$ is a quotient of $\pi_1(K_1)$ 
which is finite since $K_1$, contaning a dense property (T) group, cannot have torus factors.
This covers the general case as well since $\pi^{(V)}_1(K_1)$ "feels" only finitely many 
factors when $K_1$ is written as an inverse limit of connected Lie groups and finite groups.
Considering the continuations of $f_\cdot$ along $V$-quasi-paths $e\to e$ we get a homomorphism
$\pi^{(V)}_1(K_1)\to\Lambda$ which must be trivial since $\Lambda$ was assumed to be torsion free.
Therefore we obtain a family of measurable maps $f_t:K_1\to\Lambda$ indexed by $t\in K_1$
and still satisfying (\ref{e:f-coc}).

\medskip

Let $\Gamma_1=\tau^{-1}(K_1)$. Then the index $[\Gamma:\Gamma_1]=[K:K_1]$ is finite.
We shall focus on the restriction $c_1$ of $c$ to $\Gamma_1\acts K_1$.
Note that (\ref{e:f-coc}) is a cocycle equation for the \emph{simply transitive action} of $K_1$ on itself.
It follows by a standard argument that it is a coboundary. 
Indeed, for a.e. $x_0\in K_1$ equation (\ref{e:f-coc}) holds
for a.e. $t,s\in K_1$. In particular, for a.e. $t, x\in K_1$, using $s=x^{-1}x_0$, one obtains 
$f_{tx^{-1}x_0}(x_0)=f_t(x) f_{x^{-1}x_0}(x_0)$. This gives
\[
	f_t(x)=\phi(xt^{-1})\phi(x)^{-1},\qquad \text{where}\qquad \phi(x)=f_{x^{-1}x_0}(x_0).
\] 
Equation $c_t=c^{f_t}$ translates into the fact that
the cocycle $c^\phi(g,x)=\phi(gx)^{-1}c(g,x)\phi(x)$ satisfies for a.e. $x, t$
\[
	c^\phi(g,xt^{-1})=c^\phi(g,x).
\]
Thus $c(g,x)$ does not depend on the space variable. Hence it is a homomorphism 
\[
	c^\phi(g,x)=\rho(g).
\]
Finally, the fact that $c^\phi$ is induced from $c_1^\phi$ is straightforward.
\end{proof}	




\appendix 
\newpage
\section{Cocycles} 
\label{sec:cocycles}

Let $G\acts (X,\mu)$ be a measurable, measure-preserving (sometimes just measure class preserving)
action of a topological group $G$ on a standard Lebesgue space $(X,\mu)$, and $H$ be a topological group.
A Borel measurable map $c:G\times X\to H$ forms a \embf{cocycle} if for every $g_1,g_2\in G$
for $\mu$-a.e. $x\in X$ one has
\[
	c(g_2g_1,x)=c(g_2,g_1.x)\cdot c(g_1,x)
\]
If $f:X\to H$ is a measurable map and $c:G\times X\to H$ is a measurable cocycle, define
the $f$-conjugate $c^f$ of $c$ to be
\[
	c^f(g,x)=f(g.x)^{-1}\,c(g,x)\,f(x).
\]
It is straightforward to see that $c^f$ is also a cocycle. One says that $c$ and $c^f$
are (measurably) \embf{conjugate}, or \embf{cohomologous} cocycles.
The space of all measurable cocycles $\Gamma\times X\to\Lambda$
is denoted $Z^1(\Gamma\acts X,\Lambda)$ and the space of equivalence classes
by $H^1(\Gamma\acts X,\Lambda)$.

Cocycles which do not depend on the space variable:
$c(g,x)=c(g)$ are precisely homomorphisms $c:G\to H$.
So cocycles may be viewed as generalized homomorphisms.
In fact, any group action $G\acts (X,\mu)$ defines a \embf{measured groupoid}
$\mathcal{G}$ with $\mathcal{G}^{(0)}=X$, and 
$\mathcal{G}^{(1)}=\setdef{(x,gx)}{x\in X,\ g\in G}$
(see \cite{Anantharaman-Delaroche+Renault:2000} for the background). 
In this context cocycles can be viewed as homomorphisms
$\mathcal{G}\to H$.

If $\pi:(X,\mu)\to (Y,\nu)$ is an equivariant quotient map between $\Gamma$-actions 
(so $\pi_*\mu=\nu$, and $\pi\circ\gamma=\gamma$ for $\gamma\in\Gamma$) then for any
target group $\Lambda$ any cocycle $c:\Gamma\times Y\to\Lambda$ lifts to
$\bar{c}:\Gamma\times X\to\Lambda$ by
\[
	\bar{c}(g,x)=c(g,\pi(x)).
\] 
Moreover, if $c'=c^f\sim c$ in $Z^1(\Gamma\acts Y,\Lambda)$ then
the lifts $\bar{c}'=\bar{c}^{f\circ\pi}\sim \bar{c}$ in $Z^1(\Gamma\acts X,\Lambda)$;
so $X\overto{\pi} Y$ induces
\[
	H^1(\Gamma\acts X,\Lambda)\overfrom{\pi^\diamond} H^1(\Gamma\acts Y,\Lambda)
\]
Note that ${\rm Hom}(\Gamma,\Lambda)$ is $Z^1(\Gamma\acts\{pt\},\Lambda)$
and classes of cocycles on $\Gamma\times X\to\Lambda$ cohomologous to homomorphisms
is precisely the pull back of $H^1(\Gamma\acts\{pt\},\Lambda)$.

\medskip

\subsection{The canonical class of a lattice, (co-)induction} 
\label{sub:canonical_class_of_a_lattice_induction}\hfill\\

Let $\Gamma<G$ be a lattice in a lcsc group. 
By definition the transitive $G$-action
on $X=G/\Gamma$ has an invariant Borel regular probability measure $\mu$.
Let $\mathcal{F}\subset G$ be a Borel \embf{fundamental domain} for the 
right $\Gamma$-action on $G$
(i.e. $\mathcal{F}$ is a Borel subset of $G$ set which meats every coset $g\Gamma$ precisely
once).
Fundamental domains correspond to Borel cross-section 
$\sigma:G/\Gamma\to G$ of the projection $G\to G/\Gamma$.
Define:
\[
	c_{\sigma}:G\times G/\Gamma\to\Gamma,\qquad \text{by}\qquad
	c_{\sigma}(g,h\Gamma)=\sigma(gh\Gamma)^{-1}\,g\,\sigma(h\Gamma). 
\]
Clearly, this is a cocycle (a conjugate of the identity homomorphism $G\to G$); 
however $c_\sigma$ takes values in the subgroup $\Gamma$ of $G$. 
This cocycle is associated to a choice of the cross-section $\sigma$
(equivalently, the choice of the fundamental domain); starting
from another Borel cross-section $\sigma':G/\Gamma\to G$ results in 
a cohomologous cocycle:
\[
	c_{\sigma'}=c_\sigma^f\qquad \text{where}\qquad f:G/\Gamma\to\Gamma
	\qquad\text{is defined by}\qquad \sigma(x)=f(x)\sigma'(x).
\]
Let $\Gamma$ be a lattice in $G$.
%
%
Then any action $\Gamma\acts (X,\mu)$ 
gives rise to the \embf{induced} $G$-action (a.k.a. \embf{suspension}) on 
$\bar{X}=G\times_\Gamma X$ where $G$-acts on the first coordinate. 
Equivalently, $\bar{X}=G/\Gamma\times X$ and $g:(g'\Gamma,x)\mapsto (gg'\Gamma, c(g,g'\Gamma)x)$ 
where $c:G\times G/\Gamma\to \Gamma$ is in the canonical class.
Here the $G$-invariant finite measure $\bar\mu=m_{G/\Gamma}\times\mu$ is ergodic
iff $\mu$ is $\Gamma$-ergodic.
If $\alpha:\Gamma\times X\to H$ is a cocycle, the \embf{induced cocycle} $\bar{\alpha}:G\times\bar{X}\to H$
is given by $\bar{\alpha}(g,(g'\Gamma,x))=\alpha(c(g,g'\Gamma),x)$.
The cohomology of $\bar{\alpha}$ is the same as that of $\alpha$ 
(one relates maps $F:\bar{X}\to H$ to $f:X\to H$ by $f(x)=F(e\Gamma,x)$
taking instead of $e\Gamma$ a generic point in $G/\Gamma$).
In particular, $\bar\alpha$ is cohomologous to a homomorpism $\bar\pi:G\to H$
iff $\alpha$ is cohomologous to a homomorphism $\Gamma\to H$;
see \cite{Zimmer:book:84} for details.

Cocycles appear quite naturally in a number of situations such as (volume preserving)
smooth actions on manifolds, where choosing a measurable trivialization of the tangent bundle,
the derivative becomes a matrix valued cocycle.
We refer the reader to David Fisher's survey \cite{Fisher:survey} where this type of cocycles is 
extensively discussed in the context of Zimmer's programme.
Here we shall be interested in a different type of cocycles: "rearrangement" cocycles associated 
to Orbit Equivalence, Measure Equivalence etc. as follows.

\medskip

\subsection{OE-cocycles} 
\label{sub:oe_cocycles}\hfill\\

Let $\Gamma\acts (X,\mu)$ and $\Lambda\acts (Y,\nu)$ be two measurable, measure preserving,
ergodic actions on probability spaces, and $T:(X,\mu)\to (Y,\nu)$ be an Orbit Equivalence. 
Assume that the $\Lambda$-action is \emph{essentially free}, i.e., for $\nu$-.a.e $y\in Y$,
the stabilizer $\Lambda_y=\setdef{h\in\Lambda}{h.y=y}$ is trivial. 
Then for every $g\in\Gamma$ and $\mu$-a.e. $x\in X$,
the points $T(g.x), T(x)\in Y$ lie on the same $\Lambda$-orbit.
Let $\alpha(g,x)\in\Lambda$ denote the (a.e. unique) element of $\Lambda$ with
\[
	T(g.x)=\alpha(g,x).T(x)
\]
Considering $x,g.x, g'g.x$ one checks that $\alpha$ is actually
a cocycle $\alpha:\Gamma\times X\to\Lambda$. 
We shall refer to such $\alpha$ as the \embf{OE-cocycle}, or the \embf{rearrangement} cocycle,
corresponding to $T$.  

Note that for $\mu$-a.e. $x$, the map $\alpha(-,x):\Gamma\to \Lambda$
is a bijection; it describes how the $\Gamma$-names of points
$x'\in\Gamma.x$ translate into the $\Lambda$-names of $y'\in \Lambda.T(x)$
under the map $T$. 
The inverse map $T^{-1}:(Y,\nu)\to (X,\mu)$ defines an OE-cocycle $\beta:\Lambda\times Y\to \Gamma$
which serves as an "inverse" to $\alpha$ in the sense that a.e.
\[
	\beta(\alpha(g,x),T(X))=g\qquad (g\in\Gamma).
\]

\medskip

\subsection{ME-cocycles} 
\label{sub:me_cocycles}\hfill\\

Let $(\Omega,m)$ be an ME-coupling of two groups $\Gamma$ and $\Lambda$ and let $Y,X\subset \Omega$
be fundamental domains for $\Gamma$, $\Lambda$ actions respectively. 
The natural identification $\Omega/\Lambda\cong X$, $\Lambda\omega\mapsto \Lambda\omega\cap X$,
translates the $\Gamma$-action on $\Omega/\Lambda$ to $\Gamma\acts X$ by
\[
	\gamma:X\ni x \mapsto  g \alpha(g,x) x \in X
\] 
where $\alpha(\gamma,x)$ is the unique element in $\Lambda$ taking $\gamma x \in\Omega$
into $X\subset\Omega$. It is easy to see that $\alpha:\Gamma\times X\to\Lambda$
is a cocycle with respect to the above $\Gamma$-action on $X$ which we denote by a dot $\gamma\cdot x$
to distinguish it from the $\Gamma$-action on $\Omega$.
(If $\Gamma$ and $\Lambda$ are lattices in $G$ then $\alpha:\Gamma\times G/\Lambda\to\Lambda$
is the restriction of the canonical cocycle $G\times G/\Lambda\to\Lambda$).
Similarly we get a cocycle $\beta:\Lambda\times Y\to\Gamma$.  
So the $(\Gamma,\Lambda)$ ME-coupling $\Omega$ and a choice of fundamental domains 
$Y\cong \Omega/\Gamma$, $X\cong \Omega/\Lambda$ define a pair of cocycles
\begin{equation}\label{e:ME-cocycles}
	\alpha:\Gamma\times \Omega/\Lambda\to \Lambda,\qquad
	\beta:\Lambda\times \Omega/\Gamma\to\Gamma
\end{equation}
Changing the fundamental domains amounts to conjugating the cocycles and vise versa.

\begin{remark}
One can characterize ME-cocycles among all measurable cocycles $\alpha:\Gamma\times X\to\Lambda$
as \embf{discrete} ones with \embf{finite covolume}. 
These concepts refer to the following construction: 
let $(\tilde{X},\tilde{\mu})=(X\times\Lambda,\mu\times m_{\Lambda})$ and let
$\Gamma$-act by $g:(x,h)\mapsto (g.x, \alpha(g,x)h)$.
Say that the cocycle is \emph{discrete and has finite covolume} if the action
$\Gamma\acts (\tilde{X},\tilde{\mu})$ admits a finite measure fundamental domain.
\end{remark}




\end{document}